\documentclass[aos,preprint]{imsart}
\usepackage{amsthm,amsmath}
\usepackage[colorlinks,citecolor=blue,urlcolor=blue]{hyperref}
\usepackage{pdfpages} 
\usepackage{mathtools}
\DeclarePairedDelimiter\ceil{\lceil}{\rceil}
\DeclarePairedDelimiter\floor{\lfloor}{\rfloor}
\usepackage[margin=1in]{geometry}
\usepackage{bm}
\usepackage{graphicx}
\usepackage{amssymb,amsmath,amsthm,mathrsfs}
\usepackage{epstopdf}
\usepackage{amsfonts,delarray}
\usepackage{algorithm,algcompatible,amsmath}
\usepackage{rotating,color}
\usepackage{multirow}
\usepackage{titlesec,textcase}
\usepackage{longtable}
\usepackage{bbm}
\usepackage{rotating}
\usepackage{subcaption}

\usepackage{tikz}
\usetikzlibrary{topaths,calc}
\usetikzlibrary{shapes.geometric}

\usepackage{amssymb,amsmath,amsthm,mathrsfs}
\usepackage{enumitem, hyperref}
\newtheorem{Theorem}{Theorem}[section]
\newtheorem{Lemma}[Theorem]{Lemma}
\newtheorem{Corollary}[Theorem]{Corollary}
\newtheorem{Proposition}[Theorem]{Proposition}

\newtheorem{Remark}{Remark}[section]
\theoremstyle{definition}
\theoremstyle{definition}
\newtheorem{definition}{Definition}[section]

  \usepackage{adjustbox}
  
  \def\boxit#1{\vbox{\hrule\hbox{\vrule\kern6pt\vbox{\kern6pt#1\kern6pt}\kern6pt\vrule}\hrule}}

\begin{document}
 
\begin{frontmatter}
\maketitle

\title{Testing Community Structure for Hypergraphs}
	\runtitle{Testing Hypergraph Models}
	
	\begin{aug}
		\author{\fnms{Mingao } \snm{Yuan}\thanksref{m0}\ead[label=e4]{mingao.yuan@ndsu.edu}},
		\author{\fnms{Ruiqi} \snm{Liu}\thanksref{m1}\ead[label=e1]{ruiqliu@ttu.edu}},
		\author{\fnms{Yang} \snm{Feng}\thanksref{m2,t1}\ead[label=e2]{yang.feng@columbia.edu}}
		\and
		\author{\fnms{Zuofeng} \snm{Shang}\thanksref{m3,t2}
			\ead[label=e3]{zshang@njit.edu}}

		\thankstext{t1}{Corresponding author. Supported by NSF CAREER Grant DMS-2013789.}
		
		\thankstext{t2}{Corresponding author. Supported by NSF DMS-1764280 and NSF DMS-1821157.}
		\runauthor{Yuan, Liu, Feng and Shang}

		\affiliation{North Dakota State University\thanksmark{m0}, 
		Texas Tech University\thanksmark{m1}, New York University\thanksmark{m2} and New Jersey Institute of Technology \thanksmark{m3}}

		\address{Department of Statistics,\\
		North Dakota State University,\\
		1230 Albrecht Blvd, \\
		Fargo, ND 58108\\
		\printead{e4}\\
		\phantom{E-mail:\ }}
		
		\address{Department of Mathematical Sciences,\\
			Texas Tech University,\\
			402 N Blackford St,\\
			Indianapolis, IN 46202\\
			\printead{e1}\\
			\phantom{E-mail:\ }}
		
		\address{Department of Biostatistics, \\
			New York University,\\
			715 Broadway,\\
			New York, NY 10012\\
			\printead{e2}\\}
		\address{Department of Mathematical Sciences,\\
			IUPUI,\\
			402 N Blackford St,\\
			Indianapolis, IN 46202\\
			\printead{e3}\\
			\phantom{E-mail:\ }}
	\end{aug}
	
\begin{abstract}
Many complex networks in the real world can be formulated as
hypergraphs where community detection has been widely used. However,
the fundamental question of whether communities exist or not in an
observed hypergraph remains unclear. This work aims to tackle
this important problem. Specifically, we systematically study when a
hypergraph with community structure can be successfully distinguished
from its Erd\"{o}s-R\'{e}nyi  counterpart, and propose concrete test statistics
when the models are distinguishable. 
The main contribution of this paper is threefold.
First, we discover a phase transition in the hyperedge probability for distinguishability.
Second, in the bounded-degree regime, we derive a sharp signal-to-noise ratio (SNR) threshold for distinguishability 
in the special two-community 3-uniform hypergraphs, and
derive nearly tight SNR thresholds in the general two-community $m$-uniform hypergraphs. 
Third, in the dense regime, we propose a computationally feasible test based on sub-hypergraph counts, 
obtain its asymptotic distribution, and analyze its power. Our results
are further extended to non-uniform hypergraphs in which a new test
involving both edge and hyperedge information is proposed. The proofs rely on Janson's contiguity theory \cite{J95},
a high-moments driven asymptotic normality result by Gao and Wormald \cite{GWALD}, and a truncation technique for analyzing the likelihood ratio. 
\end{abstract}
	
	
\end{frontmatter}

\section{Introduction.}

Community detection is a fundamental problem in network data analysis. 
For instance, in social networks \cite{Fort,GZFA,ZLZ11}, protein to protein interactions  \cite{chenyuan}, image segmentation \cite{SM97},
among others, many algorithms have been developed for identifying community structure.
Theoretical studies on community detection have mostly been focusing on ordinary graph setting in which each possible edge contains exactly two vertices
(see \cite{Bolla,Agarwal,Rodrigues,ZLZ11, ZLZ12,GMZZ,ACBL13}).
One common assumption made in these references is the existence of communities.
Recently, a number of researchers have been devoted to testing this assumption,
e.g., \cite{BS16,L16,MS16,BM17,B18,GL17a,GL17b,YFS18a}. Besides,  hypothesis testing has been used to test the number of communities in a network \cite{BS16,L16}.

Real-world networks are usually more complex than ordinary graphs.
Unlike ordinary graphs where the data structure is typically unique, e.g., 
edges only contain two vertices, \emph{hypergraphs} demonstrate a number of possibly overlapping data structures. 
For instance, in coauthorship networks \cite{ER05,OGM17, RDP04,new01}, the number of coauthors varies across different papers so that 
one cannot consider edges consisting of two coauthors only. 
Instead, a new type of ``edge,'' called \emph{hyperedge}, must be considered 
which allows the connectivity of arbitrarily many coauthors. 
The complex structures of hypergraphs create new challenges in both theoretical and methodological studies.
As far as we know, existing hypergraph literature mostly focuses on community detection in algorithmic aspects \cite{Bulo,cherke,Bolla,Rodrigues,Agarwal,GD14,KBG17,LCW17}.
Only recently, Ghoshdastidar and Dukkipati \cite{GD14,GhDu} provided a statistical study in which a spectral algorithm based on
adjacency tensor was proposed for identifying community structure
and asymptotic results were developed. Nonetheless, the important problem of testing the existence of community structure in an
observed hypergraph remains untreated.

In this paper, we aim to tackle the problem of testing community structure for hypergraphs.
We first consider the relatively simpler but widely used 
uniform hypergraphs in which each hyperedge consists of an equal number of vertices. 
For instance, the (user, resource, annotation) structure in folksonomy may be represented as a uniform hypergraph where each hyperedge 
consists of three vertices \cite{GZCN}; the (user, remote host, login time, logout time) structure in the login-data can be modeled as a
uniform hypergraph where each hyperedge contains four vertices \cite{GKR}; 
the point-set matching problem is usually formulated as identifying a strongly connected component in a uniform hypergraph
\cite{cherke}. 
We provide various theoretical or methodological studies ranging from dense uniform hypergraphs to sparse ones
and investigate the possibility of a successful test in each scenario.
Our testing results in the dense case are then extended to the more general
non-uniform hypergraph setting, in which a new test statistic involving both edge and hyperedge is proposed. 
One important finding is that our new test is more powerful than the classic one involving edge information only,
showing the advantage of using hyperedge information to boost the testing performance.
A more notable contribution is a nearly tight threshold for signal-to-noise ratio to examine the existence of community structure (Theorem \ref{sharp:SNR:phase:transition}).

\subsection{Review of Hypergraph Model And Relevant Literature.}
In this section, we review some basic notions in hypergraphs and recent progress in the literature.
Let us first review the notion of the uniform hypergraph.
An \textit{$m$-uniform hypergraph} $\mathcal{H}_m=(\mathcal{V}, \mathcal{E})$ consists of a vertex set $\mathcal{V}$ and a hyperedge set $\mathcal{E}$, where each hyperedge in $\mathcal{E}$ is a
subset of $\mathcal{V}$ consisting of exactly $m$ vertices. Two hyperedges are the same if they are equal as vertex sets. 
An \textit{$l$-cycle} in $\mathcal{H}_m$
is a cyclic ordering $\{v_1,v_2,\ldots,v_{r}\}$ 
of a subset of the vertex set with hyperedges like $\{v_i,v_{i+1},\ldots,v_{i+m-1}\}$ and any two adjacent hyperedges have exactly $l$ common
vertices. 
An $l$-cycle is \textit{loose} if $l=1$ 
and \textit{tight} if $l=m-1$. To better
illustrate the notion, consider a $3$-uniform hypergraph $\mathcal{H}_3=(\mathcal{V},\mathcal{E})$, 
where $\mathcal{V}=\{v_1,v_2,v_3,v_4,v_5,v_6,v_7\}$, $\mathcal{E}=\{(v_i,v_j,v_t)|1\leq i<j<t\leq 7\}$. Then $(\{v_1,v_2,v_3,v_4,v_5,v_6\},\{(v_1,v_2,v_3), (v_3, v_4, v_5), (v_5,v_6,v_1)\})$ is a  \textit{loose} cycle and $(\{v_1,v_2,v_3,v_4\},\{(v_1,v_2,v_3),(v_2,v_3,v_4), (v_3,v_4,v_1), (v_4,v_1,v_2)\})$ is a 
\textit{tight} cycle (see Figure \ref{hypercycle}).

\begin{figure}[ht]
\centering

\begin{subfigure}{0.35\textwidth}
 \begin{tikzpicture}[
    dot/.style={draw, circle, inner sep=1pt, fill=black},
    group/.style={draw=#1, ellipse, ultra thick, minimum width=5cm, minimum height=1cm}
    ]

   \node[dot,label={[xshift=-0.3cm, yshift=-1cm]$v_3$}] at (0,0) (v3) {};   
   \node[dot,label=left:$v_4$] at (2,0) (v4) {};
   \node[dot,label={[xshift=0.3cm, yshift=-1cm]$v_5$}] at (4,0) (v5) {};
   \node[dot,label=below:$v_2$] at (1,2) (v2) {};
   \node[dot,label=below:$v_6$] at (3,2) (v6) {};
   \node[dot,label={[xshift=0cm, yshift=0.3cm]$v_1$}] at (2,4) (v1) {};
    
    \node[group=cyan, label={[cyan]below:$E_2$}] at (v4) (E2) {};
    \node[group=olive, label={[olive]above right:$E_1$}, rotate=63] at (v2) (E1) {};
    \node[group=red, label={[red]above left:$E_3$}, rotate=-63] at (v6) (E3) {};
\end{tikzpicture}
\end{subfigure}
 \begin{subfigure}{0.35\textwidth}

 \begin{tikzpicture}
    \node (v1) at (0,2) {};
    \node (v2) at (2,0) {};
    \node (v3) at (0,-2) {};
    \node (v4) at (-2,0) {};

    \begin{scope}[fill opacity=0.8]    
        \filldraw[fill=red]($(v1)+(0,0.7)$)
        to[out=0,in=90]($(v2)+(0.7,0)$)
        to[out=270,in=0]($(v3)+(0,-0.7)$)
        to[out=180,in=180]($(v3)+(0,0.7)$)
        to[out=0,in=270]($(v2)+(-0.7,0)$)
        to[out=90,in=0]($(v1)+(0,-0.7)$)
        to[out=180,in=180]($(v1)+(0,0.7)$);
        
        \filldraw[fill=blue]($(v2)+(-0.5,0)$)
        to[out=90,in=90]($(v2)+(0.5,0)$)
        to[out=270,in=0]($(v3)+(0,-0.5)$)
        to[out=180,in=270]($(v4)+(-0.5,0)$)
        to[out=90,in=90]($(v4)+(0.55,0)$)
        to[out=270,in=180]($(v3)+(0,0.5)$)
        to[out=0,in=270]($(v2)+(-0.5,0)$);
        
        \filldraw[fill=orange]($(v3)+(0,0.35)$)
        to[out=0,in=0]($(v3)+(0,-0.6)$)
        to[out=180,in=270]($(v4)+(-0.6,0)$)
        to[out=90,in=180]($(v1)+(0,0.6)$)
        to[out=0,in=0]($(v1)+(0,-0.35)$)
        to[out=180,in=90]($(v4)+(0.35,0)$)
        to[out=270,in=180]($(v3)+(0,0.35)$);
        
        \filldraw[fill=green]($(v4)+(0.2,0)$)
        to[out=270,in=270]($(v4)+(-0.45,0)$)
        to[out=90,in=180]($(v1)+(0,0.45)$)
        to[out=0,in=90]($(v2)+(0.45,0)$)
        to[out=270,in=270]($(v2)+(-0.2,0)$)
        to[out=90,in=0]($(v1)+(0,-0.2)$)
        to[out=180,in=90]($(v4)+(0.2,0)$);
    
    \end{scope}
        
        \foreach \v in {1,2,3,4} {
        \fill (v\v) circle (0.1);
    }

    \fill (v1) circle (0.1) node [left] {$v_1$};
    \fill (v2) circle (0.1) node [above] {$v_2$};
    \fill (v3) circle (0.1) node [left] {$v_3$};
    \fill (v4) circle (0.1) node [above] {$v_4$};     

    \node[ label={[red]right:$E_1$}]at (2.5,0){};
    \node[ label={[blue]below:$E_2$}]at (0,-2.5){};
    \node[ label={[orange]left:$E_3$}] at (-2.5,0){};
    \node[ label={[green]above:$E_4$}]at (0,2.5) {};    
   
\end{tikzpicture}
 \end{subfigure}

  \caption{\it\footnotesize Left: a \textit{loose} cycle of three edges $E_1,E_2,E_3$. Right: a \textit{tight} cycle of four edges
  $E_1,E_2,E_3,E_4$. 
  Both cycles are subgraphs of the $3$-uniform hypergraph $\mathcal{H}_3(\mathcal{V},\mathcal{E})$.}\label{hypercycle}
 \end{figure}
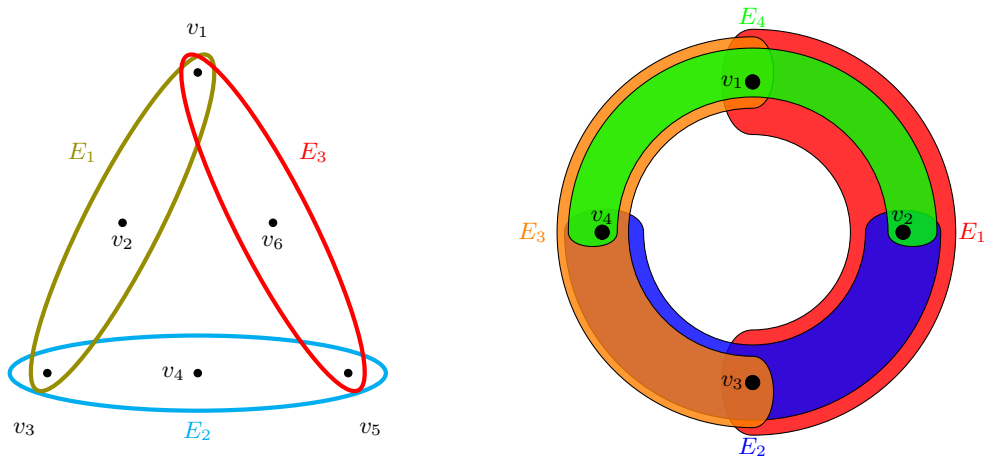

Next, let us review uniform hypergraphs with a planted partitioning structure, also known as stochastic block model (SBM).
For any positive integers $n,m,k$ with $m,k\ge2$, and positive sequences $0<q_n<p_n<1$ (possibly depending  on $n$),
let $\mathcal{H}_m^k(n,p_n,q_n)$ denote 
a $m$-uniform hypergraph of $n$ vertices and $k$ balanced communities,
in which $p_n$ ($q_n$) represents the hyperedge probability
within (between) communities. More explicitly,
any vertex $i\in[n]\equiv\{1,2,\dots,n\}$ is assigned, independently and uniformly at random, 
a label $\sigma_i\in[k]\equiv\{1,2,\dots,k\}$, and then each possible hyperedge
$(i_1,i_2,\dots, i_m)$ is included with probability $p_n$ if $\sigma_{i_1}=\sigma_{i_2}=\dots=\sigma_{i_m}$
and with probability $q_n$ otherwise. 
In particular, $\mathcal{H}_2^2(n,p_n,q_n)$ (with $m=k=2$)
reduces to the ordinary bisection stochastic block models considered by \cite{MNS15,YFS18a}.  
Let $A\in\{0,1\}^{\tiny \underbrace{n\times n\times\dots\times n}_m}$ denote the symmetric
adjacency tensor of order $m$ associated with $\mathcal{H}_m^k(n,p_n,q_n)$. 
By symmetry we mean that $A_{i_1i_2\dots i_m}=A_{\psi(i_1)\psi(i_2)\dots\psi(i_m)}$ for any permutation $\psi$ of $(i_1,i_2,\dots,i_m)$. 
For convenience, assume $A_{i_1i_2\dots i_m}=0$ if $i_s=i_t$ for some distinct $s, t\in\{1,2,\dots,m\}$,
i.e., the hypergraph has no self-loops. 
Conditional on $\sigma_1,\ldots,\sigma_n$, 
the $A_{i_1i_2\dots i_m}$'s, with $i_1,\ldots,i_m$ pairwise distinct, 
are assumed to be independent following the distribution below:
\begin{equation}\label{sbm:model}
\mathbb{P}(A_{i_1i_2\dots i_m}=1|\sigma)=p_{i_1i_2\dots i_m}(\sigma),\ \mathbb{P}(A_{i_1i_2\dots i_m}=0|\sigma)=q_{i_1i_2\dots i_m}(\sigma),
\end{equation}
where $\sigma=(\sigma_1,\ldots,\sigma_n)$,
\[
p_{i_1i_2\dots i_m}(\sigma)=\left\{\begin{array}{cc}
p_n,&\sigma_{i_1}=\dots=\sigma_{i_m}\\
q_n,&\textrm{otherwise}
\end{array}\right.,\ \
q_{i_1i_2\dots i_m}(\sigma)=1-p_{i_1i_2\dots i_m}(\sigma).
\]
In other words, each possible hyperedge
$(i_1,\ldots,i_m)$ is included with probability $p_n$
if the vertices $i_1,\ldots,i_m$ belong to the same community, and with probability $q_n$ otherwise. 
Let $\mathcal{H}_m(n,\frac{p_n+(k^{m-1}-1)q_n}{k^{m-1}})$ denote the
$m$-uniform hypergraph without community structure,
i.e., an Erd\"{o}s-R\'{e}nyi model 
in which each possible hyperedge is included with common probability $\frac{p_n+(k^{m-1}-1)q_n}{k^{m-1}}$.
We consider such a special choice of hyperedge probability in order to make the model have the same average degree as
$\mathcal{H}_m^k(n,p_n,q_n)$.
In particular, $\mathcal{H}_2(n,\frac{p_n+(k-1)q_n}{k})$ with $m=2$ becomes the traditional 
Erd\"{o}s-R\'{e}nyi model that has been well studied
in ordinary graph literature; see \cite{B01,BE76, FM15,erdos60,W99}.
Non-uniform hypergraphs can be simply viewed as a superposition of uniform ones;
see Section \ref{sec:non:uniform}. Throughout this paper, we assume $k$ and $m$ are fixed constants independent of $n$.

Given an observed adjacency tensor $A$, \textit{does $A$ represent a hypergraph that exhibits community structure?}
In the present setting, this problem can be formulated as testing the following hypothesis:  
\begin{equation}\label{ht}
H_0: A\sim\mathcal{H}_m\Big(n,\frac{p_n+(k^{m-1}-1)q_n}{k^{m-1}}\Big)\ \ vs.\ \ H_1: A\sim \mathcal{H}_m^k\Big(n,p_n,q_n\Big).
\end{equation}
When $m=k=2$, problem (\ref{ht}) has been well studied in the literature.
Specifically, for extremely sparse scenario $p_n\asymp q_n\ll n^{-1}$, \cite{MNS15} show that $H_0$ and $H_1$ are always indistinguishable
in the sense that all tests are asymptotically powerless; 
for bounded degree case $p_n\asymp q_n\asymp n^{-1}$, the two models are distinguishable if and only if the signal-to-noise ratio (SNR) is greater than 1 \cite{MNS15,MNS17,YFS18a}; 
for dense scenario $p_n\asymp q_n\gg n^{-1}$, $H_0$ and $H_1$ are always distinguishable and a number of algorithms have been developed (see \cite{L16,GL17a,GL17b,BM17,AS17,BS16}). 
When $m=2$ and $k\geq3$, the above statements remain true for extremely sparse and dense scenarios; but for bounded degree scenario, SNR$>1$ is only a sufficient condition 
for successfully distinguishing $H_0$ from $H_1$ while a necessary condition remains an open problem (see \cite{AS17, BM16,YFS18b}). 
Abbe \cite{A17} provides a comprehensive review of the recent development in this field.
From the best of our knowledge, there is a lack of literature dealing with the testing problem (\ref{ht}) for general $m$.
The literature on hypergraph analysis mainly focused on community detection
(see \cite{ACK,GD14,GhDu,Bulo,cherke,Govin,KBG17,LCW17,Mnach}). 

\subsection{Our Contributions.}

The aim of this paper is to provide a study on hypergraph testing under a spectrum of hyperedge probability scenarios. 
Our results consist of four major parts. 
Section \ref{sec:contiguity} deals with 
the extremely sparse scenario $p_n\asymp q_n\ll n^{-m+1}$, in which we show that $H_0$ and $H_1$ are always indistinguishable
in the sense of contiguity.
Section \ref{sec:bdd:dgree} deals with bounded degree case $p_n\asymp q_n\asymp n^{-m+1}$, 
in which we show that $H_0$ and $H_1$ are distinguishable if the SNR of uniform hypergraph is greater than a certain threshold, but 
indistinguishable if the SNR is below another threshold. 
Interestingly, when $k=2$, the two thresholds are nearly tight in that they are of the same order $2^{-m}$(up to universal constants).
We also construct a powerful test statistic when SNR is greater than one based on counting the ``long loose cycles". 
Section \ref{EZuniform} deals with dense scenario $p_n\asymp q_n\gg n^{-m+1}$. 
We propose a test based on counting the hyperedges, $l$-hypervees, 
 and $l$-hypertriangles with $l$ determined by the order of $p_n$ (or $q_n$), and show that the power of the proposed test approaches one as the number of vertices goes to infinity.
In Section \ref{sec:non:uniform}, we extend some of the previous results to non-uniform hypergraph testing.
We propose a new test involving both edge and hyperedge information and show that it is generally more powerful than the
classic test using edge information only (see Remark \ref{new:test:rem}). 
The results of the present paper can be viewed as nontrivial extensions of the ordinary graph testing results such as \cite{MNS15,MNS17,GL17a}. 
Section \ref{sec:sim:real} provides
numerical studies to support our theory.
Possible extensions are discussed in Section \ref{sec:discussion} and proof of the main results are collected in Section \ref{sec:main:proofs}. 

Figure \ref{fig:phasetransition} displays a phase transition phenomenon in the special 3-uniform hypergraph,
based on our results in Sections \ref{sec:contiguity} and \ref{EZuniform}.
We find that $H_0$ and $H_1$ are indistinguishable if the hyperedge probabilities satisfy $p_n,q_n=o(n^{-2})$ (see red zone),
and are distinguishable if $p_n,q_n\gg n^{-2}$ (see green zone),
which is consistent with \cite{ALS18} who showed that community detection with weak consistency is possible 
if and only if $p_n,q_n\gg n^{-2}$. Therefore, the seemingly different perspectives, i.e., hypothesis testing and community detection,
appear to coincide here. In contrast,
the spectral algorithm proposed by \cite{GhDu} is able to detect communities with strong consistency if $p_n,q_n\gg n^{-2}(\log{n})^2$
(later improved to $p_n,q_n\gg n^{-2}\log{n}$ by \cite{LCW17,ALS16}). For bounded degree case $p_n,q_n\asymp n^{-2}$, detection algorithms better than random guess 
were proposed by \cite{ALS18,CLW18,FP16}.  Overall, in the references \cite{ALS18,CLW18}, the SNR conditions are not comparable to our $\kappa>1$ since unknown constants are involved in their conditions. The SNR condition in \cite{FP16} seems more restrictive than our condition $\kappa>1$.
\begin{figure}[H]
\centering
\includegraphics[scale=0.3]{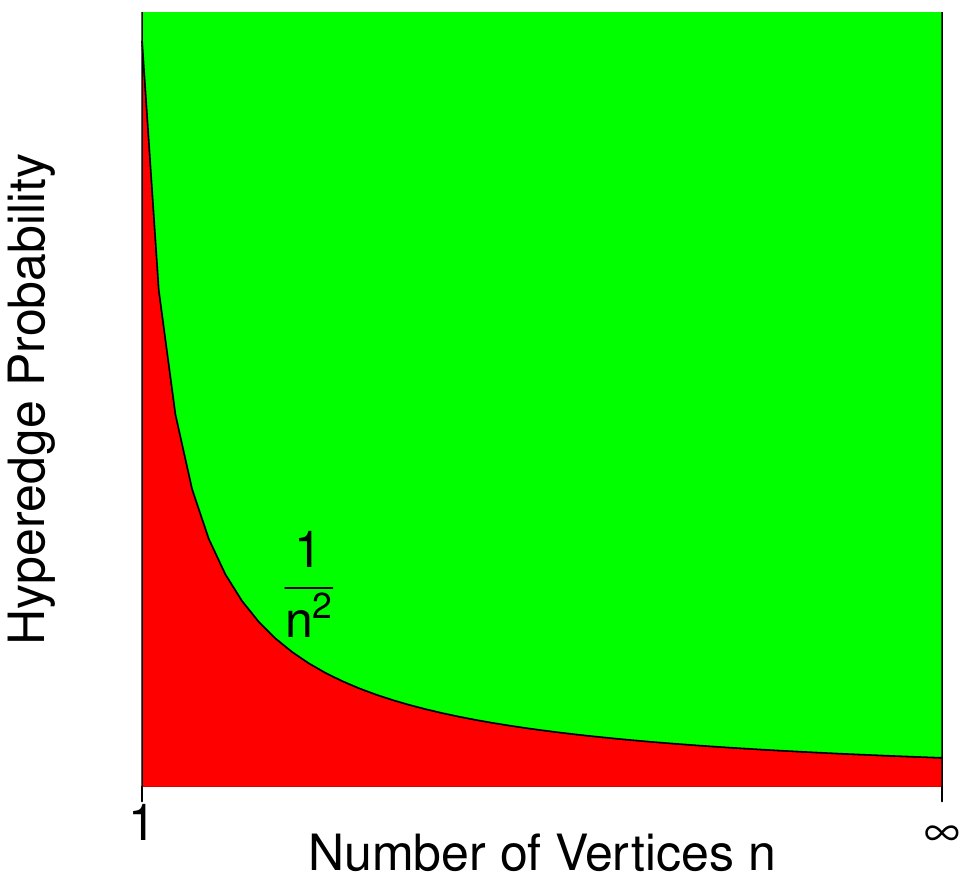}
\caption{\it\footnotesize Phase transition for 3-uniform hypergraph. Red: indistinguishable;
green: distinguishable.}
\label{fig:phasetransition}
\end{figure}

On the next page, Figure \ref{fig:phasetransitionbounded} demonstrates the distinguishable and indistinguishable regions for two-community graph (left) 
and two-community 3-uniform hypergraph (right) in the bounded degree regime, i.e., 
$p_n=\frac{a}{n^{m-1}}$, $q_n=\frac{b}{n^{m-1}}$. The regions are characterized by $(a,b)$ with $a>b>0$.
The left plot is based on \cite{MNS15} who derived the decision boundary $(a-b)^2=2(a+b)$.
The right plot is based on our Theorem \ref{sharp:SNR:phase:transition} with the decision boundary $(a-b)^2=4(a+3b)$. It can be observed that $m=3$ yields a larger indistinguishable region than $m=2$, which reveals a substantial difference
for hypothesis testing in the two models.
\begin{figure}[H]
\centering
\includegraphics[scale=0.3]{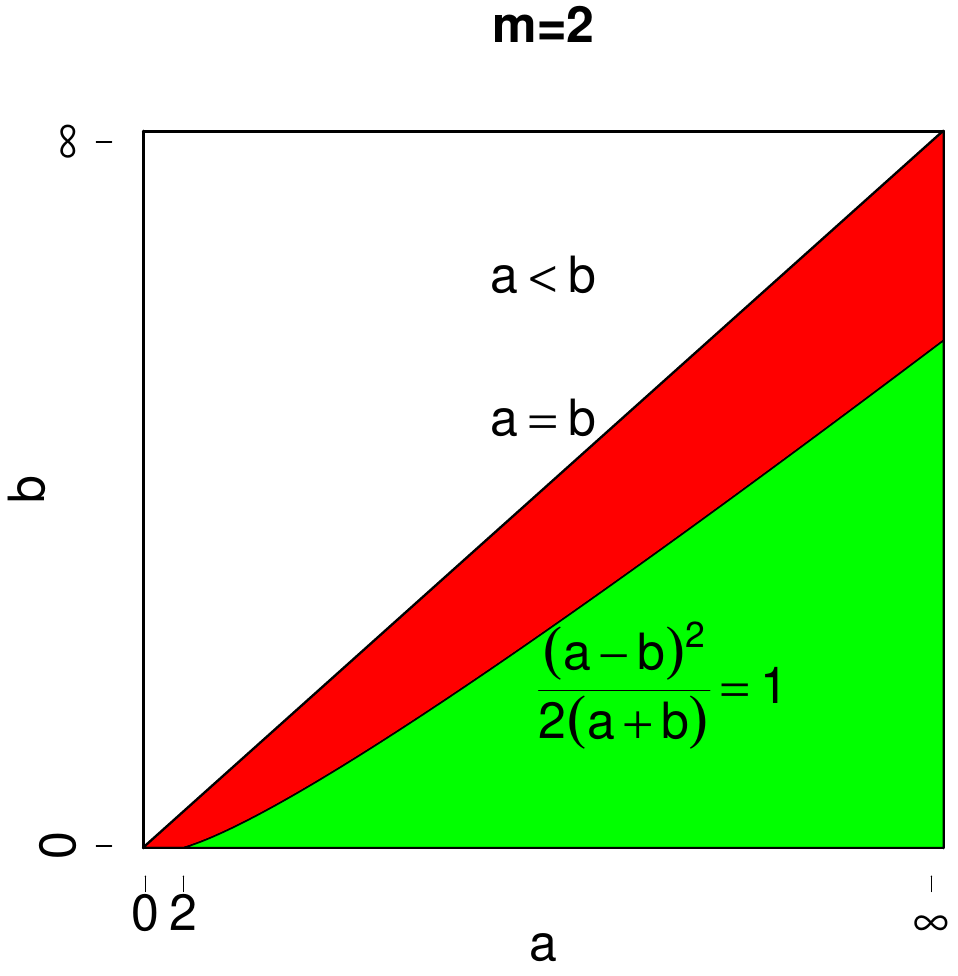}
\hspace{20mm}
\includegraphics[scale=0.3]{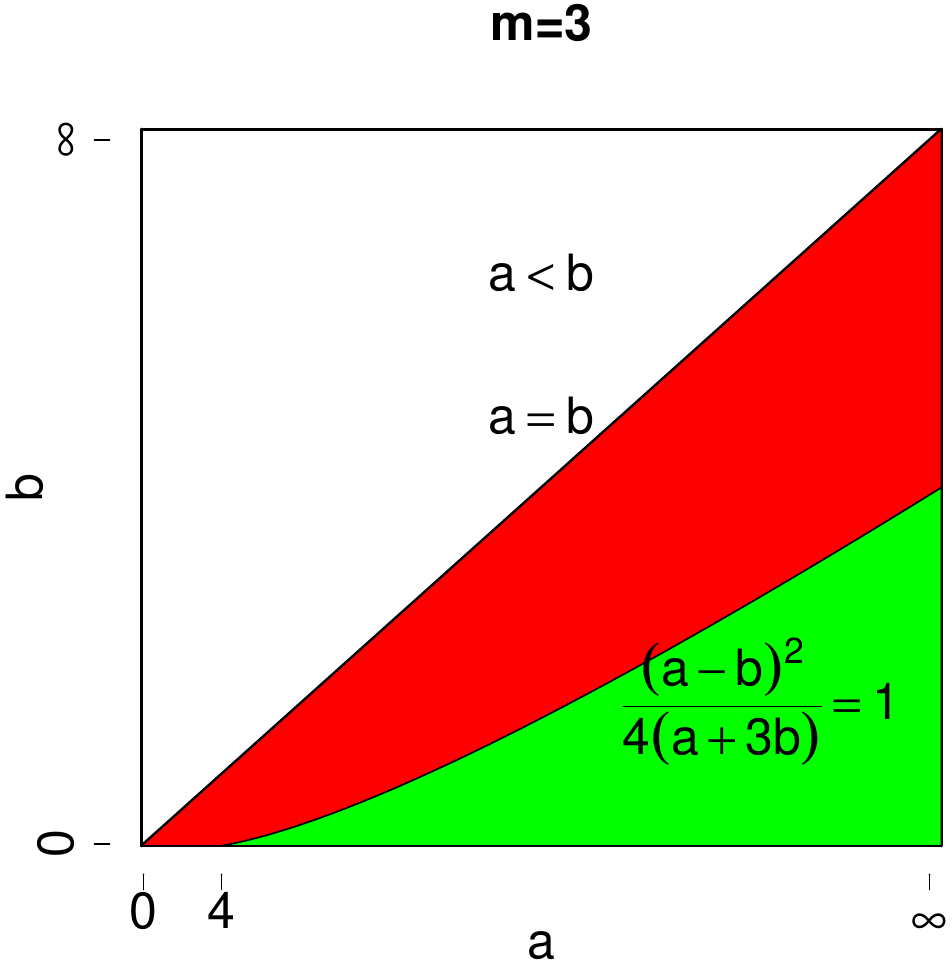}
\caption{\it\footnotesize Phase transition in bounded degree case. Red: indistinguishable;
green: distinguishable.}
\label{fig:phasetransitionbounded}
\end{figure}

 \section{Main Results.}
In this section, we present our main results in three parts, organized by the sparsity of the network. The contiguity theory for the extremely sparse case is presented in Section \ref{sec:contiguity}, followed by the contiguity and orthogonality result for the bounded degree case in Section \ref{sec:bdd:dgree}.  In Section \ref{EZuniform}, we construct a powerful test by counting the hyperedges, $l$-hypervees, and $l$-hypertriangles for the dense case.  Throughout this paper, we assume $k$ and $m$ are fixed positive integers.

\subsection{A Contiguity Theory for Extremely Sparse Case.}\label{sec:contiguity}
In this section, we consider the testing problem (\ref{ht}) with $p_n\asymp q_n\ll n^{-m+1}$, i.e., the hyperedge probability of 
the hypergraph is extremely low. For technical convenience, we only consider $p_n=\frac{a}{n^\alpha}$ and $q_n=\frac{b}{n^\alpha}$
with constants $a>b>0$ and $\alpha>m-1$. 
The results in this section may be extended to general orders of $p_n$ and $q_n$ with more cumbersome arguments.
We will show that no test can successfully distinguish $H_0$ from $H_1$ in such a situation.
The proof proceeds by showing that the probability measures associated with $H_0$ and $H_1$ 
are contiguous (see Theorem \ref{contiguous}). We remark that contiguity has also been
used to prove indistinguishability for ordinary graphs (see \cite{MNS15,MNS17}).

Let $\mathbb{P}_n$ and $\mathbb{Q}_n$ be sequences of probability measures on a common probability space $(\Omega_n,\mathcal{F}_n)$. 
We say that $\mathbb{P}_n$ and $\mathbb{Q}_n$ are mutually \textit{contiguous} if for every sequence of measurable sets $A_n\subset\Omega_n$, $\mathbb{P}_n(A_n)\rightarrow0$ if and only if $\mathbb{Q}_n(A_n)\rightarrow0$ as $n\rightarrow\infty$. They are said to be \textit{orthogonal} if there exists a sequence of measurable sets $A_n$ such that $\mathbb{P}_n(A_n)\rightarrow0$ and $\mathbb{Q}_n(A_n)\rightarrow1$ as $n\rightarrow\infty$. 
According to \cite{MNS15}, two probability models are indistinguishable 
if their associated probability measures are mutually contiguous, 
and two probability models are distinguishable if their associated probability measures are orthogonal. 
The following theorem shows that $H_0$ and $H_1$ are indistinguishable.
\begin{Theorem}\label{contiguous}
 If $\alpha>m-1$ and $a>b>0$ are fixed constants,
 then the probability measures associated with $H_0$ and $H_1$ are mutually contiguous.
 \end{Theorem}
The proof of Theorem \ref{contiguous} proceeds by showing that the ratio of the likelihood function of $H_1$ over $H_0$  converges in distribution  to 1 under $H_0$, which implies the contiguity of $H_1$ and $H_0$ \cite{J95}. 
Theorem \ref{contiguous} says that the hypergraphs in $H_0$ and $H_1$
are indistinguishable, and hence, no test can successfully separate the two hypotheses. 
One intuitive explanation is that when $\alpha>m-1$, the average degree of both hypergraph models converges to zero.
To see this, the average degree is 
\begin{equation}\label{average:degree:eqn}
\binom{n}{m-1}\frac{a+(k^{m-1}-1)b}{k^{m-1}n^{\alpha}},
\end{equation} 
which goes to zero as $n\rightarrow\infty$ if $\alpha>m-1$.
Therefore, the signals in both models are not strong enough to support a successful test.
It is easy to see that the average degree becomes bounded when $\alpha=m-1$
which will be investigated in the next section.

\subsection{Bounded Degree Case.}\label{sec:bdd:dgree}

In this section, we consider $p_n\asymp q_n\asymp n^{-m+1}$ which leads to bounded average degrees 
for the models in $H_0$ and $H_1$; see (\ref{average:degree:eqn}). 
For convenience, let us denote $p_n=\frac{a}{n^{m-1}}$ and $q_n=\frac{b}{n^{m-1}}$
for fixed $a>b>0$.
Define the signal to noise ratio (SNR) for $H_0$ and $H_1$ as 
\begin{equation}\label{def:SNR}
\kappa=\frac{(a-b)^2}{k^{m-1}(m-2)![a+(k^{m-1}-1)b]}.
\end{equation}
When $m=k=2$, it is easy to check that $\kappa=\frac{(a-b)^2}{2(a+b)}$ which becomes the classic SNR of ordinary stochastic block models considered by \cite{MNS15,AS16,AS17}. 
Hence, it is reasonable to view $\kappa$ defined in (\ref{def:SNR}) as a generalization of the classic SNR to the hypergraph model $\mathcal{H}_m^k(n,\frac{a}{n^{m-1}},\frac{b}{n^{m-1}})$.
Like the classic SNR, the value of $\kappa$ characterizes the separability between communities. 
Intuitively, when $\kappa$ is large which means that the communities are very different, 
the testing problem (\ref{ht}) becomes simpler.
The following result showes that when $\kappa>1$, successful testing becomes possible.
\begin{Theorem}\label{orthogonality}
Suppose that $a>b>0$ are fixed constants, $m,k\ge2$.
If $\kappa>1$, then the probability measures associated with $H_0$ and $H_1$ are orthogonal.
\end{Theorem}
We prove Theorem \ref{orthogonality} by constructing a sequence of events dependent on the number of long loose cycles and showing that the probabilities of the events converge to 1 (or 0) under $H_0$ (or $H_1$), based on the high moments driven asymptotic normality theorem from Gao and Wormald \cite{GWALD}. 
Theorem \ref{orthogonality} says that it is possible to distinguish the hypotheses $H_0$ and $H_1$ provided that
$\kappa>1$. Abbe and Sandon \cite{AS17} obtained relevant results in the ordinary graph setting,
i.e., $m=2$ and $k\ge2$ in our case; see Corollary 2.8 therein which states that community detection in
polynomial time becomes possible if SNR$>1$.
Whereas Theorem \ref{orthogonality}
holds for arbitrary $m,k\ge2$. Hence our result can be viewed as an extension of 
\cite{AS17} to hypergraph setting. 

Let us now propose a test statistic based on ``long loose cycles'' that can successfully distinguish $H_0$ and $H_1$ when $\kappa>1$.
Let $\xi_n$ be a positive integer sequence diverging along with $n$.
Let $X_{\xi_n}$ be the number of loose cycles, each consisting of exactly $\xi_n$ edges.
Define 
\[ \mu_{n0}=\frac{\lambda_m^{\xi_n}}{2\xi_n},\,\,\,\,
\mu_{n1}=\mu_{n0}+\frac{k-1}{2\xi_n}\Big[\frac{a-b}{k^{m-1}(m-2)!}\Big]^{\xi_n},
\]
where $\lambda_m=\frac{a+(k^{m-1}-1)b}{{k}^{m-1}(m-2)!}$ for any $m\ge2$.
Note that when $m=2$, $\lambda_m=\frac{a+(k-1)b}{k}$ is the average degree \cite{BM16}. 
Let $\mathbb{P}_{H_1}$ denote the probability measure induced by $A$ under $H_1$. We have the following theorem about the asymptotic property of $X_{\xi_n}$.

\begin{Theorem}\label{normaltest}
Suppose $\kappa>1$ and
$1\ll \xi_n\leq\delta_0\log_{\lambda_m}\log_{\gamma}n$, where $\gamma>1$ and $0<\delta_0<2$ are constants. 
Then, under $H_l$ for $l=0,1$,
$\frac{X_{\xi_n}-\mu_{nl}}{\sqrt{\mu_{nl}}}\overset{d}{\to}N(0,1)$ as $n\to\infty$. 
Furthermore, for any constant $C>0$, $\mathbb{P}_{H_1}\left(\big|\frac{X_{\xi_n}-\mu_{n0}}{\sqrt{\mu_{n0}}}\big|>C\right)\rightarrow1$ as $n\to\infty$.
\end{Theorem}

The proof is based on the asymptotic normality theory developed by \cite{GWALD}. 
According to Theorem \ref{normaltest}, we propose the following test statistic 
\[
T_{\xi_n}=\frac{X_{\xi_n}-\mu_{n0}}{\sqrt{\mu_{n0}}}.
\]
We remark that computation of $T_{\xi_n}$ is typically in super-polynomial time since it requires to find $X_{\xi_n}$ which has complexity $n^{O(\xi_n)}$.
By Theorem \ref{normaltest}, $T_{\xi_n}\overset{d}{\to}N(0,1)$ under $H_0$. Hence, we construct the following testing rule 
at significance level $\alpha\in(0,1)$:
\[
\textrm{reject $H_0$ if and only if $|T_{\xi_n}|>z_{\alpha/2}$},
\]
where $z_{\alpha/2}$ is the $(1-\alpha/2)$-quantile of $N(0,1)$.
It follows by Theorem \ref{normaltest} that $\mathbb{P}_{H_1}(|T_{\xi_n}|>z_{\alpha/2})\to1$,
i.e., the power of $T_{\xi_n}$ approaches one when $\kappa>1$.

Theorem \ref{normaltest} requires $\xi_n\to\infty$ and to grow slower than an iterative logarithmic order. 
This is due to the use of \cite{GWALD} which requires $\xi_n$ to diverge with $\xi_n\lambda_m^{\xi_n}=o(\log n)$. 
In practice, we suggest choosing $\xi_n=\floor{\delta_0\log_{\lambda_m}\log_{\gamma}n}$ with
$\gamma$ close to 1 and $\delta_0$ close to 2. Such $\gamma$ and $\delta_0$ will make
$\xi_n$ suitably large so that the test statistic $T_{\xi_n}$ becomes valid.
For instance, Table \ref{table:kn} demonstrates the values of $\xi_n$ along with $n$
with $\delta_0=1.99$, $\gamma=1.01$, $\lambda_m=10$.
We can see that, for a moderate range of $n$, the values of $\xi_n$ are sufficiently large to make the test valid. When $\xi_n=l$ is fixed and the exact $\alpha$-level test is needed, we should use Poisson distribution as the null limiting distribution. In this case, the number of $l$-loose cycle $X_{l}$ converges in distribution to Poisson distribution with mean $\mu_{0}=\frac{\lambda_m^l}{2l}$ under $H_0$ (It's implied by the proof of Theorem 2.5).
\begin{table}[h]
\caption{\it\footnotesize Minimal $n$ to achieve a desirable value of $\xi_n$.}\label{table:kn}

\begin{center}
\begin{tabular}{|c|cccc| } 
 \hline
 Desirable $\xi_n$         & 3 & 4& 5 & 6\\
  \hline
 Minimal $n$           &  2  &3 & 25     & 29786 \\
  \hline
\end{tabular}
\end{center}
\end{table}
It should be mentioned that the calculation of $T_{\xi_n}$ requires known values of $a$ and $b$.
When $a$ and $b$ are unknown, motivated by the ordinary graph \cite{MNS15}, they can be estimated as follows. 
Define
\[
\widehat{\lambda}_m=\frac{n^{m-1}|\mathcal{E}|}{(m-2)!\binom{n}{m}},\ \ \widehat{f}=(2\xi_nX_{\xi_n}-\widehat{\lambda}_m^{\xi_n})^{\frac{1}{\xi_n}},
\]
where $|\mathcal{E}|$ is the number of observed hyperedges and $X_{\xi_n}$ is the number of loose cycles of length $\xi_n$.
Let $\widehat{a}_n=(m-2)!\Big[\widehat{\lambda}_m+(k^{m-1}-1)(k-1)^{-\frac{1}{\xi_n}}\widehat{f}\Big]$ and $\widehat{b}_n=(m-2)!\Big[\widehat{\lambda}_m-(k-1)^{-\frac{1}{\xi_n}}\widehat{f}\Big]$.
The following theorem says that $\widehat{a}_n$ and $\widehat{b}_n$ are consistent estimators of $a$ and $b$, respectively.

\begin{Theorem}\label{consistent:a:b} Suppose $\kappa>1$ and $\xi_n$ satisfies the condition in Theorem \ref{normaltest}. 
Then $\widehat{a}_n\to a$ and $\widehat{b}_n\to b$ in probability.
\end{Theorem}

Another interesting question is to investigate for what values of $\kappa$ a successful test becomes impossible.
When $m=k=2$, \cite{MNS15} showed that no test can successfully distinguish $H_0$ from $H_1$ provided $\kappa<1$;
and a successful test becomes possible provided $\kappa>1$.
It is substantially challenging to obtain such a sharp result when $k$ becomes larger.
For instance, in the ordinary graph setting, \cite{NN14} obtained a (non-sharp) condition in terms of SNR when $k\ge3$ under which
successful test becomes impossible.  
In Theorem \ref{ctgbounded} below, we address a similar question in the hypergraph setting.
For any integers $m\ge3,k\ge2$, define
$\tau_1(m,k)=
\binom{m}{2}^{-1}\sum_{i=1}^{\ceil{\frac{m}{2}-1}}\frac{1}{k^{2i-1}}\binom{m}{i+2}$
and 
$\tau_2(m,k)=1+\binom{m}{2}^{-1}\sum_{i=1}^{m-2}\frac{1}{k^{2i}}\binom{m}{i+2}$.
The quantities $\tau_1(m,k)$ and $\tau_2(m,k)$ will jointly characterize a spectrum of $(m,k,\kappa)$
such that successful test does not exist. 

\begin{Theorem}\label{ctgbounded}
Suppose that $m\ge3, k\ge2$ are integers satisfying $\tau_1(m,k)\leq1$,
$a>b>0$ are fixed constants and $\alpha=m-1$.
If 
\begin{equation}\label{a:range:of:kappa}
0<\kappa<\frac{1}{\tau_2(m,k)(k^2-1)},
\end{equation}
then the probability measures associated with $H_0$ and $H_1$ are mutually contiguous.
\end{Theorem}
The proof of Theorem \ref{ctgbounded} relies on Janson's contiguity theory \cite{J95}.
Theorem \ref{ctgbounded} says that when $\tau_1(m,k)\le1$ and $\kappa$ falls in the range (\ref{a:range:of:kappa}),
there is no test that can successfully distinguish the hypotheses $H_0$ and $H_1$.
It should be emphasized that the condition $\tau_1(m,k)\leq1$ holds for a broad range of pairs $(m,k)$. 
For instance, such condition holds for any $k\ge2$ and $3\le m\le 6$. To see this, for any $k\ge2$, 
$\tau_1(3,k)=\frac{1}{3k}<1$, $\tau_1(4,k)=\frac{2}{3k}<1$, $\tau_1(5,k)=\frac{1}{k}+\frac{1}{2k^3}<1$ and $\tau_1(6,k)=\frac{4}{3k}+\frac{1}{k^3}<1$. Note that $m\leq 6$ covers most of the practical cases (see \cite{GhDu}). 

Combining Theorems \ref{ctgbounded} and \ref{orthogonality}, it is still unknown
whether $H_0$ and $H_1$ are distinguishable when $\frac{1}{\tau_2(m,k)(k^2-1)}\le\kappa\le1$. 
Such result can be further improved for the special case $k=2$, and we close the gap if in addition $m=3$, as presented in the following theorem. 
\begin{Theorem}\label{sharp:SNR:phase:transition}
For $k=2$, the following results hold.
\begin{enumerate}
\item\label{thm:sharp:ii}
For any $m\ge2$, if $0<\kappa<2^{2-m}$, then $H_0$ and $H_1$ are indistinguishable. 
Moreover, for any given constant $\kappa_0$ such that $\kappa_0>\frac{m(m-1)\log{2}}{2^{m-1}-1}$, there exist $a>b>0$ such that
the SNR $\kappa$ for the hypotheses $H_0$ and $H_1$ is equal to $\kappa_0$,
and $H_0$ and $H_1$ are distinguishable by the likelihood ratio test.
\item\label{thm:sharp:i} For any $m\ge2$, if $0<\kappa<\frac{m(m-1)}{2N_m}$, where
$N_m=[3^m+(-1)^m]/4-2^{m-1}+1/2$,
then $H_0$ and $H_1$ are indistinguishable.
\end{enumerate}
\end{Theorem}
Specifically, Part \ref{thm:sharp:ii}
indicates that, when SNR is below $2^{2-m}$, $H_0$ and $H_1$ are indistinguishable;
while they are possible to be distinguishable when SNR is greater than $\frac{m(m-1)\log{2}}{2^{m-1}-1}$. 
Essentially, Part \ref{thm:sharp:ii} implies that the derived SNR upper and lower bounds satisfy the following relationship:
\[
0<\sup_{m\ge1}\frac{1}{m^2}\cdot\frac{\text{SNR upper bound}}{\text{SNR lower bound}}<\infty,
\]
and
\[
0<\min_{m\ge1}\frac{\text{SNR lower bound}}{2^{-m}}\le\max_{m\ge1}\frac{\text{SNR lower bound}}{2^{-m}}<\infty.
\]
Part \ref{thm:sharp:i} in Theorem \ref{sharp:SNR:phase:transition} provides an SNR interval for $H_1$ and $H_0$ to be indistinguishable. When $m=3$, $N_3=3$ leads to an SNR interval $0<\kappa<1$
which is sharp since $\kappa>1$ implies that $H_0$ and $H_1$ are distinguishable thanks to Theorem \ref{orthogonality}. For $k=2$ and general $m$, the upper bound $\frac{m(m-1)}{2N_m}$ may be less sharp as $m$ grows. In particular, when $m$ is large, the upper bound is of order $m^23^{-3}$, which can actually be improved as shown in Part \ref{thm:sharp:ii} of Theorem \ref{sharp:SNR:phase:transition}. Specifically, Part \ref{thm:sharp:ii} indicates that, when SNR is below $2^{2-m}$, $H_0$ and $H_1$ are indistinguishable; while they are possible to be distinguishable when SNR is greater than $m^22^{-m}$ up to a constant in the special case $k=2$.
Note that for $3\leq m\leq 8$, $\frac{m(m-1)}{2N_m}>2^{2-m}$ and for $m\geq9$, $\frac{m(m-1)}{2N_m}<2^{2-m}$.

The proof of Theorem \ref{sharp:SNR:phase:transition} relies on
a truncation technique to show the stochastic boundedness of the likelihood ratio
and a delicate derivation of a lower bound for the truncated likelihood ratio.
An interesting consequence of Part \ref{thm:sharp:ii} is that the likelihood ratio test is possible to distinguish $H_0$ and $H_1$
even when $\kappa$ is below 1 (but greater than $\frac{m(m-1)\log{2}}{2^{m-1}-1}$). However, the computation of the likelihood ratio 
is NP-hard. When $\kappa>1$, the $l$-cycle based test can distinguish $H_0$ and $H_1$ as well (see Theorem \ref{normaltest}),
and is computationally less expensive.

\begin{Remark}
We provide more details about why truncation technique is needed in our setting.
The proof of Theorem \ref{sharp:SNR:phase:transition} relies on the first moment technique which requires the analysis of
$\mathbb{E}_1Y_n$ where $\mathbb{E}_1$ is the expectation taken under $H_1$ and $Y_n=\frac{dP_1}{dP_0}$ is the likelihood ratio of $H_1$ to $H_0$.
We find that the expression of $\mathbb{E}_1Y_n$ includes terms like 
\begin{equation}\label{eqn:2}
\mathbb{E}_\sigma\exp\left(\sum_{i_1<\cdots<i_m}\textrm{poly}(\sigma_{i_1},\ldots,\sigma_{i_m})\right)
\end{equation}
where $\text{poly}(\sigma_{i_1},\ldots,\sigma_{i_m})$ is an $m$-th-order polynomial of $\sigma_{i_1},\ldots,\sigma_{i_m}\in\{\pm1\}$.
When $m=2$, (\ref{eqn:2}) becomes a second-order polynomial which is asymptotically $\chi^2$ by CLT. And so, (\ref{eqn:2}) is heuristically $\mathbb{E}\exp(\text{const}\times \chi^2)$ which is finite.
This is why no truncation technique is needed here.

However, when $m=3$, the above polynomial is third-order which is asymptotically $Z^3$ where $Z\sim N(0,1)$. And as a result, (\ref{eqn:2}) is heuristically $\mathbb{E}\exp(\text{const}\times Z^3)$ which is infinite.
This is why we used the truncation technique, i.e., to truncate the likelihood ratio on an even with high probability so that the higher-order polynomials are well controlled, and the truncated likelihood ratio has a finite expectation.
\end{Remark}

\subsection{A Powerful Test for Dense Uniform Hypergraph.}\label{EZuniform}

In this section, we consider the problem of testing community structure in dense $m$-uniform hypergraphs
with $p_n\asymp q_n\gg n^{-m+1}$.
Our approach is based on counting the hyperedges, $l$-hypervees, and $l$-hypertriangles in the observed hypergraph.
To ensure the success of our test, $l$ needs to be properly selected according to the hyperedge probability of the model.
Under such correct selection, we derive asymptotic normality for the test and analyze its power.
We also discuss the effect of misspecified $l$ in Remark \ref{rem:misspecified:l}. 
Our method can be viewed as a generalization of \cite{GL17a,GL17b} from ordinary graph testing.  
The different features of the hypergraph cycles make our generalization nontrivial.

For convenience, let us denote $p_n=\frac{a_n}{n^{m-1}}$ and $q_n=\frac{b_n}{n^{m-1}}$ with 
diverging $a_n,b_n$. 
Therefore, (\ref{ht}) becomes the following hypothesis testing problem: 
 \begin{equation}\label{ht1}
H_0': A\sim\mathcal{H}_m\Big(n,\frac{a_n+(k^{m-1}-1)b_n}{k^{m-1}n^{m-1}}\Big)\ \ vs.\ \ H_1': A\sim \mathcal{H}_m^k\Big(n,\frac{a_n}{n^{m-1}},\frac{b_n}{n^{m-1}}\Big).
\end{equation} 
We temporarily assume that there exists an integer $1\leq l\leq \frac{m}{2}$ such that
$n^{l-1}\ll a_n\asymp b_n \ll n^{l-\frac{2}{3}}$. Such a requirement will be relaxed by invoking a sparsification technique. 
Note that model (\ref{ht1}) allows $1\ll a_n\asymp b_n\ll n^{1/3}$ (with $l=1$),
compared with spectral algorithm \cite{GhDu} which requires $a_n\gg (\log{n})^{2}$ or $a_n\gg \log{n}$ in \cite{LCW17}.

We consider the following degree-corrected SBM in \cite{ACT19,AC19,GL17a} which is more general than (\ref{sbm:model}) and generalizes its counterpart in ordinary graphs.
Let $\{W_i, i=1,\dots,n\}$ be nonnegative i.i.d. random variables with $\mathbb{E}(W_1^2)=1$ and 
$\{\sigma_i, i=1,\dots, n\}$ be i.i.d. random variables from multinomial distribution Mult$(k,1,1/k)$. Assume that $W_i$'s and $\sigma_i$'s are independent. Given $W_i$'s and $\sigma_i$'s, the $A_{i_1i_2\ldots i_m}$'s, with pairwise distinct $i_1,\ldots,i_m$,
are conditional independent satisfying
 \begin{eqnarray}\label{degree:correct:sbm}
 \mathbb{P}(A_{i_1i_2\dots i_m}=1|W,\sigma)&=&W_{i_1}\dots W_{i_m}p_{i_1i_2\dots i_m}(\sigma), \\ \nonumber 
 \mathbb{P}(A_{i_1i_2\dots i_m}=0|W,\sigma)&=&1-W_{i_1}\dots W_{i_m}p_{i_1i_2\dots i_m}(\sigma),
 \end{eqnarray}
where $W=(W_1,\ldots,W_n)$,  
\[
p_{i_1i_2\dots i_m}(\sigma)=\left\{\begin{array}{cc}
\frac{a_n}{n^{m-1}},&\sigma_{i_1}=\dots=\sigma_{i_m}\\
\frac{b_n}{n^{m-1}},&\mbox{otherwise}
\end{array}\right..\ \
\]
We call (\ref{degree:correct:sbm}) the degree-corrected SBM in hypergraph setting. The degree-correction
weights $W_i$'s can capture the degree of inhomogeneity exhibited in many social networks.
When $m=2$, (\ref{degree:correct:sbm}) reduces to the classical degree-corrected SBM for ordinary graphs (see \cite{ACT19,AC19,GL17a}). For ordinary graphs, 
\cite{GL17a} proposed a test through counting small subgraphs to distinguish the degree-corrected SBM from an Erd\"{o}s-R\'{e}nyi model. 
In what follows, we generalize their results to hypergraphs through counting small sub-hypergraphs, including hyperedges, $l$-hypervee, and $l$-hypertriangles, with definitions given below. 

\begin{definition}
An $l$-\textit{hypervee} consists of two hyperedges with $l$ common vertices. 
An $l$-\textit{hypertriangle} is an $l$-cycle consisting of three hyperedges.
\end{definition}
For example, in Figure \ref{hypervee}, the hyperedge set $\{(v_1,v_2,v_3,v_4), (v_3,v_4,v_5,v_6)\}$ is a 2-hypervee, and $\{(v_1,v_2,v_3,v_4), (v_3,v_4,v_5,v_6),(v_5,v_6, v_1,v_2)\}$ is a 2-hypertriangle.

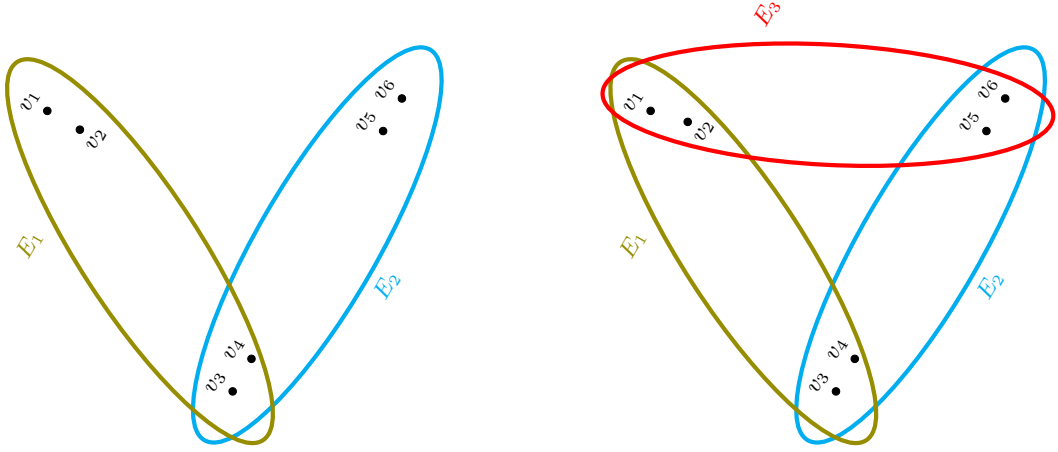
\begin{figure}[h]
\vspace{-20mm}
\begin{subfigure}{0.4\textwidth}
\trimbox{1.5cm -.5cm -1cm 0cm}{ 
\rotatebox{60}{
 \begin{tikzpicture}[
    dot/.style={draw, circle, inner sep=1pt, fill=black},
    group/.style={draw=#1, ellipse, ultra thick, minimum width=6cm, minimum height=1.6cm}
    ]
    \node[dot,label=above:$v_3$] at (0,0) (v3) {};
    \node[dot,label=above:$v_4$] at (0.5,0) (v4) {};
    \node[dot,label=above:$v_5$] at (4,0) (v5) {};
    \node[dot,label=below:$v_2$] at (2,3.5) (v2) {};
    \node[dot,label=above:$v_6$] at (4.5,0) (v6) {};
    \node[dot,label=$v_1$] at (2,4) (v1) {};
    
    \node[group=cyan, label={[cyan]below:$E_2$}] at (2.25,0) (E2) {};
    \node[group=olive, label={[olive]above right:$E_1$}, rotate=63] at (1,2) (E1) {};
\end{tikzpicture}
}
}
\end{subfigure}
\hspace{10mm}
\begin{subfigure}{0.4\textwidth }
\trimbox{1.3cm -.5cm -1cm 0cm}{ 
\rotatebox{60}{
 \begin{tikzpicture}[
    dot/.style={draw, circle, inner sep=1pt, fill=black},
    group/.style={draw=#1, ellipse, ultra thick, minimum width=6cm, minimum height=1.6cm}
    ]

    \node[dot,label=above:$v_3$] at (0,0) (v3) {};
    \node[dot,label=above:$v_4$] at (0.5,0) (v4) {};
    \node[dot,label=above:$v_5$] at (4,0) (v5) {};
    \node[dot,label=below:$v_2$] at (2.12,3.5) (v2) {};
    \node[dot,label=above:$v_6$] at (4.5,0) (v6) {};
    \node[dot,label=$v_1$] at (2,4) (v1) {};
    
    \node[group=cyan, label={[cyan]below:$E_2$}] at (2.25,0) (E2) {};
    \node[group=olive, label={[olive]above right:$E_1$}, rotate=63] at (1,2) (E1) {};
    \node[group=red, label={[red]above left:$E_3$}, rotate=-63] at (3.25,2) (E3) {};
\end{tikzpicture}
}
}
\end{subfigure}
 \caption{\it\footnotesize Examples of hypervee (left) and hypertriangle (right) with two common vertices between consecutive hyperedges.}
 \label{hypervee}
\end{figure}

Consider the following probabilities of hyperedge, hypervee and hypertriangle in $\mathcal{H}_m^k\Big(n, \frac{a_n}{n^{m-1}}, \frac{a_n}{n^{m-1}}\Big)$:
 \begin{eqnarray*}
 E&=&\mathbb{P}(A_{i_1i_2\dots i_m}=1),\\
 V&=&\mathbb{P}(A_{i_1i_2\dots i_m}A_{i_{m-l+1}\dots i_{2m-l}}=1),\\
 T&=&\mathbb{P}(A_{i_1i_2\dots i_m}A_{i_{m-l+1}\dots i_{2m-l}}A_{i_{2m-2l+1}\dots i_{3(m-l)}i_1\dots i_l}=1).
 \end{eqnarray*}
It follows from direct calculations that
 \begin{eqnarray*}
 E&=&(\mathbb{E}W_1)^m\frac{a_n+(k^{m-1}-1)b_n}{n^{m-1}k^{m-1}},\\
 V&=&(\mathbb{E}W_1)^{2(m-l)}\Bigg(\frac{(a_n-b_n)^2}{n^{2(m-1)}k^{2m-l-1}}+\frac{2(a_n-b_n)b_n}{n^{2(m-1)}k^{m-1}}+\frac{b_n^2}{n^{2(m-1)}}\Bigg),\\
 T&=&(\mathbb{E}W_1)^{3(m-2l)}\Bigg(\frac{(a_n-b_n)^3}{n^{3(m-1)}k^{3(m-l)-1}}+\frac{3(a_n-b_n)^2b_n}{n^{3(m-1)}k^{2m-l-1}}+\frac{3(a_n-b_n)b_n^2}{n^{3(m-1)}k^{m-1}}+\frac{b_n^3}{n^{3(m-1)}}\Bigg).
 \end{eqnarray*}
Define $\mathcal{T}=T-\Big(\frac{V}{E}\Big)^{3}$.
The following result demonstrates a strong relationship between $\mathcal{T}$ and $H_0', H_1'$.

  \begin{Proposition}\label{proph1}
  Under $H_0'$, $\mathcal{T}=0$ and under $H_1'$, $\mathcal{T}\neq0$.
  \end{Proposition}
 
Proposition \ref{proph1} says that $H_0'$ holds if and only if $\mathcal{T}=0$. 
Hence, it is reasonable to use an empirical version of $\mathcal{T}$, namely, $\widehat{\mathcal{T}}$, as a test statistic for (\ref{ht1}). 

Prior to constructing $\widehat{\mathcal{T}}$,
let us introduce some notation. For convenience, 
we use $i_1:i_m$ to represent the ordering $i_1i_2\dots i_m$.
Also define $C_{2m-l}(A)$ and  $C_{3(m-l)}(A)$ for any adjacency tensor $A$ as follows.
\begin{eqnarray*}
C_{2m-l}(A)&=&A_{i_1:i_m}A_{i_{m-l+1}:i_{2m-l}}+A_{i_2:i_{m+1}}A_{i_{m-l+2}:i_{2m-l}i_1}+\dots+A_{i_{2m-l}i_1:i_{m-1}}A_{i_{m-l}:i_{2m-l-1}},\\
C_{3(m-l)}(A)&=&A_{i_1:i_m}A_{i_{m-l+1}:i_{2m-l}}A_{i_{2m-2l+1}:i_{3(m-l)}i_1:i_l}
+A_{i_2:i_{m+1}}A_{i_{m-l+2}:i_{2m-l+1}}A_{i_{2m-2l+2}:i_{3(m-l)}i_1:i_{l+1}}\\
&&+\dots+A_{i_{m-l}:i_{2m-l-1}}A_{i_{2(m-l)}:i_{3(m-l)}i_1:i_{l-1}}A_{i_{3(m-l)}i_1:i_{m-1}}.
\end{eqnarray*}
Note that $C_{2m-l}(A)$ is the number of hypervees in the given vertex ordering $i_1i_2\dots i_{2m-l}$, 
while $C_{3(m-l)}(A)$ counts the number of hypertriangles in the given vertex ordering $i_1i_2\dots i_{3(m-l)}$.
Define $\widehat{E}$, $\widehat{V}$, $\widehat{T}$ as the empirical versions of $E,V,T$:
\begin{equation}\label{ez:test:unif:aux}
\widehat{E}=\frac{1}{\binom{n}{m}}\sum_{i\in c(m,n)}A_{i_1:i_m},
\widehat{V}=\frac{1}{\binom{n}{2m-l}}\sum_{i\in c(2m-l,n)}\frac{C_{2m-l}(A)}{2m-l},
\widehat{T}=\frac{1}{\binom{n}{3(m-l)}}\sum_{i\in c(3(m-l),n)}\frac{C_{3(m-l)}(A)}{m-1},
\end{equation}
where, for any positive integers $s,t$, $c(s,t)=\{(i_1,\dots,i_s): 1\leq i_1<\dots<i_s\leq t\}$.
We have the following asymptotic normality result.
  
 \begin{Theorem}\label{normality}
 Suppose $\mathbb{E}W_1^4=O(1)$ and $n^{l-1}\ll a_n\asymp b_n \ll n^{l-\frac{2}{3}}$ for some integer $1\leq l\leq \frac{m}{2}$. 
Moreover, let
\begin{equation}\label{limit:ez:test}
\delta:=\frac{\sqrt{\binom{n}{3(m-l)}(m-l)}}{\sqrt{T}}\Bigg[T-\Big(\frac{V}{E}\Big)^{3}\Bigg]\in[0,\infty).
\end{equation}
Then we have, as $n\to\infty$,
 \begin{equation}\label{normality1}
 \frac{\sqrt{\binom{n}{3(m-l)}(m-l)}\Big[\widehat{T}-\Big(\frac{\widehat{V}}{\widehat{E}}\Big)^3\Big]}{\sqrt{\widehat{T}}}-\delta \overset{d}{\to} N(0,1),
 \end{equation}
 \begin{equation}\label{normality2}
 2\sqrt{\binom{n}{3(m-l)}(m-l)}\Bigg[
 \sqrt{\widehat{T}}-\Big(\frac{\widehat{V}}{\widehat{E}}\Big)^{\frac{3}{2}}\Bigg]-\delta\overset{d}{\to} N(0,1).
 \end{equation}
 \end{Theorem}
 
  When $l=1$ and $m=2$, Theorem \ref{normality} becomes Theorem 2.2 of \cite{GL17a}. 

Following (\ref{normality1}) in Theorem \ref{normality}, we can construct a test statistic for (\ref{ht1}) as
\begin{equation}\label{ez:test:unif}
\widehat{\mathcal{T}}_m=\frac{\sqrt{\binom{n}{3(m-l)}(m-l)}\Big[\widehat{T}-\Big(\frac{\widehat{V}}{\widehat{E}}\Big)^3\Big]}{\sqrt{\widehat{T}}}.
\end{equation}
In practice, $\widehat{T}$ might be close to zero which may cause computational instability,
an alternative test can be constructed based on (\ref{normality2}) as
\begin{equation}\label{ez:test:unif:prime}
\widehat{\mathcal{T}}'_m=2\sqrt{\binom{n}{3(m-l)}(m-l)}\Bigg[\sqrt{\widehat{T}}-\Big(\frac{\widehat{V}}{\widehat{E}}\Big)^{\frac{3}{2}}\Bigg].
\end{equation}
We remark that computation of $\widehat{\mathcal{T}}_m$ and $\widehat{\mathcal{T}}'_m$
is in polynomial time since the computations of $\widehat{T}$, $\widehat{V}$, and $\widehat{E}$
all have complexity $O(n^{3(m-l)})$.
Theorem \ref{normality} proves asymptotic normality for $\widehat{\mathcal{T}}_m$ and $\widehat{\mathcal{T}}'_m$
under both $H_0'$ and $H_1'$. 
Under $H_0'$, i.e., $\delta=0$, both $\widehat{\mathcal{T}}_m$ and $\widehat{\mathcal{T}}'_m$ are asymptotically standard normal.
Under $H_1'$, both $\widehat{\mathcal{T}}_m$ and $\widehat{\mathcal{T}}'_m$
are asymptotically normal with mean $\delta>0$ and unit variance.
When $\widehat{T}$ has a large magnitude, both test statistics
can be used to construct valid rejection regions.

The following Theorem \ref{power} says that the power of our test tends to one if $\delta$ goes to infinity.  

\begin{Theorem}\label{power}
 Suppose $\mathbb{E}W_1^4=O(1)$ and $n^{l-1}\ll a_n\asymp b_n \ll n^{l-\frac{2}{3}}$ for some integer $1\leq l\leq \frac{m}{2}$. 
 Under $H_1'$, as $n,\delta\rightarrow\infty$,
 $\mathbb{P}(|\widehat{\mathcal{T}}_m|>z_{\alpha/2})\rightarrow1$.
 The same result holds for $\widehat{\mathcal{T}}'_m$.
\end{Theorem}

\begin{Remark}\label{rem:misspecified:l}
When there are multiple possible choices for $l$, Theorem \ref{normality} and Theorem \ref{power} may fail if $l$
is misspecified. For example, if $m=4$ and
the ``correct" value is $l_0=2$ (corresponding to the true hyperedge probability), but we count $1$-cycle. Then under $H_0$, the test statistic in (\ref{normality1}) or (\ref{normality2}) is of order $O_p(n^{\frac{3}{2}})$, 
i.e., the limiting distribution does not exist. 
Whereas, if the correct value is $l_0=1$ but we count $2$-cycle, then the test statistic in (\ref{normality1}) or (\ref{normality2}) have the same limiting distribution
(if it exists) under $H_0$ and $H_1$, i.e., the power of the test does not approach one.
In practice, we recommend using the hyperedge proportion to get a rough estimate for $l$.
\end{Remark}

Theorem \ref{normality} and Theorem \ref{power} work for relatively sparse hypergraphs. For denser hypergraphs, we propose a sparsification procedure so that Theorem \ref{normality} and Theorem \ref{power} are valid. 
For any index $i_1<i_2<\dots<i_m$, generate $\epsilon_{i_1i_2\dots i_m}\overset{iid}{\sim} Bernoulli(r_n)$. 
Consider a new hypergraph with adjacency tensor $\tilde{A}$ defined by $\tilde{A}_{i_1i_2\dots i_m}=\epsilon_{i_1i_2\dots i_m}A_{i_1i_2\dots i_m}$, where $A_{i_1i_2\dots i_m}$ are the elements of the original observed adjacency tensor. 
Under $H_0^{\prime}$, we have
\[\mathbb{E}[\tilde{A}_{i_1i_2\dots i_m}]=(\mathbb{E}W_1)^m\frac{(r_na_n)+(k^{m-1}-1)(r_nb_n)}{k^{m-1}n^{m-1}}.\]
Set $\tilde{a}_n=r_na_n$ and $\tilde{b}_n=r_nb_n$. For dense hypergraphs, we could replace $A$, $a_n$ and $b_n$ in (\ref{ht1}) by $\tilde{A}$, $\tilde{a}_n$ and $\tilde{b}_n$ respectively. Note that the hypergraphs $\tilde{A}$ and $A$ have the same global community structure. A properly selected $r_n$ will make Theorem \ref{normality} and Theorem \ref{power} valid.

\begin{Corollary}\label{cor:verygood}
Suppose $\mathbb{E}W_1^4=O(1)$ and $1\ll a_n\asymp b_n \leq n^{m-1}$. If $r_n=o(1)$ and $n^{l-1}\ll r_na_n\asymp r_nb_n \ll n^{l-\frac{2}{3}}$ for some integer $1\leq l\leq \frac{m}{2}$. Then the results of Theorems \ref{normality} and \ref{power} based on $l$-cycle
continue to hold based on the sparsified hypergraph $\tilde{A}$.
\end{Corollary}

Note that Corollary \ref{cor:verygood} is valid for a broad range of hyperedge probabilities
$\frac{1}{n^{m-1}}\ll p_n\asymp q_n\le 1$. Since $H_0$ and $H_1$ are indistinguishable when $p_n\asymp q_n\ll\frac{1}{n^{m-1}}$ (see Section \ref{sec:contiguity}),
it covers all density regimes of interest.
One just needs to select the sparsification factor $r_n$ to ensure that $r_na_n$ and $r_nb_n$ fall into the range
$n^{l-1}\ll r_na_n\asymp r_nb_n \ll n^{l-\frac{2}{3}}$, provided that one wants to use $l$-cycles to construct the
test. The selection of $l$ has been discussed in Remark \ref{rem:misspecified:l}.

\begin{Remark}
In some literature, the degree correction variable $W_i$ in (\ref{degree:correct:sbm}) are assumed to be deterministic \cite{KN11,GMZZ, KSX20}. In this case, Theorem \ref{normality} still holds under mild conditions and the proof goes through with slight modifications. To illustrate this, we consider $m=3$. Let $W=(W_1,\dots,W_n)$ be a given and deterministic degree correction vector and denote $||W||_t^t=\sum_{i=1}^nW_i^t$ for positive integer $t$. Let $\widehat{T}_1$, $\widehat{E}_1$,  $\widehat{V}_1$ be defined as 
\[\widehat{T}_1=\frac{\sum_{i_1,\dots,i_6:distinct}A_{i_1i_2i_3}A_{i_3i_4i_5}A_{i_5i_6i_1}}{n^6},\ \widehat{V}_1=\frac{\sum_{i_1,\dots,i_5:distinct}A_{i_1i_2i_3}A_{i_3i_4i_5}}{n^5},\ \widehat{E}_1=\frac{\sum_{i_1,i_2,i_3:distinct}A_{i_1i_2i_3}}{n^3}.\]

Then we have the following result.

\begin{Proposition}\label{fixeddegree}
Suppose $1\ll ||W||_t^t=O( ||W||_1)$ for $2\leq t\leq 12$, $||W||_1\asymp ||W||_2^2=O(n)$, $p_0||W||_1^2\gg1$ and $p_0^2||W||_1^3=o(1)$. Then under $H_0^{\prime}$ we have
 \begin{equation}\label{fixnormality1}
 \widehat{\mathcal{T}}_3=\sqrt{\frac{n^6}{\widehat{T}_1}} \Bigg[\widehat{T}_1-\Big(\frac{\widehat{V}_1}{\widehat{E}_1}\Big)^{3}\Bigg] \overset{d}{\to} N(0,1).
 \end{equation}
Further, if $1\ll a_n\asymp b_n \ll n^{\frac{1}{3}}$, then the power of the test $\widehat{\mathcal{T}}_3$ goes to 1 as $\delta_1\rightarrow\infty$, where
\begin{equation*}\label{fixlimit:ez:test}
\delta_1:=\sqrt{\frac{n^6}{T_1}} \Bigg[T_1-\Big(\frac{V_1}{E_1}\Big)^{3}\Bigg],\hskip 1cm E_1=\frac{a_n+(k^{2}-1)b_n}{n^{2}k^{2}}\frac{||W||_1^3}{n^3},
\end{equation*}
 \begin{eqnarray*}
 V_1&=&\Bigg(\frac{(a_n-b_n)^2}{n^{4}k^{4}}+\frac{2(a_n-b_n)b_n}{n^{4}k^{2}}+\frac{b_n^2}{n^{4}}\Bigg)\frac{||W||_2^2||W||_1^4}{n^5},\\
 T_1&=&\Bigg(\frac{(a_n-b_n)^3}{n^{6}k^{5}}+\frac{3(a_n-b_n)^2b_n}{n^{6}k^{4}}+\frac{3(a_n-b_n)b_n^2}{n^{6}k^{2}}+\frac{b_n^3}{n^{6}}\Bigg)\frac{||W||_2^6||W||_1^3}{n^6}.
 \end{eqnarray*}
\end{Proposition}

The proof of Proposition \ref{fixeddegree} is given in the supplement. 
In Proposition \ref{fixeddegree},  the conditions $p_0||W||_1^2\gg1$ and $p_0^2||W||_1^3=o(1)$ require the hypergraph to be moderately sparse. At first glance, the conditions $1\ll ||W||_t^t=O( ||W||_1)=O(n)$ for $2\leq t\leq 12$ and $||W||_1\asymp||W||_2^2$ seem very restrictive. However, these conditions are easy to satisfy and can accommodate severe degree heterogeneity. For example, when $W_i=\frac{i}{n}$ for $i=1,2,\dots,n$, we have $||W||_t^t=\frac{n}{t+1}(1+o(1))$ for any positive integer $t$. In this case, the average degrees $d_1$ and $d_n$ for vertices 1 and $n$ are
\[d_1=\sum_{1<j<k}W_1W_jW_kp_0=p_0W_1\Big(0.5\big(\sum_{j=2}^nW_j\big)^2-\sum_{j=2}^nW_j^2\Big)=\frac{np_0}{8}(1+o(1)),\]
\[d_n=\sum_{j<k<n}W_nW_jW_kp_0=p_0W_n\Big(0.5\big(\sum_{j=1}^{n-1}W_j\big)^2-\sum_{j=1}^{n-1}W_j^2\Big)=\frac{n^2p_0}{8}(1+o(1)).\]
Clearly, $d_n\asymp nd_1$ and hence the hypergraph is highly heterogeneous. Another example is to take $W_i\in [c_1,c_2]$ with positive constants $c_1<c_2$, which yields a hypergraph with less heterogeneous degrees.
\end{Remark}

\section{Extensions to Non-uniform Hypergraph.}\label{sec:non:uniform}
Non-uniform hypergraph can be viewed as a superposition of a collection of uniform hypergraphs, introduced by \cite{GhDu}
in which the authors proposed a spectral algorithm for community detection. 
In this section, we study the problem of testing community structure over
a non-uniform hypergraph. 

Let $\mathcal{H}^k(n,M)$ be a non-uniform hypergraph over $n$ vertices, with the vertices uniformly
and independently partitioned into $k$ communities, and
$M\ge2$ is an integer representing the maximum length of the hyperedges.
Following \cite{GhDu}, we can write $\mathcal{H}^k(n, M)=\cup_{m=2}^M \mathcal{H}_m^k\left(n,\frac{a_{mn}}{n^{m-1}},\frac{b_{mn}}{n^{m-1}}\right)$, where 
$\mathcal{H}_m^k\left(n,\frac{a_{mn}}{n^{m-1}},\frac{b_{mn}}{n^{m-1}}\right)$ 
are independent uniform hypergraphs with degree-corrected vertices introduced in Section \ref{EZuniform}. 
Correspondingly, define $\mathcal{H}(n, M)=\cup_{m=2}^M\mathcal{H}_m\Big(n,\frac{a_{mn}+(k^{m-1}-1)b_{mn}}{k^{m-1}n^{m-1}}\Big)$ as a superposition of Erd\"{o}s-R\'{e}nyi models.
Clearly, each Erd\"{o}s-R\'{e}nyi model in $\mathcal{H}(n,M)$ has 
the same average degree as its counterpart in $\mathcal{H}^k(n,M)$,
and $\mathcal{H}(n,M)$ has no community structure.
Let $A_{m}$ denote the adjacency tensor for $m$-uniform sub-hypergraph and $A=\{A_m, m=2,\ldots,M\}$
is a collection of $A_m$'s. We are interested in the following hypotheses:
\begin{eqnarray}\label{htnon}
H_0'': A\sim \mathcal{H}(n, M)\ vs.\ H_1'': A\sim\mathcal{H}^k(n,M).
\end{eqnarray}

\subsection{ Non-uniform homogeneous hypergraphs with bounded degree.}

To enhance readability, we assume $M=3$, i.e., $\mathcal{H}=\mathcal{H}_2\cup\mathcal{H}_3$, and the hypergraphs are homogeneous without degree correction. The results are extendable to arbitrary $M$ with more tedious arguments. 
The following Corollary \ref{boundedcor1}, extending Theorem \ref{contiguous}, shows that 
it is impossible to distinguish $H_0^{\prime\prime}$ and $H_1^{\prime\prime}$ in extremely sparse regime.
The proof is essentially the same as Theorem \ref{contiguous} which also relies on the conditional independence of $\mathcal{H}_2$ and $\mathcal{H}_3$. 

\begin{Corollary}\label{boundedcor1}
If $a_{mn}\asymp b_{mn}=o(1)$, then
$H_0^{\prime\prime}$ and $H_1^{\prime\prime}$ are mutually contiguous.
\end{Corollary}

The following Corollary \ref{boundedcor2} extends the
bounded degree results from Section \ref{sec:bdd:dgree}. Let $a_{mn}=a_m$, $b_{mn}=b_m$ be positive constants,
and $\kappa_m=\frac{(a_m-b_m)^2}{k^{m-1}[a_m+(k^{m-1}-1)b_m]}$.

\begin{Corollary}\label{boundedcor2}
If $\kappa_2>1$ or $\kappa_3>1$, then $H_0^{\prime\prime}$ and $H_1^{\prime\prime}$ are asymptotically orthogonal.
If
\[
\left[\kappa_2+\frac{\kappa_3}{3}\left(1+\frac{1}{3k^2}\right)\right](k^2-1)<1,
\]
then $H_0^{\prime\prime}$ and $H_1^{\prime\prime}$ are mutually contiguous.
Furthermore, the results of Theorems \ref{normaltest} and \ref{consistent:a:b}
still hold with the corresponding quantities therein replaced by those in $\mathcal{H}_m$.
\end{Corollary}

\subsection{ Non-uniform hypergraph with growing degree.}

Assume that, for $2\le m\le M$, $a_{mn},b_{mn}$ are proxies of the hyperedge densities 
satisfying $n^{l_m-1}\ll a_{mn}\asymp b_{mn} \ll n^{l_m-\frac{2}{3}}$, 
for some integer $1\leq l_m\leq \frac{m}{2}$. 

For any $2\le m\le M$, let $\widehat{\mathcal{T}}_m$ and
$\delta_m$ be defined as in (\ref{ez:test:unif}) and (\ref{limit:ez:test}), respectively,
based on the $m$-uniform sub-hypergraph.
We define a test statistic for (\ref{htnon}) as 
\begin{equation}\label{new:test:nonuniform}
\widehat{\mathcal{T}}=\sum_{m=2}^M c_m\widehat{\mathcal{T}}_m,
\end{equation} 
where $c_m$ are constants with normalization $\sum_{m=2}^Mc_m^2=1$. 
As a simple consequence of Theorems \ref{normality} and \ref{power}, we get the 
asymptotic distribution of $\widehat{\mathcal{T}}$ as follows.

\begin{Corollary}\label{nonnormality}
Suppose that the degree-correction weights satisfy the same conditions as in Theorem \ref{normality},
and for any $2\le m\le M$, $n^{l_m-1}\ll a_{mn}\asymp b_{mn} \ll n^{l_m-\frac{2}{3}}$, 
for some integer $1\leq l_m\leq \frac{m}{2}$. Then, as $n\to\infty$,
$\widehat{\mathcal{T}}-\sum_{m=2}^Mc_m\delta_m\overset{d}{\to}N(0,1)$.
Furthermore, for any constant $C>0$, under $H_1''$, 
$\mathbb{P}(|\widehat{\mathcal{T}}|>C)\rightarrow1$,
provided that $\sum_{m=2}^Mc_m\delta_m\rightarrow\infty$ as $n\to\infty$.
\end{Corollary}
Under $H_0''$, i.e., each $m$-uniform subhypergraph has no community structure,
we have $\delta_m=0$ by Proposition \ref{proph1}.
Corollary \ref{nonnormality} says that 
$\widehat{\mathcal{T}}$ is asymptotically standard normal.
Hence, an asymptotic testing rule at significance $\alpha$ would be 
\[
\textrm{reject $H_0''$ if and only if $|\widehat{\mathcal{T}}|>z_{\alpha/2}$}.
\]
The quantity $\sum_{m=2}^Mc_m\delta_m$ may represent the degree of separation between $H_0''$ and $H_1''$.
By Corollary \ref{nonnormality}, under $H_1''$,
the test will achieve high power when $\sum_{m=2}^Mc_m\delta_m$ is large. 

\begin{Remark}\label{new:test:rem}
According to Corollary \ref{nonnormality}, to make $\widehat{\mathcal{T}}$ having the largest power, we need to maximize the value
of $\sum_{m=2}^M c_m\delta_m$ subject to $\sum_{m=2}^M c_m^2=1$.
The maximizer is $c^*_m=\frac{\delta_m}{\sqrt{\sum_{m=2}^M\delta_m^2}}$, $m=2,3,\dots, M$. 
The corresponding test $\widehat{\mathcal{T}}^*=\sum_{m=2}^M c_m^*\widehat{\mathcal{T}}_m$
becomes asymptotically the most powerful among (\ref{new:test:nonuniform}).
In particular, $\widehat{\mathcal{T}}^*$ is more powerful
than $\widehat{\mathcal{T}}_m$ for a single $m$.
This can be explained by the additional hyperedge information involved in the test.
This intuition is further confirmed by numerical studies in Section \ref{sec:sim:real}.
Note that $\widehat{\mathcal{T}}_2$ (m=2) is the classic test proposed by \cite{GL17a}
in ordinary graph settings. 
\end{Remark}

\section{Numerical Studies.}\label{sec:sim:real}

In this section, we provide 
a simulation study in Section \ref{sec:sim:study} 
and real data analysis in Section \ref{sec:real:data} to assess the finite sample performance of our tests.

 \subsection{Simulation.}\label{sec:sim:study}
We generated a non-uniform hypergraph $\mathcal{H}^2(n, 3)=\mathcal{H}_2^2(n, a_2, b_2) \cup \mathcal{H}_3^2(n, a_3, b_3)$,
with $n=100$ under various choices of $\{(a_m, b_m), m=2, 3\}$. 
In each scenario, we calculated $Z_2:=\widehat{\mathcal{T}}_2'$ and $Z_3:=\widehat{\mathcal{T}}_3'$ by (\ref{ez:test:unif:prime}). 
Note that $Z_2=\widehat{\mathcal{T}}_2'$ is the test for ordinary graph considered in \cite{GL17a}.
For testing the community structure on the non-uniform hypergraph, we calculated the statistic 
$Z:=\widehat{\mathcal{T}}=(\widehat{\mathcal{T}}_2'+\widehat{\mathcal{T}}_3')/\sqrt{2}$. {In addition, we considered a strategy similar to  \cite{ALS18} by first reducing the hypergraph to a weighted graph and applying a test designed for weighted graphs in \cite{tokuda2018statistical}. Specifically, given an  $m$-uniform hypergraph with hyperedges $e_1, e_2,\ldots, e_M$, we first transformed it to a weighted graph with an adjacency matrix $\widetilde{A}=[\widetilde{A}_{ij}]_{1\le i,j\le n}$ in which $\widetilde{A}_{ij}=\sum_{k=1}^MI(\{i,j\}\subset e_k)$ for $i\neq j$ and $\widetilde{A}_{ij}=0$ for $i=j$. In other words, $\widetilde{A}_{ij}$ is the total number of hyperedges containing vertices $i$ and $j$. Next, we generated a new weighted graph with an adjacency matrix $A=[{A}_{ij}]_{1\le i,j\le n}$ by zeroing out row $s$ and column $s$ of $\widetilde{A}$ if  $\sum_{j=1}^n \widetilde{A}_{sj}>c_{thr}\frac{1}{n}\sum_{i=1}^n\sum_{j=1}^n \widetilde{A}_{ij}$. Here $c_{thr}>0$ is a prespecified  threshold constant.\footnote{According to the proof of Lemmas 5 and 7 in \cite{ALS18},  $c_{thr}$ is a large enough constant such that  $\log(1+m^2 x)-m^2\geq \frac{1}{2}\log(x)$ for all $x\geq c_{thr}$. In the simulation studies, we chose $c_{thr}$ to be the largest root of $\log(1+m^2 x)-m^2= \frac{1}{2}\log(x)$} We then applied the test method proposed by
\cite{tokuda2018statistical} to the weighted graph $A$, where the test statistic is denoted by $Z_T$. }

We examined the size and power of each test by calculating the rejection proportions based on 500 independent replications at $5\%$ significance level. 
Let $\delta_m$ denote the quantity defined in (\ref{limit:ez:test})
which is the main factor that affects power.

Our study consists of two parts.
In the first part, we investigated the power change of the four testing procedures when $\delta_2=\delta_3=\delta$ increases from 0 to 10.  Specifically, we set $b_2=10b_3$, where $b_3=0.01, 0.005, 0.001$ represents 
the dense,  moderately dense and sparse network, respectively; $a_m=r_mb_m$ for $m=2,3$ with the values of $r_m$ summarized in Table \ref{table:choice:rm:fix:b3}.
It can be checked that such choice of $(a_m, b_m)$ indeed makes $\delta$ range from $0$ to $10$.  We also considered both balanced and imbalanced networks with the probability ($\varsigma$) of the smaller community  takes the value of 0.5 and 0.3, respectively. 

The rejection proportions under various settings are summarized in Figures \ref{figure:dense:network} through \ref{figure:sparse:network}. 
Several interesting findings should be discussed. First, the rejection proportions of all test statistics except the $Z_T$ (based on the graph transformation) at $\delta=0$ are close to the nominal level $0.05$ under different choices of $\varsigma$ and $b_3$, which demonstrates that these three test statistics are valid. {We observe that the size (corresponding to $\delta=0$) and power (corresponding to $\delta>0$)  of the graph-transformation test are almost $100\%$ regardless of the choice of $b_3$, 
which implies that the testing procedure $Z_T$ is asymptotically invalid.}
Second, as expected, the rejection proportions of all tests increase with $\delta$, regardless of the choices of $b_3$ and $\varsigma$. Third, in most cases, the testing procedure based on non-uniform hypergraph ($Z$) has larger power than the one based only on the $3$-uniform hypergraph ($Z_3$) or the ordinary graph ($Z_2$). 
This agrees with our theoretical finding since more information has been used in the combined test; see Remark \ref{new:test:rem} for a detailed explanation. 

\begin{Remark}
The failure of the graph-transformation-based testing procedure $Z_T$ is possibly due to the dependence between the edges of the transformed graph. {Given the number of communities $k$, many existing community detection algorithms do not require the independence assumption about the edges. However, this assumption is important to derive the limit distributions of the corresponding statistics in the hypothesis testing problems about $k$ (e.g., see \cite{BS16, GL17a, L16, tokuda2018statistical}).}
 The graph-transformation-based method might still be promising for testing hypergraphs, but new asymptotic theory based on dependent edges seems necessary.	
\end{Remark}

%
\begin{table}[H]
\centering
\resizebox{15cm}{!}{
\begin{tabular}{ccccccccccccc}
\hline\hline
$b_3$                   & $\delta$ & 0 & 1    & 2    & 3     & 4     & 5     & 6     & 7     & 8     & 9     & 10    \\
\hline
\multirow{2}{*}{0.01}  & $r_3$  & 1 & 2.26 & 2.65 & 2.93  & 3.17  & 3.38  & 3.58  & 3.75  & 3.91  & 4.06  & 4.21  \\
                       & $r_2$  & 1 & 2.07 & 2.43 & 2.71  & 2.95  & 3.16  & 3.35  & 3.53  & 3.71  & 3.87  & 4.02  \\
                       \hline
\multirow{2}{*}{0.005} & $r_3$  & 1 & 2.89 & 3.51 & 3.98  & 4.39  & 4.75  & 5.08  & 5.38  & 5.67  & 5.94  & 6.20  \\
                       & $r_2$  & 1 & 2.66 & 3.29 & 3.79  & 4.22  & 4.61  & 4.97  & 5.31  & 5.64  & 5.94  & 6.24  \\
                       \hline
\multirow{2}{*}{0.001} & $r_3$  & 1 & 6.50 & 8.83 & 10.73 & 12.41 & 13.95 & 15.39 & 16.76 & 18.03 & 19.28 & 20.48 \\
                       & $r_2$  & 1 & 6.57 & 9.31 & 11.59 & 13.64 & 15.51 & 17.26 & 18.92 & 20.51 & 22.00 & 23.46\\
                       \hline
\end{tabular}
}
\caption{\it\footnotesize Choices of $r_2,r_3,b_3$ for $\delta$ to range from 1 to 10.}
\label{table:choice:rm:fix:b3}
\end{table}

\begin{figure}[H]
\centering
\includegraphics[width=2.4 in, height=2.4 in]{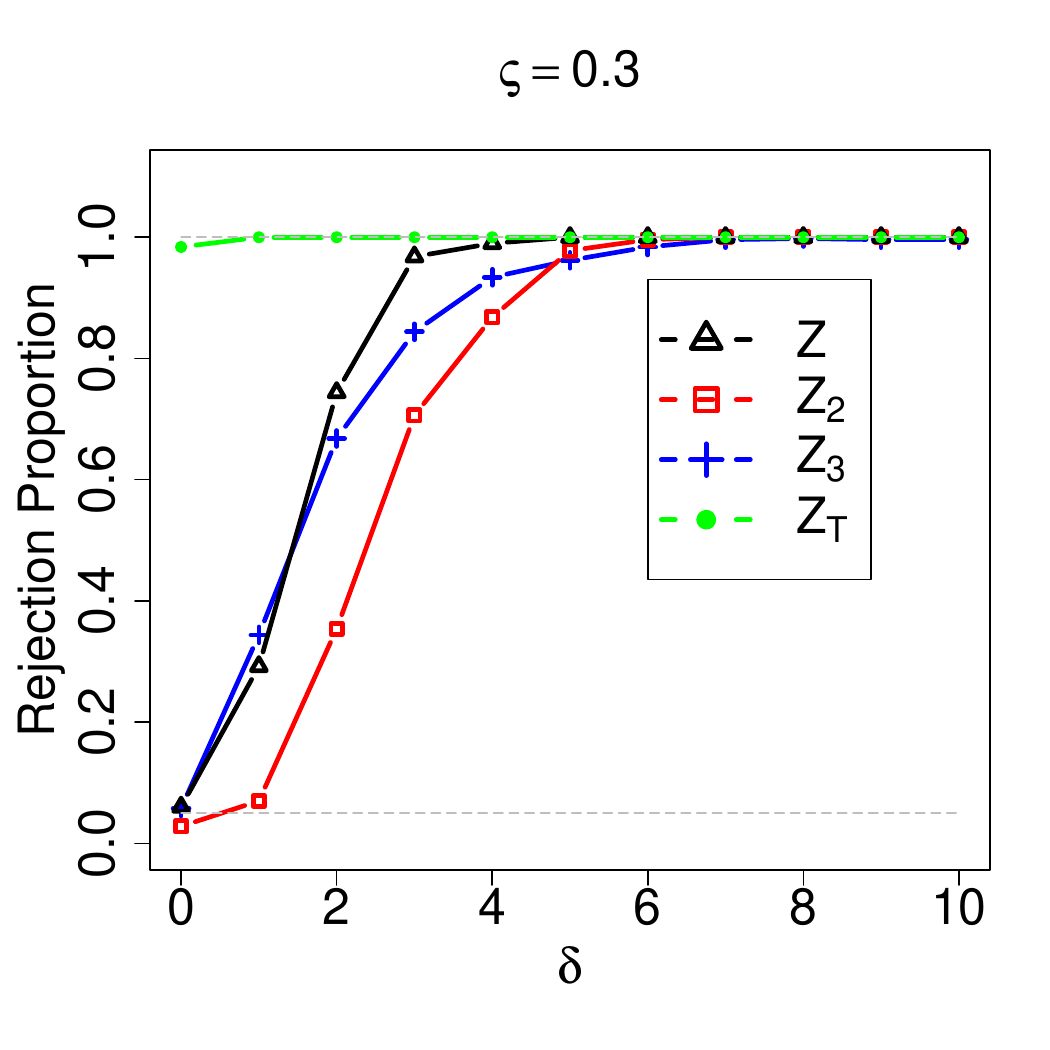}
\includegraphics[width=2.4 in, height=2.4 in]{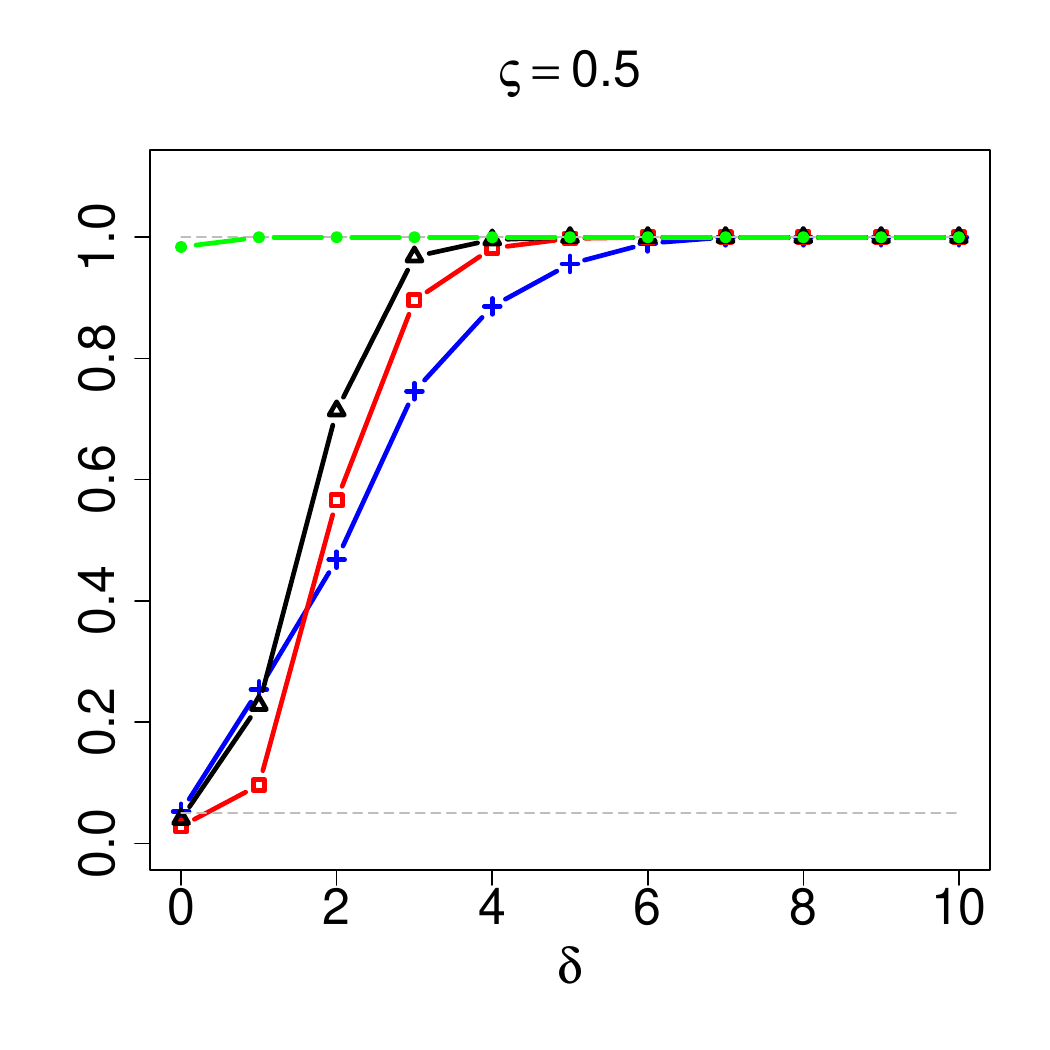}
\caption{\it\footnotesize Rejection proportions in dense case with $b_3=0.1\times b_2=0.01$.}
\label{figure:dense:network}
\end{figure}

\begin{figure}[H]
\centering
\includegraphics[width=2.4 in, height=2.4 in]{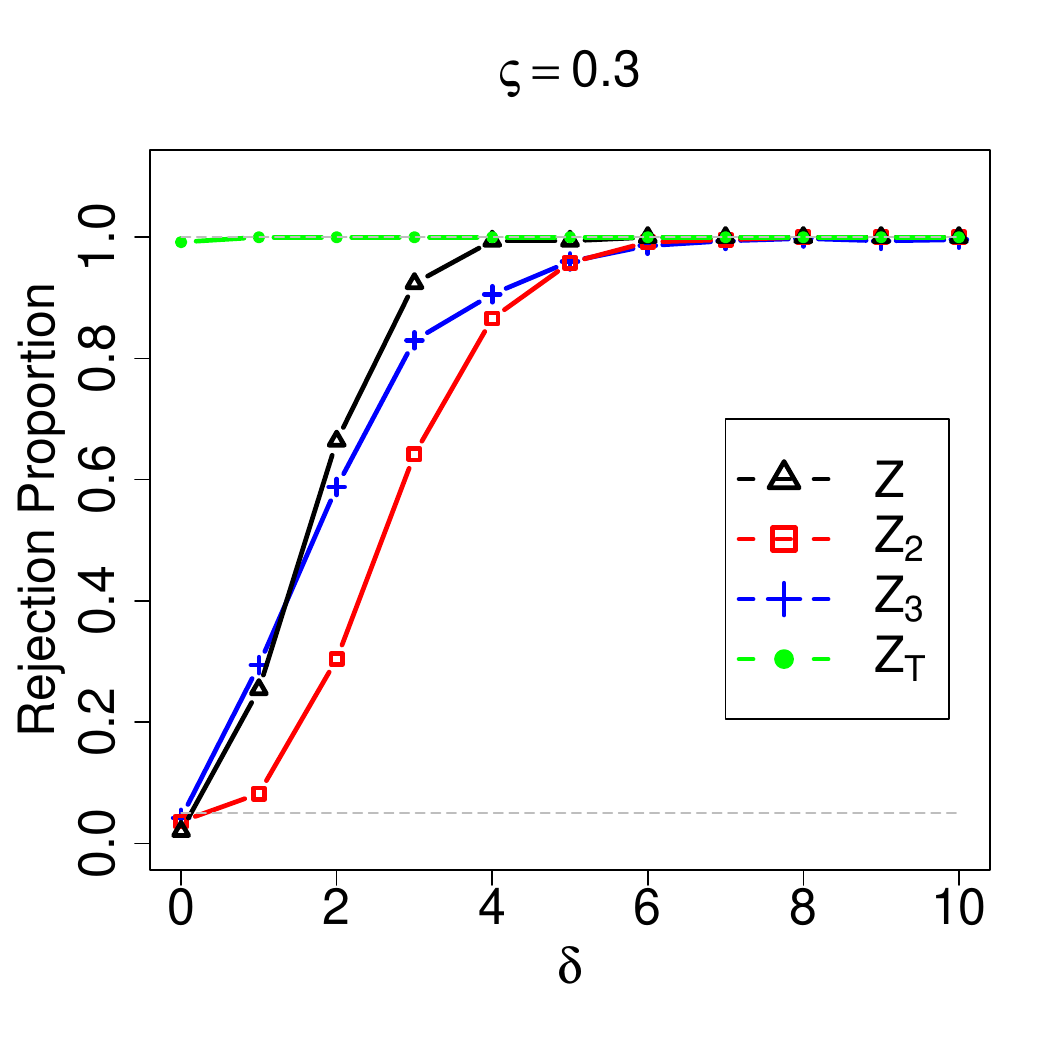}
\includegraphics[width=2.4 in, height=2.4 in]{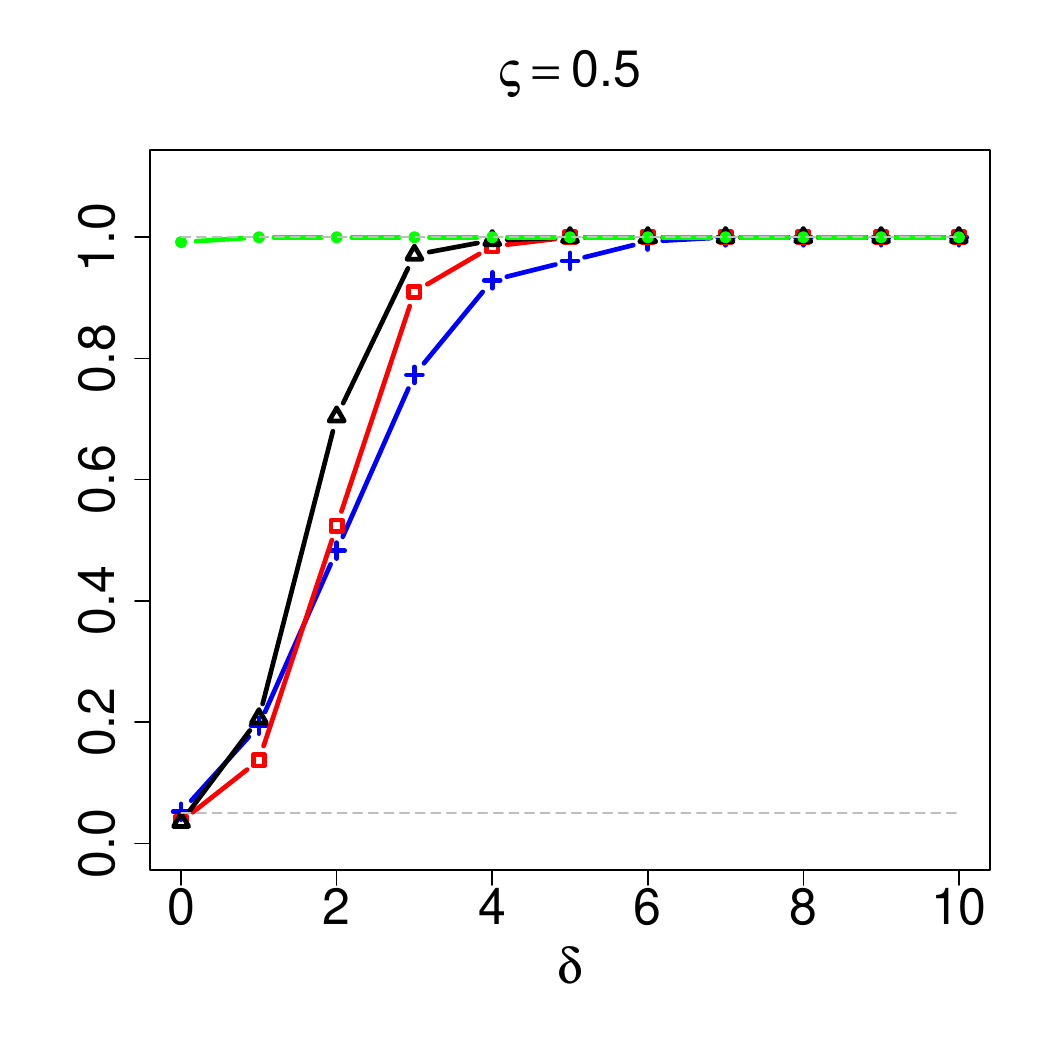}
\caption{\it\footnotesize Rejection proportions in moderately dense case with $b_3=0.1\times b_2=0.005$.}
\label{figure:moderately:dense:network}
\end{figure}

\begin{figure}[H]
\centering
\includegraphics[width=2.4 in, height=2.4 in]{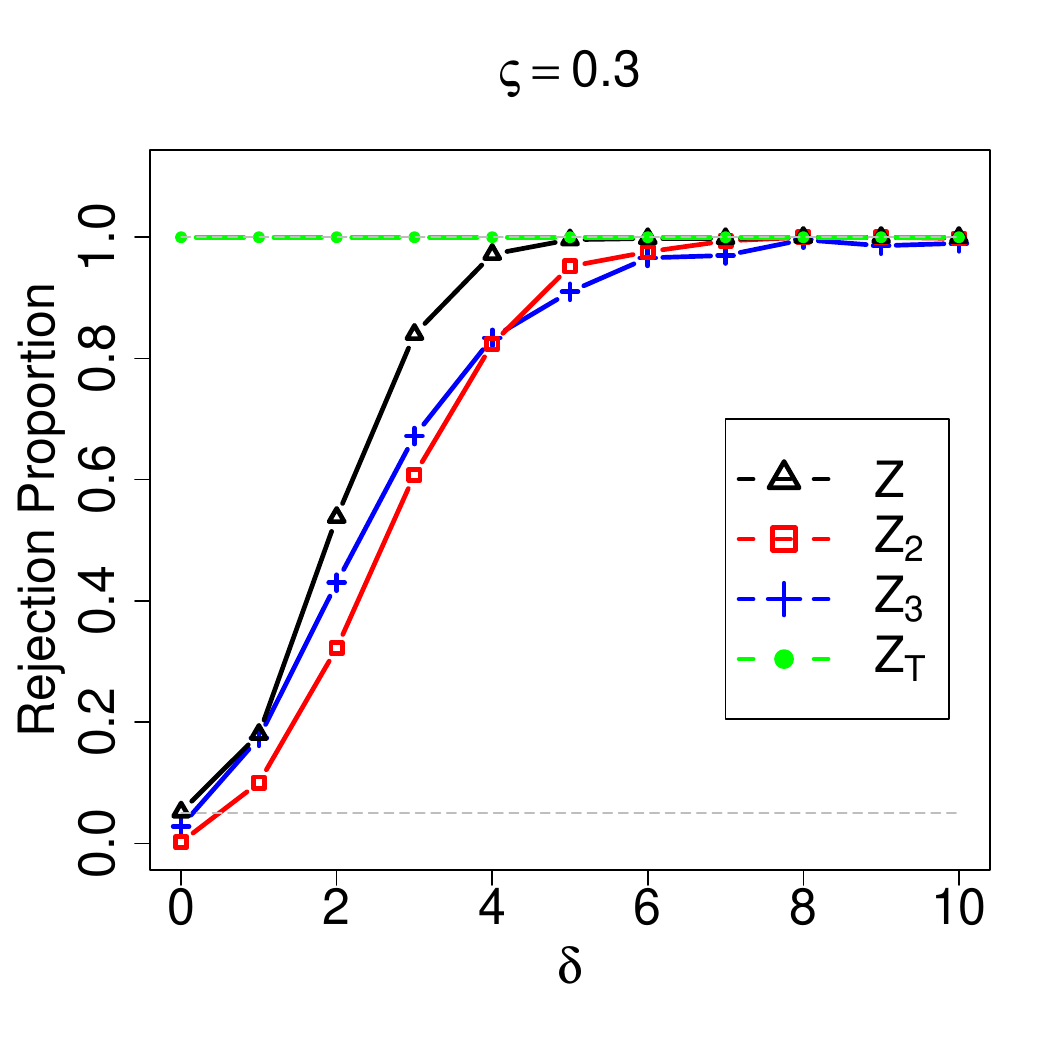}
\includegraphics[width=2.4 in, height=2.4 in]{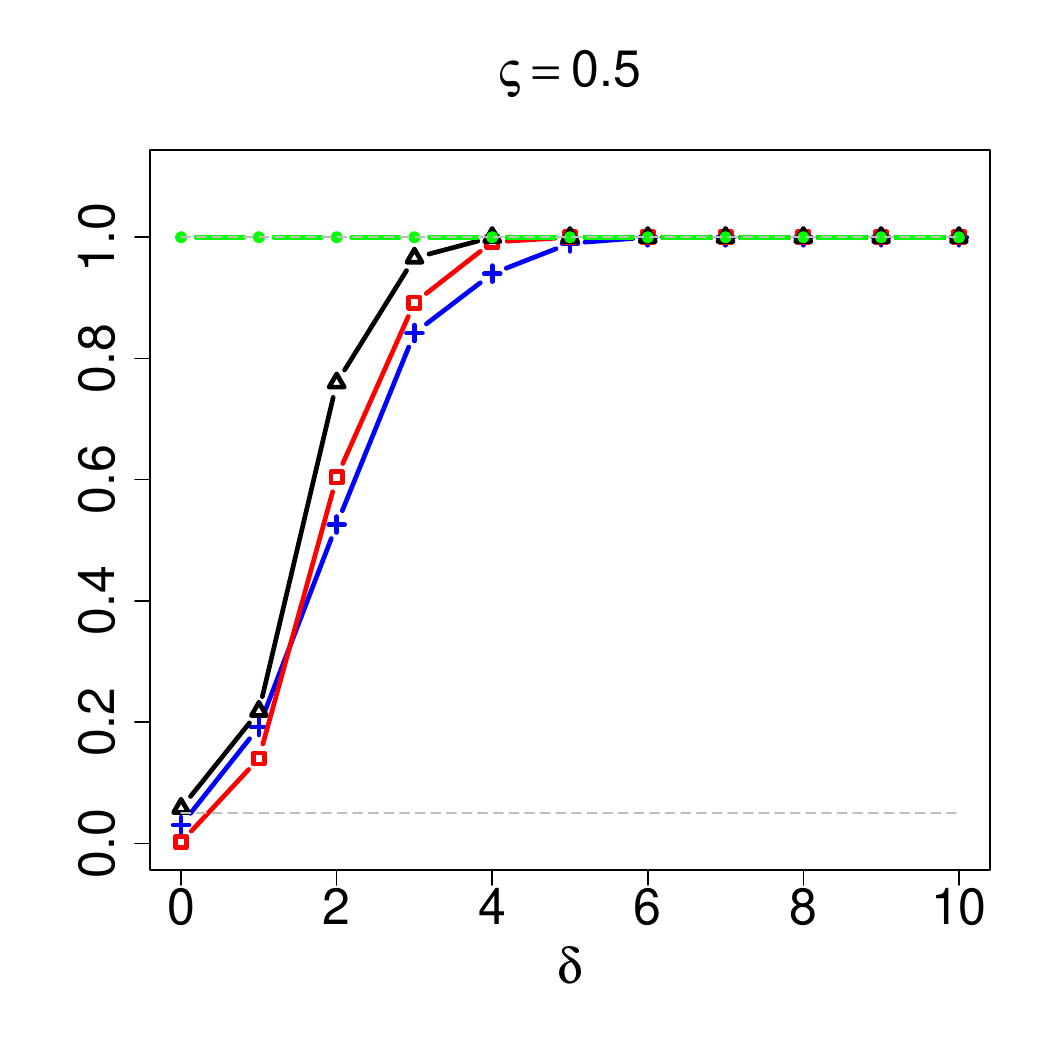}
\caption{\it\footnotesize Rejection proportions in sparse case with $b_3=0.1\times b_2=0.001$.}
\label{figure:sparse:network}
\end{figure}

In the second part, we investigated how the powers of the tests change along with the hyperedge probability. 
For convenience, we report the results based on the log-scale of $b_3$
which ranges from $-8$ to $-6$. We chose $\delta=1$ and $3$, $\varsigma=0.3$ and 0.5, $b_2=10b_3$. 
Similar to the first part, we set $a_m=r_mb_m$ with $m=2$ and $3$ to guarantee that $\log b_3$ indeed
ranges from $-8$ to $-6$. The values of $r_m$ were summarized in Table \ref{table:choice:rm:fix:delta}. 
Figures \ref{figure:fix:delta:1} and \ref{figure:fix:delta:3} report the rejection proportions for $\delta=1$ and $3$ under various hyperedge densities.   \textcolor{black}{We note that the rejection proportion of $Z_T$ is always $100\%$ under all settings.} Moreover, $Z$ is 
more powerful than $Z_2$ and $Z_3$ in the cases $\varsigma=0.3, 0.5$ and $\delta=3$.  \textcolor{black}{For the remaining scenarios, all procedures  have satisfactory performance. }
\begin{table}[H]
\caption{\it\footnotesize Choices of $r_2$, $r_3$, and $\delta$ for $\log(b_3)$ to range from $-8$ to $-6$.}
\label{table:choice:rm:fix:delta}
\centering
\resizebox{5.5cm}{!}{
\begin{tabular}{ccccc}
\hline \hline
$\delta$              & $\log(b_3)$ & -8    & -7    & -6     \\
\hline
\multirow{2}{*}{1} & $r_3$      & 14.18 & 6.88  & 3.93  \\
                   & $r_2$      & 15.78 & 7.03  & 3.72  \\
                   \hline
\multirow{2}{*}{3} & $r_3$      & 26.37 & 11.51 & 5.82  \\
                   & $r_2$      & 30.68 & 12.54 & 5.83 \\
                   \hline
\end{tabular}
}
\end{table}


\begin{figure}[H]
\centering
\includegraphics[width=2.4 in, height=2.4 in]{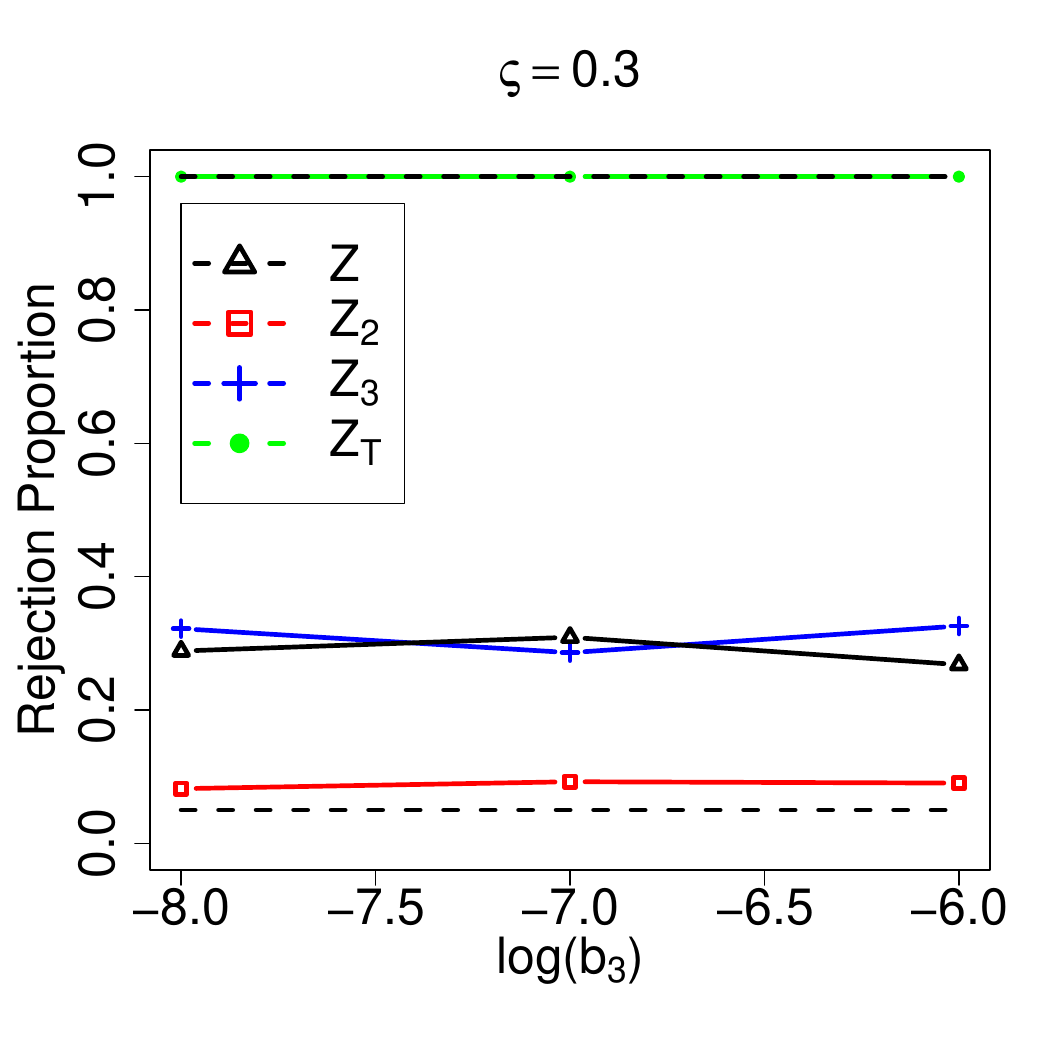}
\includegraphics[width=2.4 in, height=2.4 in]{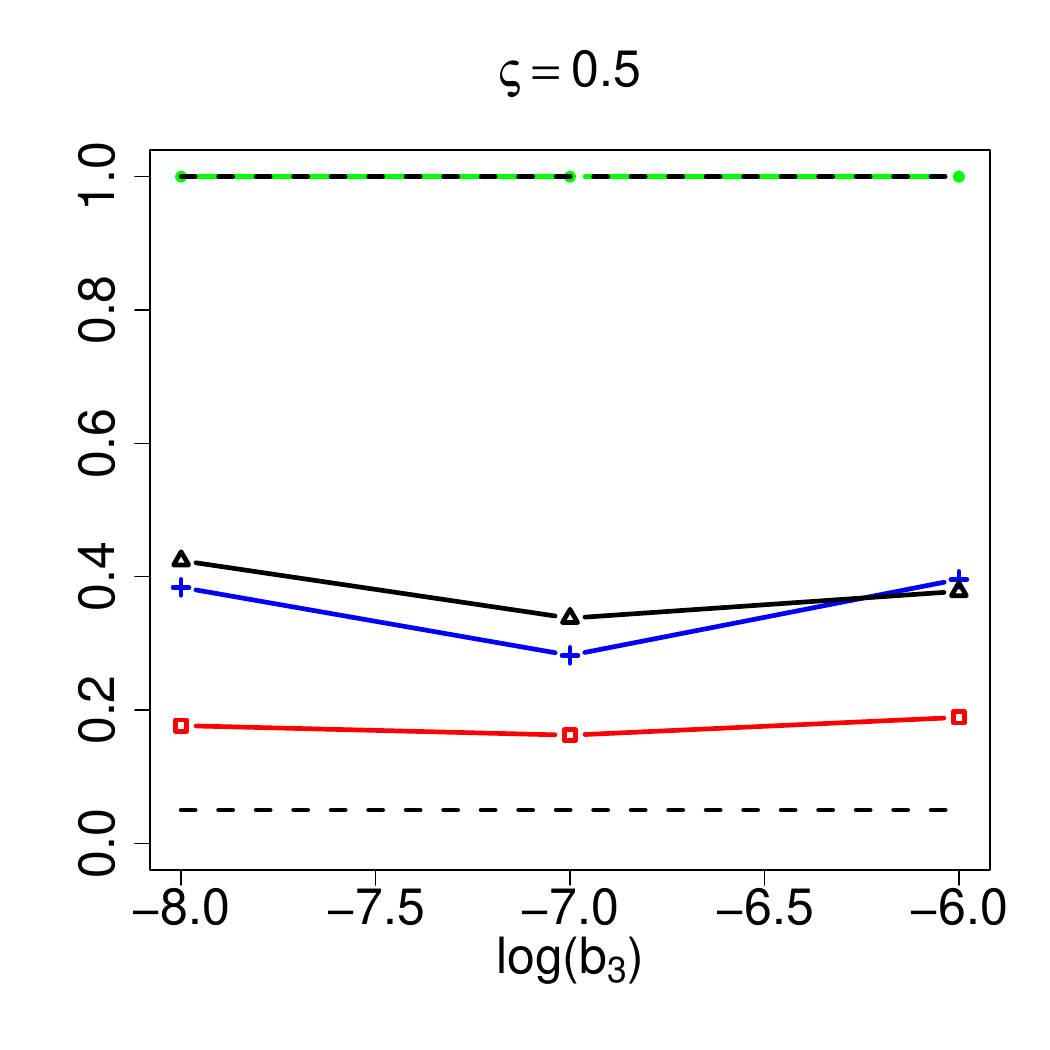}
\caption{\it\footnotesize Rejection proportions when $\delta=1$ and $b_2=10b_3$.}
\label{figure:fix:delta:1}
\end{figure}
\begin{figure}[H]
\centering
\includegraphics[width=2.4 in, height=2.4 in]{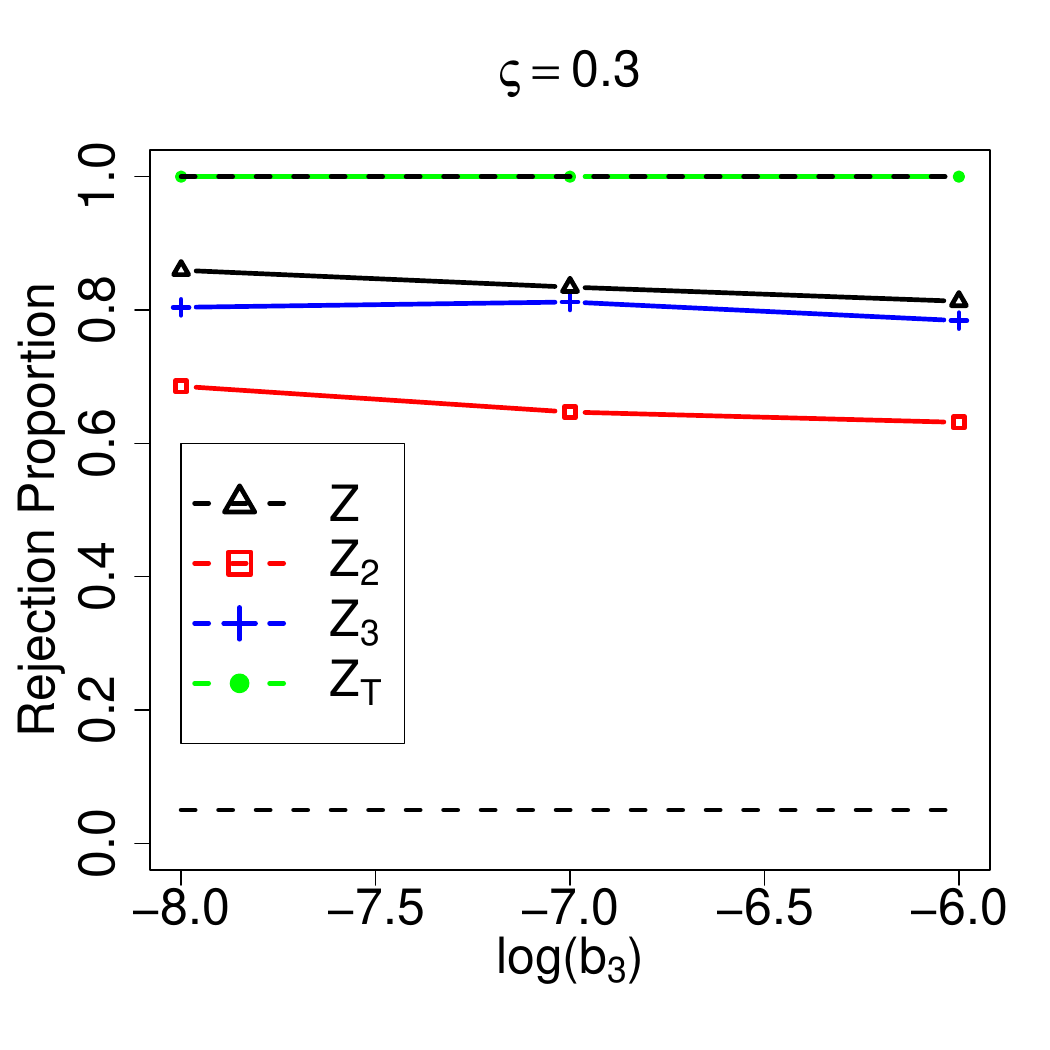}
\includegraphics[width=2.4 in, height=2.4 in]{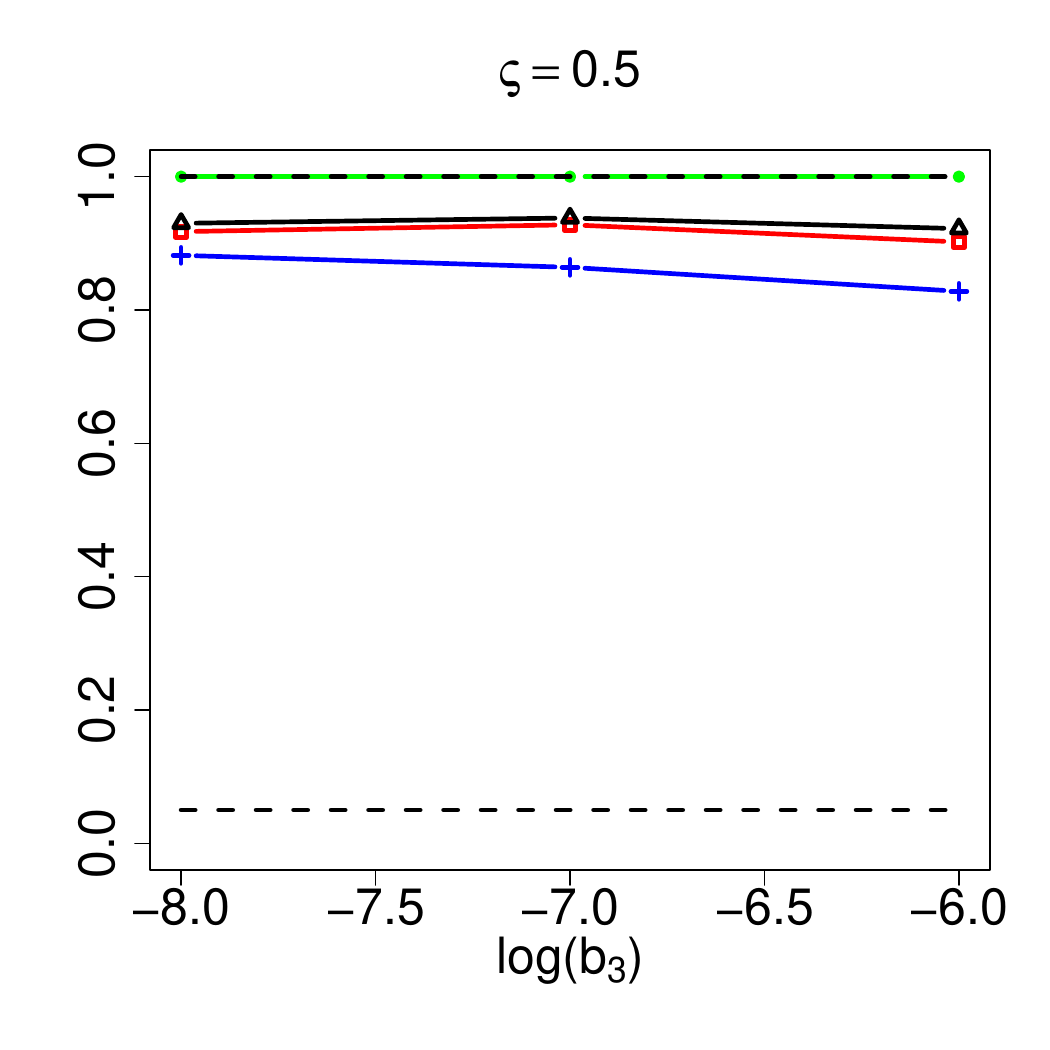}
\caption{\it\footnotesize Rejection proportions when $\delta=3$  and $b_2=10b_3$.}
\label{figure:fix:delta:3}
\end{figure}

 \subsection{Analysis of Coauthorship Data.}\label{sec:real:data}

In this section, we applied our testing procedure to study the community structure of a coauthorship network dataset, available at \url{https://static.aminer.org/lab-datasets/soinf/}. The dataset contains a 2-author ordinary graph and a 3-author hypergraph. After removing vertices with degrees less than ten or larger than 20, we obtained a hypergraph (hereinafter referred to as global network) with 58 nodes, 110 edges, and 40 hyperedges. The vertex-removal process aims to obtain a suitably sparse network so that 
our testing procedure is applicable.
We examined our procedures based on the global network and subnetworks.
To do this, we first performed the spectral algorithm proposed by \cite{GhDu} to partition the global network into four subnetworks
which consist of 7, 13, 14, 24 vertices, respectively
(see Figure \ref{figure:author:network}).
In Figure \ref{figure:incidence:matrix}, we plotted the incidence matrices of the 2- and 3-uniform hypergraphs,
denoted 2-UH and 3-UH, respectively,
as well as their superposition (Non-UH).
The black dots represent vertices within the same communities.
The red crosses represent vertices between different communities.
An edge or hyperedge is drawn between the black dots or red crosses that are vertically aligned.
It is observed that the between-community (hyper)edges
are sparser than the within-community ones,
indicating the validity of the partitioning.
 
We conducted testing procedures based on $Z_2$, $Z_3$, and $Z$ at significance level 0.05 (similar to Section \ref{sec:sim:study}) to both global network and subnetworks. 
The values of the test statistics are summarized in Table \ref{table:author:statistics}. 
Observe that $Z_2$ and $Z$ yield very large test values for the global network indicating strong rejection of the null hypothesis. 
For subnetwork testing, $Z_2$ rejects the null hypothesis for subnetwork 3; 
while $Z_3$ and $Z$ do not reject the null hypotheses for any subnetworks. This demonstrates that the community detection results are reasonable in general, and the subnetworks may no longer have finer community structures.  

\begin{figure}[h]
\centering
\includegraphics[width=1.7 in, height=1.5 in]{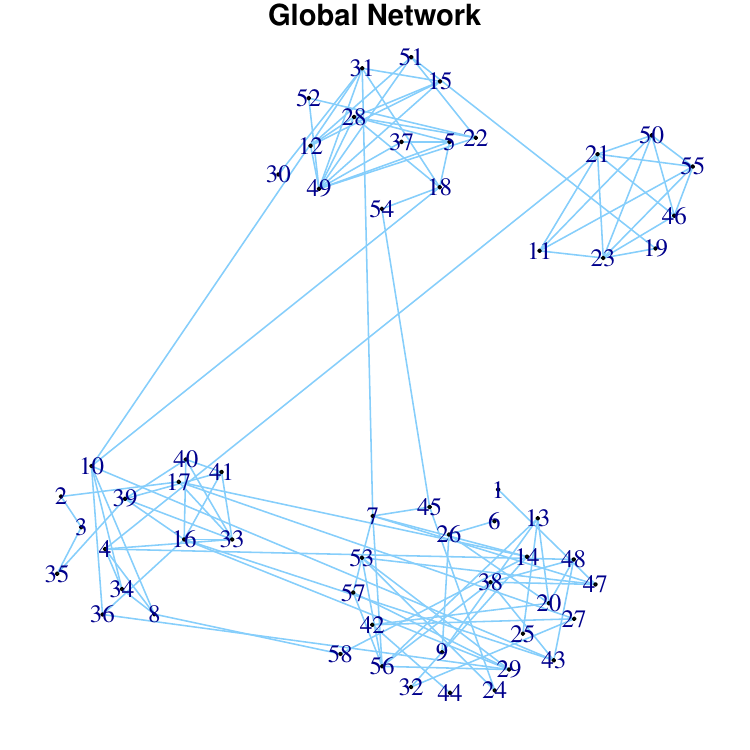}
\includegraphics[width=1.7 in, height=1.5 in]{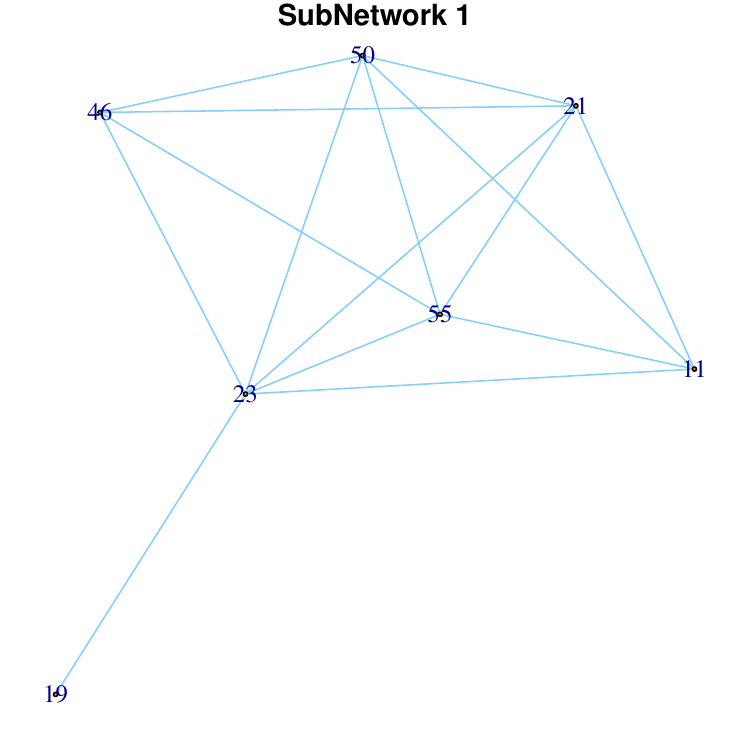}
\includegraphics[width=1.7 in, height=1.5 in]{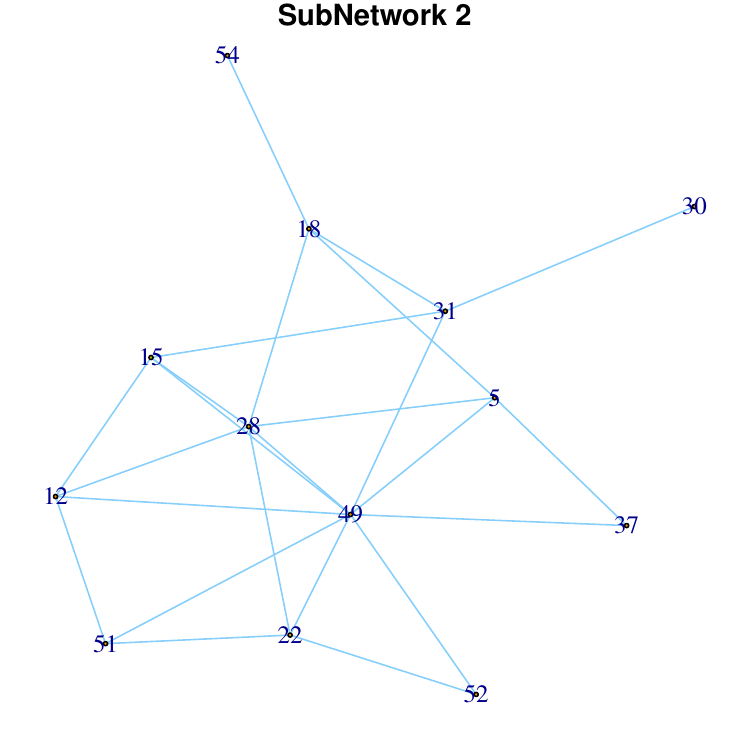}
\includegraphics[width=1.7 in, height=1.5 in]{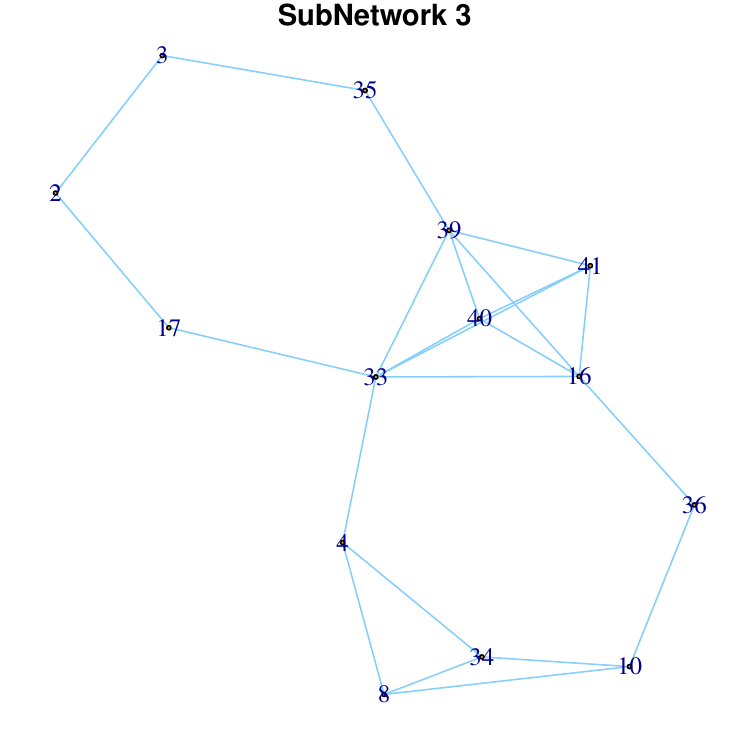}
\includegraphics[width=1.7 in, height=1.5 in]{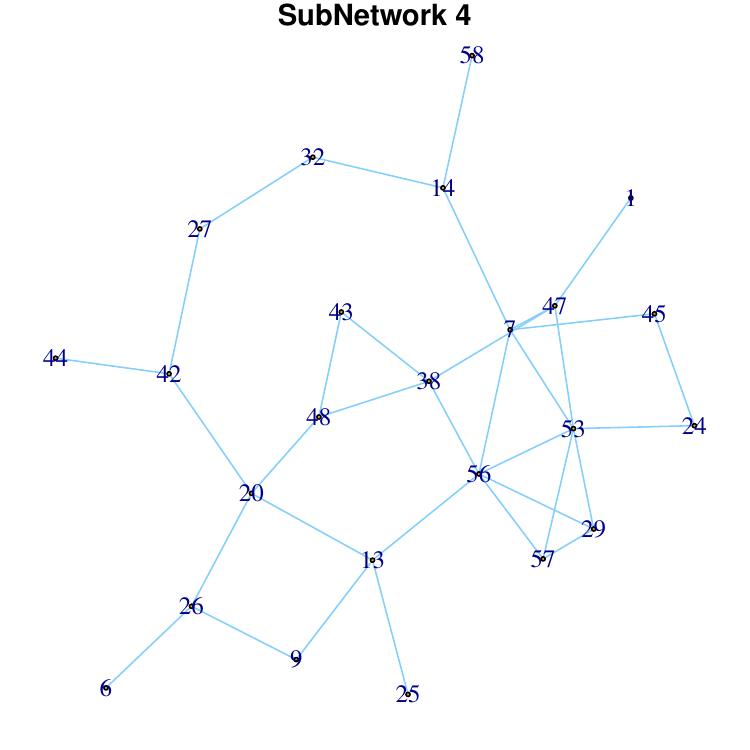}
\caption{\it\footnotesize Global network and four subnetworks based on coauthorship data.}
\label{figure:author:network}
\end{figure}

\begin{figure}[h]
\centering
\includegraphics[width=6 in, height=2 in]{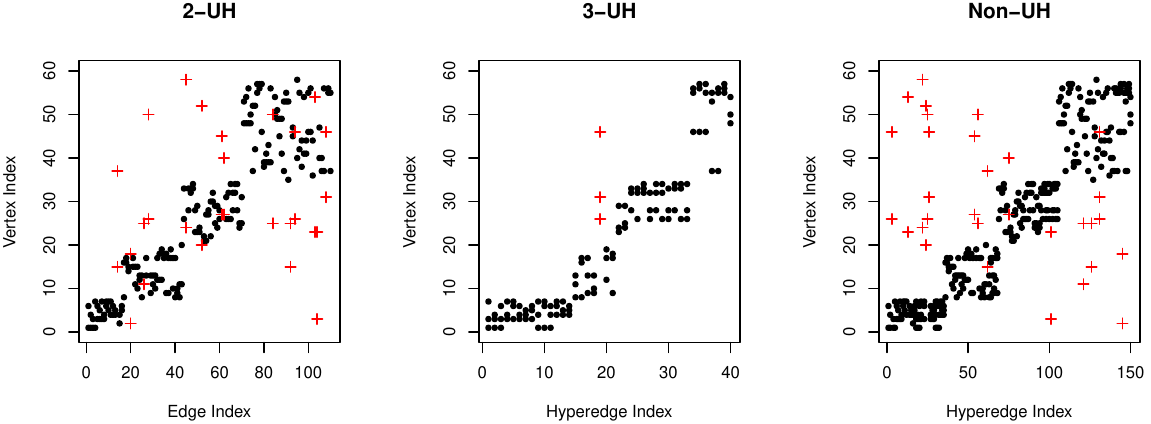}
\caption{\it\footnotesize Incidence matrices based on coauthorship data.
Left: 2-uniform hypergraph; Middle: 3-uniform hypergraph;
Right: non-uniform hypergraph.
}
\label{figure:incidence:matrix}
\end{figure}

\begin{table}[h]
\caption{\it\footnotesize Values of test statistics based on global network and four subnetworks. Symbols $**$ and $*$ indicate the strength of rejection,
i.e., p-value$<0.001$ and  p-value$<0.05$ respectively.}
\label{table:author:statistics}
\centering
\resizebox{15cm}{!}{
\begin{tabular}{cccccc}
\hline \hline
   & Global Network & SubNetwork 1 & SubNetwork 2 & SubNetwork 3 & SubNetwork 4\\
$n$  & 58    & 7      & 13     & 14     & 24     \\
$Z_2$ & 8.360** & 0.161  & -0.030 & 2.667* & 1.661  \\
$Z_3$ & 1.451 & -0.100 & -0.211 & -0.289 & -0.052 \\
$Z$ & 6.938** & 0.043  & -0.171 & 1.682  & 1.137 \\
\hline
\end{tabular}
}

\end{table}

\section{Discussion.}\label{sec:discussion}
In the context of community testing for hypergraphs, we systematically considered various scenarios in terms of hyperedge densities and investigated distinguishability
or indistinguishability of the hypotheses in each scenario. Extensions of our results are possible.

The first line is to extend the test statistic in Section \ref{EZuniform} to tackle the model selection problem for SBM in hypergraphs. In particular, one possibility is to study the hypothesis testing problem of $H_0: k=k_0$ vs. $H_1: k>k_0$ for $k_0=1,2,\ldots$ sequentially and stop when observing a rejection. 
The second line is to extend the current results to the increasingly popular degree-corrected stochastic block models.
However, based on the current second-moment technique, the bounded degree results are not easy to establish. 
The main reason is that the moments of the
likelihood ratio do not have an explicit expression in terms of $a,b,\kappa$.
Even in the ordinary graph setting with $m=2$, this is already very difficult. 
To see this, when $\sigma_{i}=\sigma_{j}$ and $\eta_{i}=\eta_{j}$ for all $i<j$, it can be shown that
\begin{eqnarray}
&&E_WE_{W^*}\prod_{i<j}\left(\frac{p_{ij}(\sigma,W)p_{ij}(\eta,W^*)}{p_0}+\frac{q_{ij}(\sigma,W)q_{ij}(\eta,W^*)}{q_0}\right)\nonumber\\
&\approx& E_WE_{W^*}\exp\left(\beta\sum_{i<j}(W_iW_ja-d)(W_i^*W_j^*a-d)\right),\label{hard}
\end{eqnarray}
where $\beta=\frac{1}{dn^\alpha}+\frac{1}{n^{2\alpha}}$.
The expected value (\ref{hard}) seems difficult to analyze under general random weights $W,W^*$, even for the above special choice of $\sigma,\eta$.
Hence, a precise contiguity region in terms of $a,b,\kappa$ is not available using the current second-moment method. 

The third line is to test more general and complicated hypotheses. The current paper only deals with the relatively simple Erd\"{o}s-R\'{e}nyi null hypotheses,
whereas the proposed methods may be extended to more general settings. For instance, in light of Theorem \ref{normaltest}, 
the test statistics based on long loose cycles may also test the null hypothesis that
the hypergraph is an SBM with $k$ communities in which $k>1$ is given; in light of Theorems \ref{normality} and \ref{power},
the test statistics based on sub-hypergraph counts may be extended to the null hypothesis that
the model is a degree-corrected SBM with $k$ communities. It is a worthwhile project to investigate the validity of these methods, especially in real-world situations. 

\section{Proof of Main Results.}\label{sec:main:proofs}
In this section, we prove the main results of this paper. The proofs of Lemmas \ref{trace}, \ref{sumtrace},
\ref{cyclepois}, \ref{matrixde}, \ref{rate} and Propositions \ref{proph1}, \ref{fixeddegree} are relegated to the supplement.

 \subsection{Proof of Theorem \ref{contiguous}.}
The proof is based on one result in Janson \cite{J95} as below. 
  \begin{Proposition}[Janson, 1995]\label{propct}
 Suppose that $L_n=\frac{d\mathbb{Q}_n}{d\mathbb{P}_n}$, regarded as a random variable on $(\Omega_n,\mathcal{F}_n, \mathbb{P}_n)$, converges in distribution to some random variable $L$ as $n\rightarrow\infty$. Then $\mathbb{P}_n$ and $\mathbb{Q}_n$ are contiguous if and only if $L>0$ a.s. and $\mathbb{E}L=1$.
 \end{Proposition}

We prove Theorem \ref{contiguous} for $k=2$. The general case can be proved similarly, but with more tediousness.
For convenience, we use $\sigma_i=+$ or $-$ (rather than $\sigma_i=1$ or $2$) to represent the potential community label of $i$.
We use $i_1:i_m$ to represent the ordering $i_1i_2\dots i_m$, and hence,
$A_{i_1:i_m}=A_{i_1i_2\dots i_m}$. Define $I[\sigma_{i_1}:\sigma_{i_m}]=I[\sigma_{i_1}=\sigma_{i_2}=\dots=\sigma_{i_m}]$.
  Let $d=\frac{a+(2^{m-1}-1)b}{2^{m-1}}$, $p_0=\frac{d}{n^{\alpha}}$, $q_0=1-p_0$. Therefore,
the hyperedge probabilities $p_{i_1i_2\dots i_m}(\sigma)$ and $q_{i_1i_2\dots i_m}(\sigma)$ are rewritten as 
 \[
 p_{i_1:i_m}(\sigma)=\mathbb{P}(A_{i_1:i_m}=1|\sigma)=\Big(\frac{a}{n^{\alpha}}\Big)^{I[\sigma_{i_1}:\sigma_{i_m}]}\Big(\frac{b}{n^{\alpha}}\Big)^{1-I[\sigma_{i_1}:\sigma_{i_m}]},  
 \]
 and $q_{i_1:i_m}(\sigma)=1-p_{i_1:i_m}(\sigma)$. Let $Y_n=\mathbb{P}_{H_1}(A)/\mathbb{P}_{H_0}(A)$
 be the likelihood ratio of the adjacent tensor $A$, where
 $\mathbb{P}_{H_0}$ and $\mathbb{P}_{H_1}$ are the probability measures under $H_0$ and $H_1$ respectively.
  Then $Y_n=2^{-n}\sum_{\sigma\in\{\pm\}^n}\prod_{i\in c(m,n)}
 \Big(\frac{p_{i_1:i_m}(\sigma)}{p_0}\Big)^{A_{i_1:i_m}}\Big(\frac{q_{i_1:i_m}(\sigma)}{q_0}\Big)^{1-A_{i_1:i_m}}$ which leads to that
 \[Y_n^2=2^{-2n}\sum_{\sigma,\eta\in\{\pm\}^n}\prod_{i\in c(m,n)}\Big(\frac{p_{i_1:i_m}(\sigma)p_{i_1:i_m}(\eta)}{p_0^2}\Big)^{A_{i_1:i_m}}\Big(\frac{q_{i_1:i_m}(\sigma)q_{i_1:i_m}(\eta)}{q_0^2}\Big)^{1-A_{i_1:i_m}}.\]
 The expectation of $Y_n^2$ under $H_0$ is
 \begin{eqnarray}\label{ctg1}
 \mathbb{E}_0Y_n^2&=&2^{-2n}\sum_{\sigma,\eta\in\{\pm\}^n}\prod_{i\in c(m,n)}\Big(\frac{p_{i_1:i_m}(\sigma)p_{i_1:i_m}(\eta)}{p_0}+\frac{q_{i_1:i_m}(\sigma)q_{i_1:i_m}(\eta)}{q_0}\Big).
 \end{eqnarray}
For any $\sigma,\eta\in\{\pm\}^n$, 
define $s_2=\#\{1\leq i_1<i_2<\dots<i_m\leq n: I[\sigma_{i_1}:\sigma_{i_m}]+I[\eta_{i_1}:\eta_{i_m}]=2\}$,
$s_1=\#\{1\leq i_1<i_2<\dots<i_m\leq n: I[\sigma_{i_1}:\sigma_{i_m}]+I[\eta_{i_1}:\eta_{i_m}]=1\}$
and $s_0=\#\{1\leq i_1<i_2<\dots<i_m\leq n: I[\sigma_{i_1}:\sigma_{i_m}]+I[\eta_{i_1}:\eta_{i_m}]=0\}$. 
Note that $s_0$, $s_1$, $s_2$ are bounded above by $n^m$. By direct examinations, we have
 \begin{eqnarray*}
 \frac{1}{p_0}\Big(\frac{a}{n^{\alpha}}\Big)^2+\frac{1}{q_0}\Big(1-\frac{a}{n^{\alpha}}\Big)^2&=&1+\frac{(a-d)^2}{dn^{\alpha}}+\frac{(a-d)^2}{n^{2\alpha}}+O(\frac{1}{n^{3\alpha}}),\\
  \frac{1}{p_0}\frac{a}{n^{\alpha}}\frac{b}{n^{\alpha}}+\frac{1}{q_0}\Big(1-\frac{a}{n^{\alpha}}\Big)\Big(1-\frac{b}{n^{\alpha}}\Big)&=&1+\frac{(a-d)(b-d)}{dn^{\alpha}}+\frac{(a-d)(b-d)}{n^{2\alpha}}+O(\frac{1}{n^{3\alpha}}),\\
  \frac{1}{p_0}\Big(\frac{b}{n^{\alpha}}\Big)^2+\frac{1}{q_0}\Big(1-\frac{b}{n^{\alpha}}\Big)^2&=&1+\frac{(b-d)^2}{dn^{\alpha}}+\frac{(b-d)^2}{n^{2\alpha}}+O(\frac{1}{n^{3\alpha}}).
 \end{eqnarray*}
Then for $\alpha>\frac{m}{2}$, we have by (\ref{ctg1}) that
 \begin{eqnarray}\nonumber
 \mathbb{E}_0Y_n^2&=&(1+o(1))\mathbb{E}_{\sigma\eta}\Big\{\Big(1+\frac{(a-d)^2}{dn^{\alpha}}\Big)^{s_2}
 \Big(1+\frac{(a-d)(b-d)}{dn^{\alpha}}\Big)^{s_1}  
\Big(1+\frac{(b-d)^2}{dn^{\alpha}}\Big)^{s_0}\Big\}\\  \label{ctg}
&=&(1+o(1))\mathbb{E}_{\sigma\eta}\exp\Big\{\frac{(a-d)^2}{dn^{\alpha}}s_2+\frac{(a-d)(b-d)}{dn^{\alpha}}s_1+\frac{(b-d)^2}{dn^{\alpha}}s_0\Big\}.
 \end{eqnarray}
If $\alpha> m$, then $\frac{s_j}{n^{\alpha}}\rightarrow 0$ for $j=0,1,2$.
Hence $ \mathbb{E}_0Y_n^2\rightarrow 1$. Since $\mathbb{E}_0Y_n=1$, we have that $Y_n$ converges to 1 in distribution. By Proposition \ref{propct}, $H_0$ and $H_1$ are contiguous.

Next we consider $\alpha=m$. Note that
 \begin{eqnarray*}
 s_2&=&\sum_{i\in c(m,n)}I[\sigma_{i_1}:\sigma_{i_m}]I[\eta_{i_1}:\eta_{i_m}],\\
 s_1&=&\sum_{i\in c(m,n)}\Big(I[\sigma_{i_1}:\sigma_{i_m}](1-I[\eta_{i_1}:\eta_{i_m}])+(1-I[\sigma_{i_1}:\sigma_{i_m}])I[\eta_{i_1}:\eta_{i_m}]\Big),\\ s_0&=&\sum_{c(i,m,n)}(1-I[\sigma_{i_1}:\sigma_{i_m}])(1-I[\eta_{i_1}:\eta_{i_m}]).
 \end{eqnarray*}
 Then the numerator of the exponent in (\ref{ctg}) can be written as
 \begin{eqnarray}\nonumber
 &&(a-d)^2s_2+(a-d)(b-d)s_1+(b-d)^2s_0\\  \nonumber
 &=&\binom{n}{m}(b-d)^2+(a-b)^2\sum_{c(i,m,n)}I[\sigma_{i_1}:\sigma_{i_m}]I[\eta_{i_1}:\eta_{i_m}]\\   \label{ctg3}
 & &+(a-b)(b-d)\Big(\sum_{i\in c(m,n)}I[\sigma_{i_1}:\sigma_{i_m}]+\sum_{c(i,m,n)}I[\eta_{i_1}:\eta_{i_m}]\Big).
 \end{eqnarray}
 
 For $s, t=+1,-1$, let 
 \begin{eqnarray*}
 \rho_{st}=\sum_{i=1}^nI[\sigma_i=t]I[\eta_i=s],\  \
 \rho_{t0}=\sum_{i=1}^nI[\sigma_i=t],\ \
 \rho_{0s}=\sum_{i=1}^nI[\eta_i=s],
 \end{eqnarray*}
 and 
  \begin{eqnarray*}
 \tilde{\rho}_{st}=\frac{1}{\sqrt{n}}\sum_{i=1}^n(I[\sigma_i=t]I[\eta_i=s]-\frac{1}{2^2}),\
 \tilde{\rho}_{t0}=\frac{1}{\sqrt{n}}\sum_{i=1}^n(I[\sigma_i=t]-\frac{1}{2}),\
 \tilde{\rho}_{0s}=\frac{1}{\sqrt{n}}\sum_{i=1}^n(I[\eta_i=s]-\frac{1}{2}).
 \end{eqnarray*}
It is easy to verify that $\sum_{s,t}\tilde{\rho}_{st}=0$, $\sum_{s}\tilde{\rho}_{s0}=0$, $\sum_{t}\tilde{\rho}_{0t}=0$ and
\[\sum_{1\leq i_1,\dots,i_m\leq n}I[\sigma_{i_1}:\sigma_{i_m}]I[\eta_{i_1}:\eta_{i_m}]=m!\sum_{i_1<i_2<\dots<i_m}I[\sigma_{i_1}:\sigma_{i_m}]I[\eta_{i_1}:\eta_{i_m}]+O(n^{m-1}).\]
Then we have
 \begin{eqnarray}\nonumber
\sum_{i\in c(m,n)}I[\sigma_{i_1}:\sigma_{i_m}]I[\eta_{i_1}:\eta_{i_m}]&=&\frac{1}{m!}\sum_{1\leq i_1,\dots,i_m\leq n}I[\sigma_{i_1}:\sigma_{i_m}]I[\eta_{i_1}:\eta_{i_m}]+O(n^{m-1})\\  \nonumber
&=&\frac{1}{m!}\sum_{1\leq i_1,\dots,i_m\leq n}\sum_{s,t=-1,+1}\prod_{j=1}^mI[\sigma_{i_j}=s]I[\eta_{i_j}=t]+O(n^{m-1})\\ \label{ctg4} 
&=&\frac{1}{m!}\sum_{s,t=-1,+1}\rho_{st}^m+O(n^{m-1})\\  \nonumber
&=&\frac{1}{m!}\sum_{s,t=-1,+1}(\sqrt{n}\tilde{\rho}_{st}+\frac{n}{2^2})^m+O(n^{m-1})\\ \nonumber 
&=&\frac{1}{m!}\frac{4n^m}{2^{2m}}+\frac{1}{m!}n^{m-1}\sum_{s,t}\tilde{\rho}_{st}^2\sum_{k=2}^m\binom{m}{k}\frac{1}{2^{2(m-k)}}\Big(\frac{\tilde{\rho}_{st}}{\sqrt{n}}\Big)^{k-2}+O(n^{m-1}),
 \end{eqnarray}
\begin{eqnarray}\nonumber
  \sum_{i\in c(m,n)}I[\sigma_{i_1}:\sigma_{i_m}]&=&\frac{1}{m!}\sum_{1\leq i_1,\dots,i_m\leq n}I[\sigma_{i_1}:\sigma_{i_m}]+O(n^{m-1})\\  \label{ctg5}
  &=&\frac{1}{m!}\sum_{t=-1,+1}\rho_{t0}^m+O(n^{m-1})\\    \nonumber
  &=&\frac{1}{m!}\frac{2n^m}{2^m}+\frac{n^{m-1}}{m!}\sum_{t}\tilde{\rho}_{t0}^2 \sum_{k=2}^m\binom{m}{k}\frac{1}{2^{(m-k)}}\Big(\frac{\tilde{\rho}_{t0}}{\sqrt{n}}\Big)^{k-2}+O(n^{m-1}),\nonumber
\end{eqnarray}  
\begin{eqnarray}
    \sum_{i\in c(m,n)}I[\eta_{i_1}:\eta_{i_m}]&=&\frac{1}{m!}\sum_{1\leq i_1,\dots,i_m\leq n}I[\eta_{i_1}:\eta_{i_m}]+O(n^{m-1})\\    \label{ctg6}
  &=&\frac{1}{m!}\sum_{s=-1,+1}\rho_{0s}^m+O(n^{m-1})\\  \nonumber
  &=&\frac{1}{m!}\frac{2n^m}{2^m}+\frac{n^{m-1}}{m!}\sum_{s}\tilde{\rho}_{0s}^2 \sum_{k=2}^m\binom{m}{k}\frac{1}{2^{(m-k)}}\Big(\frac{\tilde{\rho}_{0s}}{\sqrt{n}}\Big)^{k-2}+O(n^{m-1}).
  \end{eqnarray}
 
 If $\alpha=m$, by (\ref{ctg3}), (\ref{ctg4}), (\ref{ctg5}), (\ref{ctg6}), and law of large number, we have
 \begin{eqnarray}\nonumber
 (a-d)^2\frac{s_2}{n^m}+(a-d)(b-d)\frac{s_1}{n^m}+(b-d)^2\frac{s_0}{n^m}
&\rightarrow&(a-b)^2\frac{4}{2^{2m}}+(a-b)(b-d)\frac{4}{2^m}+(b-d)^2\\   \label{ctg7}
&=&\Big(\frac{a-b}{2^{m-1}}+(b-d)\Big)^2=0. 
 \end{eqnarray}
Combining (\ref{ctg}) and (\ref{ctg7}), we get that $ \mathbb{E}_0Y_n^2\rightarrow 1$, which implies that $H_0$ and $H_1$ are contiguous by Proposition \ref{propct}.

 Let $\alpha=m-1+\delta$, for $0<\delta<1$. Note that $|\frac{\tilde{\rho}_{st}}{\sqrt{n}}|$, $|\frac{\tilde{\rho}_{s0}}{\sqrt{n}}|$, $|\frac{\tilde{\rho}_{0t}}{\sqrt{n}}|$ are all bounded by 1. Hence, 
 there is a universal constant $C$ such that
 \begin{eqnarray*}
   \frac{(a-b)^2}{dm!}\Big|\sum_{k=2}^m\binom{m}{k}\frac{1}{2^{2(m-k)}}\Big(\frac{\tilde{\rho}_{ts}}{\sqrt{n}}\Big)^{k-2} \Big|&\leq&C,\\
   \frac{(a-b)(b-d)}{dm!}\Big|\sum_{k=2}^m\binom{m}{k}\frac{1}{2^{(m-k)}}\Big(\frac{\tilde{\rho}_{t0}}{\sqrt{n}}\Big)^{k-2} \Big|&\leq&C,\\
  \frac{(a-b)(b-d)}{dm!}\Big|\sum_{k=2}^m\binom{m}{k}\frac{1}{2^{(m-k)}}\Big(\frac{\tilde{\rho}_{0s}}{\sqrt{n}}\Big)^{k-2} \Big|&\leq&C.
 \end{eqnarray*}
 Note that $(b-d)^2+\frac{4}{2^{2m}}(a-b)^2+\frac{4}{2^m}(a-b)(b-d)=0$.
 Then by (\ref{ctg}), (\ref{ctg3}), (\ref{ctg4}), (\ref{ctg5}), (\ref{ctg6}), we have
 \begin{eqnarray}\label{ctg8}
 \mathbb{E}_0Y_n^2&\leq&(1+o(1))\mathbb{E}_{\sigma\eta}\exp\Big\{\sum_{s,t}\frac{C}{n^{\delta}}\tilde{\rho}_{st}^2+\sum_{t}\frac{C}{n^{\delta}}\tilde{\rho}_{t0}^2+\sum_{s}\frac{C}{n^{\delta}}\tilde{\rho}_{0s}^2+O(\frac{1}{n^{\delta}}) \Big\}.
 \end{eqnarray}
 By central limit theorem and Slutsky's theorem, $\tilde{\rho}_{st}^2$, $\tilde{\rho}_{s0}^2$ and $\tilde{\rho}_{0t}^2$ converge to chi-square distributions, 
 which implies that $\frac{C}{n^{\delta}}\tilde{\rho}_{st}^2$, $\frac{C}{n^{\delta}}\tilde{\rho}_{s0}^2$ and $\frac{C}{n^{\delta}}\tilde{\rho}_{0t}^2$ converge to zero in probability. 
 For any $\gamma>0$ and $\beta>0$, by Hoeffding inequality, we have
 \begin{eqnarray*}
 \mathbb{P}\Big(\exp\Big\{\frac{C}{n^{\delta}}\tilde{\rho}_{st}^2\Big\}>\gamma^{\beta}\Big)= \mathbb{P}\Big(\frac{|\tilde{\rho}_{ts}|}{\sqrt{n}}>\sqrt{\frac{n^{\delta}\log \gamma^{\beta}}{Cn}}\Big)
 \leq2\exp\Big\{-\frac{n^{\delta}\log \gamma^{\beta}}{Cn}\frac{n}{m}\Big\}
 =2\gamma^{-\beta\frac{n^{\delta}}{C}}.
 \end{eqnarray*}
 Choose a $n_0>0$ such that $C<\beta n_0^\delta$. For any $n\ge n_0$ and $C_1>0$, we have 
 \begin{equation}\label{eqn:uniform:integrability}
 \int_{C_1}^{\infty} \mathbb{P}\Big(\exp\Big\{\frac{C}{n^{\delta}}\tilde{\rho}_{st}^2\Big\}>\gamma^{\beta}\Big)d\gamma\leq\frac{2}{\frac{\beta n_0^{\delta}}{C}-1}C_1^{1-\frac{\beta n_0^{\delta}}{C}}. 
 \end{equation}
Notice that there are totally eight items in the summation $\sum_{s,t}+\sum_{s}+\sum_{t}$ where the sums range over $s,t=\pm$.
Therefore, we have
 \begin{eqnarray*}
 && \mathbb{P}\Big(\exp\Big\{\sum_{s,t}\frac{C}{n^{\delta}}\tilde{\rho}_{st}^2+\sum_{t}\frac{C}{n^{\delta}}\tilde{\rho}_{t0}^2+\sum_{s}\frac{C}{n^{\delta}}\tilde{\rho}_{0s}^2\Big\}>t\Big)\\
 &\leq&\sum_{s,t}  \mathbb{P}\Big(\exp\Big\{\frac{C}{n^{\delta}}\tilde{\rho}_{st}^2\Big\}>t^{\frac{1}{8}}\Big)+\sum_{s} \mathbb{P}\Big(\exp\Big\{\frac{C}{n^{\delta}}\tilde{\rho}_{t0}^2\Big\}>t^{\frac{1}{8}}\Big)+\sum_{t} \mathbb{P}\Big(\exp\Big\{\frac{C}{n^{\delta}}\tilde{\rho}_{0s}^2\Big\}>t^{\frac{1}{8}}\Big).
 \end{eqnarray*}
Together with (\ref{eqn:uniform:integrability}),
the variable in the right side of (\ref{ctg8}) is uniform integrable. 
By $\mathbb{E}_0Y_n^2\ge 1$, we conclude that $\mathbb{E}_0Y_n^2\rightarrow 1$, hence $H_0$ and $H_1$ are contiguous by Proposition \ref{propct}.

For $k>2$, let $S=\{1,2,\dots, k\}$ and $\sigma_i\in S$. It can be checked that 
\[Y_n=k^{-n}\sum_{\sigma\in S^n}\prod_{i\in c(m,n)}
 \Big(\frac{p_{i_1:i_m}(\sigma)}{p_0}\Big)^{A_{i_1:i_m}}\Big(\frac{q_{i_1:i_m}(\sigma)}{q_0}\Big)^{1-A_{i_1:i_m}}.\]
The rest of the proof follows by a line-by-line check of the $k=2$ case.

\subsection{Proof of Theorem \ref{orthogonality}.}

The key idea in proving Theorem \ref{orthogonality} is to count the
\textit{long} loose cycles and use Theorem 1 in Gao and Wormald \cite{GWALD}. Here ``long'' means 
that the number of hyperedges in the loose cycle diverges along with $n$.
Recall Theorem 1 from Gao and Wormald \cite{GWALD} below.  
\begin{Theorem}[Gao and Wormald, 2004]\label{GWTHM}
 Let $s_n>-\frac{1}{\mu_n}$ and $\sigma_n=\sqrt{\mu_n+\mu_n^2s_n}$, 
 where $\mu_n>0$ satisfies $\mu_n\rightarrow\infty$. Suppose that $\mu_n=o(\sigma_n^3)$ and $\{X_n\}$ is a sequence of nonnegative random variables satisfying
 \begin{equation*} 
 \mathbb{E}[X_n]_k\sim\mu_n^k\exp\Big(\frac{k^2s_n}{2}\Big),
 \end{equation*}
 uniformly for all integers $k$ in the range $c_1\mu_n/\sigma_n\leq k\leq c_2\mu_n/\sigma_n$ for some constants $c_2>c_1>0$. Then $(X_n-\mu_n)/\sigma_n$ 
 converges in distribution to the standard normal variable as $n\rightarrow\infty$.  Here $\alpha_n\sim\beta_n$ means $\lim_{n\rightarrow\infty}\frac{\alpha_n}{\beta_n}=1$. 
 \end{Theorem}
 
  Let $X_{\xi_n}$ be the number of $\xi_n$-hyperedge loose cycles over the observed hypergraph. We will compute the expectation of $[X_{\xi_n}]_s$ under $H_1$. 
  Consider the $s$-tuple of $\xi_n$-hyperedge loose cycles $(H_{\xi_n1},\dots,H_{\xi_ns})$ 
  in which $H_{\xi_n j}$ are $\xi_n$-hyperedge loose cycles. Let $B$ be the collection of such $s$-tuples with vertex
  disjoint cycles and $\bar{B}$ be the collection of tuples in which two cycles have common vertex. 
  The expectation of $[X_{\xi_n}]_s$ under $H_1$ can be expressed as
\[
  \mathbb{E}_1[X_{\xi_n}]_s=\sum_{B}\mathbb{E}_1I_{\cup_{i=1}^s H_{\xi_ni}}+\sum_{\bar{B}}\mathbb{E}_1I_{\cup_{i=1}^s H_{\xi_ni}}.
\]
 Let $\tau$ be a random label assignment.  The first term in the right hand side of the above equation is
 \begin{eqnarray*}
 \mathbb{E}_1I_{\cup_{i=1}^s H_{\xi_ni}}&=&\mathbb{E}_1\prod_{i=1}^sI_{H_{\xi_ni}}
 =\mathbb{E}_{\tau}\prod_{i=1}^s\mathbb{E}_1I_{H_{\xi_ni}}
 =\mathbb{E}_{\tau}\prod_{i=1}^s\prod_{\{i_1,\dots, i_m\}\in\mathcal{E}(H_{\xi_ni})}\frac{M_{i_1i_2\dots i_m}(\tau)}{n^{m-1}}\\
 &=&\prod_{i=1}^s\Big(\Big[\frac{a+(k^{m-1}-1)b}{(kn)^{m-1}}\Big]^{\xi_n}+(k-1)\Big[\frac{a-b}{(kn)^{m-1}}\Big]^{\xi_n}\Big)\\
 &=&\frac{1}{n^{(m-1)\xi_ns}}\Big(\Big[\frac{a+(k^{m-1}-1)b}{{k}^{m-1}}\Big]^{\xi_n}+(k-1)\Big[\frac{a-b}{k^{m-1}}\Big]^{\xi_n}\Big)^s,
 \end{eqnarray*}
 where $\mathcal{E}(H_{\xi_ni})$ is the hyperedge set of $H_{\xi_n i}$.
 Note that $\# B=\frac{n!}{(n-M_1)!}\Big(\frac{1}{2\xi_n(m-2)!^{\xi_n}}\Big)^s$, where $M_1=(m-1)\xi_ns$. Then  for $M_1=o(\sqrt{n})$,
 \begin{eqnarray*}
 \sum_{B}\mathbb{E}_1I_{\cup_{i=1}^s H_{\xi_ni}}&=&\#B\times\mathbb{E}_1I_{\cup_{i=1}^s H_{\xi_ni}}\\
 &=&\frac{n!}{(n-M_1)!}n^{-M_1}\Big(\frac{1}{2\xi_n}\Big[\frac{a+(k^{m-1}-1)b}{{k}^{m-1}(m-2)!}\Big]^{\xi_n}+\frac{(k-1)}{2\xi_n}\Big[\frac{a-b}{k^{m-1}(m-2)!}\Big]^{\xi_n}\Big)^s\\
 &\sim&\Big(\frac{1}{2\xi_n}\Big[\frac{a+(k^{m-1}-1)b}{{k}^{m-1}(m-2)!}\Big]^{\xi_n}+\frac{(k-1)}{2\xi_n}\Big[\frac{a-b}{k^{m-1}(m-2)!}\Big]^{\xi_n}\Big)^s.
 \end{eqnarray*}
The ``$\sim$'' is due to the trivial fact that $\frac{n!}{(n-M_1)!}n^{-M_1}\to1$ as $M_1=o(\sqrt{n})$. 
 Note that $\#\bar{B}\leq M_1^2n^{M_1-1}$ and
$
  \mathbb{E}_1[I_{\cup_{i=1}^s H_{\xi_ni}}|\tau]\leq\Big(\frac{a}{n^{m-1}}\Big)^{|\mathcal{E}(H)|}
$,
 then 
 \begin{eqnarray*}
 \sum_{\bar{B}}\mathbb{E}_1I_{\cup_{i=1}^s H_{\xi_ni}}\leq M_1^2n^{M_1-1}\Big(\frac{a}{n^{m-1}}\Big)^{|\mathcal{E}(H)|}
 = M_1^2\frac{a^{M_1}}{n} 
 \rightarrow0,
 \end{eqnarray*}
 provided that $M_1\leq\delta_1\log_an$ for a constant $0<\delta_1<1$. 
 
Define $\mu_{n1}=\frac{1}{2\xi_n}\Big[\frac{a+(k^{m-1}-1)b}{{k}^{m-1}(m-2)!}\Big]^{\xi_n}+\frac{(k-1)}{2\xi_n}\Big[\frac{a-b}{k^{m-1}(m-2)!}\Big]^{\xi_n}$ and $\mu_{n0}=\frac{1}{2\xi_n}\Big[\frac{a+(k^{m-1}-1)b}{{k}^{m-1}(m-2)!}\Big]^{\xi_n}$.
If $M_1\leq\delta_1\log_an$, then
 \begin{equation}\label{ks1}
\mathbb{E}_1[X_{\xi_n}]_s\sim \mu_{n1}^s,
\end{equation}
  \begin{equation}\label{ks0}
 \mathbb{E}_0[X_{\xi_n}]_s\sim \mu_{n0}^s.
 \end{equation}

 Note that $\kappa>1$ implies $\lambda_m>1$. 
 To see this, let $a=c+(k^{m-1}-1)d$ and $b=c-d$ for some constants $c>d>0$. Then it follows from $\kappa>1$ that $c>(m-2)!$, which yields $\lambda_m>1$. Then $\mu_{n1}, \mu_{n0}\rightarrow\infty$ as $n\rightarrow\infty$. It is obvious that 
 \[\mu_{n1}\leq\frac{\Big(\log_{\gamma}n\Big)^{\delta_0}}{\xi_n}, \ \
 \mu_{n0}\leq\frac{\Big(\log_{\gamma}n\Big)^{\delta_0}}{\xi_n}.\]
 Let $\sigma_{n1}=\sqrt{\mu_{n1}}$, $\sigma_{n0}=\sqrt{\mu_{n0}}$. For any constant $c_2>c_1>0$ and $s$ satisfying $c_1\frac{\mu_{n1}}{\sigma_{n1}}\leq s\leq c_2\frac{\mu_{n1}}{\sigma_{n1}}$ or $c_1\frac{\mu_{n0}}{\sigma_{n0}}\leq s\leq c_2\frac{\mu_{n0}}{\sigma_{n0}}$, we have for large $n$
 \[M_1=(m-1)\xi_ns=(m-1)\sqrt{\Big(\log_{\gamma}n\Big)^{\delta_0}\log_{\lambda_m}\Big(\log_{\gamma}n\Big)^{\delta_0}}\leq \delta_1\log_an,
 \]
which implies (\ref{ks1}) and  (\ref{ks0}) hold. By Theorem \ref{GWTHM}, we conclude that $\frac{X_{\xi_n}-\mu_{n1}}{\sqrt{\mu_{n1}}}$ and $\frac{X_{\xi_n}-\mu_{n0}}{\sqrt{\mu_{n0}}}$ converge in distribution to the standard normal variables under $H_1$ and $H_0$, respectively.

Since $\kappa>1$, there exits a constant $\rho$ satisfying
 \[\sqrt{\frac{a+(k^{m-1}-1)b}{{k}^{m-1}(m-2)!}}<\rho<\frac{a-b}{k^{m-1}(m-2)!}.\]
 It is easy to verify that $\mu_{n1}=o(\rho^{2\xi_n})$, $\mu_{n0}=o(\rho^{2\xi_n})$. Let $A_n=\{X_{\xi_n}\leq \mathbb{E}_0X_{\xi_n}+\rho^{\xi_n}\}$. Then we have
 \begin{eqnarray}\label{p0conv1}
 \mathbb{P}_{H_0}(A_n)&=&\mathbb{P}_{H_0}\Big(\frac{X_{\xi_n}-\mu_{n0}}{\sqrt{\mu_{n0}}}\leq\frac{\rho^{\xi_n}}{\sqrt{\mu_{n0}}}\Big)\rightarrow\Phi(\infty)=1.
 \end{eqnarray}
 Note that 
$\frac{\mu_{n1}-\mu_{n0}}{\rho^{\xi_n}}\rightarrow\infty$, 
 then for large $n$, we have $\mu_{n1}-\rho^{\xi_n}\geq \mu_{n0}+\rho^{\xi_n}$. Then it yields  
  \begin{eqnarray}\label{p1conv0}
 \mathbb{P}_{H_1}(A_n)\leq\mathbb{P}_{H_1}\Big(X_{\xi_n}\leq \mathbb{E}_1X_{\xi_n}-\rho^{\xi_n}\Big) 
 = \mathbb{P}_{H_1}\Big(\frac{X_{\xi_n}-\mu_{n1}}{\sqrt{\mu_{n1}}}\leq-\frac{\rho^{\xi_n}}{\sqrt{\mu_{n1}}}\Big)\rightarrow\Phi(-\infty)=0.
 \end{eqnarray}
By definition, (\ref{p0conv1}) and (\ref{p1conv0}) shows that $H_0$ and $H_1$ are orthogonal.

\subsection{Proof of Theorem \ref{consistent:a:b}.}  Let  $f=\frac{a-b}{k^{m-1}(m-2)!}$.
By the proof of Theorem \ref{orthogonality}, it is easy to show that for any $\epsilon>0$,
\[\mathbb{P}_{H_1}\Big(\frac{2\xi_nX_{\xi_n}-\lambda_m^{\xi_n}-(k-1)f^{\xi_n}}{(k-1)f^{\xi_n}}>\epsilon\Big)=\mathbb{P}_{H_1}\Big(\frac{X_{\xi_n}-\mu_{n1}}{\sqrt{\mu_{n1}}}>\frac{(k-1)f^{\xi_n}\epsilon}{2\xi_n\sqrt{\mu_{n1}}}\Big)=1-\Phi\Big(\frac{(k-1)f^{\xi_n}\epsilon}{2\xi_n\sqrt{\mu_{n1}}}\Big)\rightarrow0,\]
and $\mathbb{P}_{H_1}\Big(\frac{2\xi_nX_{\xi_n}-\lambda_m^{\xi_n}-(k-1)f^{\xi_n}}{(k-1)f^{\xi_n}}<-\epsilon\Big)\rightarrow0$. Then it follows that $2\xi_nX_{\xi_n}-\lambda_m^{\xi_n}=(1+o_p(1))(k-1)f^{\xi_n}$.

Next, we show that $\widehat{\lambda}_m^{\xi_n}-\lambda_m^{\xi_n}=o_p(1)$. For simplicity, we only show $\widehat{\lambda}_3^{\xi_n}-\lambda_3^{\xi_n}=o_p(1)$, the general case follows similarly. Let $\eta_{ijt}=\frac{(a-b)I[\sigma_i=\sigma_j=\sigma_t]+b}{n^2}$. By Taylor expansion, we have
 \[\widehat{\lambda}_3^{\xi_n}-\lambda_3^{\xi_n}=\sum_{i=1}^{\xi_n}\frac{\xi_n(\xi_n-1)\dots (\xi_n-i+1)}{i!}\lambda_3^{\xi_n-i}(\widehat{\lambda}_3-\lambda_3)^i,\]
 from which it follows that
 \begin{equation}\label{ab0}
\mathbb{E}(\widehat{\lambda}_3^{\xi_n}-\lambda_3^{\xi_n})^2=\sum_{i,j=1}^{\xi_n}C_{ij}\lambda_3^{2\xi_n-i-j}\mathbb{E}(\widehat{\lambda}_3 -\lambda_3)^{i+j},
 \end{equation}
 where $C_{ij}=\frac{\xi_n(\xi_n-1)\dots (\xi_n-i+1)}{i!}\frac{\xi_n(\xi_n-1)\dots (\xi_n-j+1)}{j!}\leq \xi_n^{2\xi_n}$. For any integer $s$ with $2\leq s\leq 2\xi_n$, we calculate $\mathbb{E}(\widehat{\lambda}_3 -\lambda_3)^{s}$ as 
 follows:
 \begin{eqnarray}\nonumber
&& \mathbb{E}(\widehat{\lambda}_3 -\lambda_3)^{s}\\ \nonumber
&=& \mathbb{E}\Big[\frac{n^2}{\binom{n}{3}}\sum_{i<j<t}\Big(A_{ijt}-\frac{a+(k^2-1)b}{n^2k^2}\Big)\Big]^s\\ \nonumber
 &=&\frac{n^{2s}}{\binom{n}{3}^s}\sum_{i_r<j_r<t_r,r=1,\dots,s}\mathbb{E}\Big[\Big(A_{i_1j_1t_1}-\frac{a+(k^2-1)b}{n^2k^2}\Big)\dots\Big(A_{i_sj_st_s}-\frac{a+(k^2-1)b}{n^2k^2}\Big)\Big]\\ \nonumber
 &=&\frac{n^{2s}}{\binom{n}{3}^s}\sum_{i_r<j_r<t_r,r=1,\dots,s}\mathbb{E}\Big[\Big(A_{i_1j_1t_1}-\eta_{i_1j_1t_1}+\eta_{i_1j_1t_1}-\frac{a+(k^2-1)b}{n^2k^2}\Big)\times\dots\\  \nonumber
 &&\times\Big(A_{i_sj_st_s}-\eta_{i_sj_st_s}+\eta_{i_sj_st_s}-\frac{a+(k^2-1)b}{n^2k^2}\Big)\Big]\\  \nonumber
 &=&\frac{n^{2s}}{\binom{n}{3}^s}\sum_{i_r<j_r<t_r,r=1,\dots,s}\mathbb{E}\Big[\Big(A_{i_1j_1t_1}-\eta_{i_1j_1t_1}\Big)\dots\Big(A_{i_sj_st_s}-\eta_{i_sj_st_s}\Big)+
 \dots\\  \label{ab1}
 &&+\Big(\eta_{i_1j_1t_1}-\frac{a+(k^2-1)b}{n^2k^2}\Big)\dots\Big(\eta_{i_sj_st_s}-\frac{a+(k^2-1)b}{n^2k^2}\Big)\Big].
\end{eqnarray}
 
 There are $\binom{n}{3}^s$ index triples $(i_r,j_r,t_r)$ for $1\leq r\leq s$ in total. Among them, $\binom{n}{3}\binom{n-3}{3}\dots\binom{n-3(s-1)}{3}$ ones
 are disjoint, that is, $(i_r,j_r,t_r)$ and $(i_u,j_u,t_u)$ are disjoint for any $1\leq r<u\leq s$. In the disjoint case, the independence between $\eta_{i_rj_rt_r}(1\leq r\leq s)$ yields 
 \[\mathbb{E}\Big[\Big(\eta_{i_1j_1t_1}-\frac{a+(k^2-1)b}{n^2k^2}\Big)\dots\Big(\eta_{i_sj_st_s}-\frac{a+(k^2-1)b}{n^2k^2}\Big)\Big]=0.\]
 Let $C_1>1$ be a constant such that $|\eta_{ijt}|\leq\frac{C_1}{n^2}$ and $|\eta_{ijt}-\frac{a+(k^2-1)b}{n^2k^2}|\leq\frac{C_1}{n^2}$. Let $C_2=18C_1>1$, we have
 \begin{eqnarray*}
 &&\frac{n^{2s}}{\binom{n}{3}^s}\sum_{i_r<j_r<t_r,r=1,\dots,s}\Big|\mathbb{E}\Big[\Big(\eta_{i_1j_1t_1}-\frac{a+(k^2-1)b}{n^2k^2}\Big)\dots\Big(\eta_{i_sj_st_s}-\frac{a+(k^2-1)b}{n^2k^2}\Big)\Big]\Big|\\
 &\leq&\frac{n^{2s}}{\binom{n}{3}^s}\Big[\binom{n}{3}^s-\binom{n}{3}\binom{n-3}{3}\dots\binom{n-3(s-1)}{3}\Big]\frac{C_1^s}{n^{2s}}\leq \frac{3^{s-1}(s-1)!n^{2s+1}C_1^s}{\binom{n}{3}^s}\leq \frac{C_2^s(2\xi_n)^{2\xi_n}}{n}.
 \end{eqnarray*}
 Consider the terms in (\ref{ab1}) consisting of $v$ items $(A_{ijt}-\eta_{ijt})$ for $1\le v\le s$. Typically they have the following fashion:
 \begin{equation}\label{ab2}
 \mathbb{E}\left[(A_{i_1j_1t_1}-\eta_{i_1j_1t_1})\dots (A_{i_vj_vt_v}-\eta_{i_vj_vt_v})\Big(\eta_{i_{v+1}j_{v+1}t_{v+1}}-\frac{a+(k^2-1)b}{n^2k^2}\Big)\dots \Big(\eta_{i_sj_st_s}-\frac{a+(k^2-1)b}{n^2k^2}\Big)\right].
 \end{equation}
 The above term vanishes when $v=1$ since $\mathbb{E}[(A_{i_1j_1t_1}-\eta_{i_1j_1t_1})|\sigma]=0$. When $v=2$, 
 if $(i_1,j_1,t_1)\neq(i_2,j_2,t_2)$, then
 \[\mathbb{E}\big[(A_{i_1j_1t_1}-\eta_{i_1j_1t_1})(A_{i_2j_2t_2}-\eta_{i_2j_2t_2})\big|\sigma\big]=\mathbb{E}\big[(A_{i_1j_1t_1}-\eta_{i_1j_1t_1})\big|\sigma\big]\mathbb{E}\big[(A_{i_2j_2t_2}-\eta_{i_2j_2t_2})\big|\sigma\big]=0,\]
since $A_{ijt}$ are independent conditional on $\sigma$. This implies that (\ref{ab2}) vanishes. 
Hence, (\ref{ab2}) is nonzero if and only if $(i_1,j_1,t_1)=(i_2,j_2,t_2)$. In this case, we have
 \begin{eqnarray*}
 &&\frac{n^{2s}}{\binom{n}{3}^s}\sum_{i_r<j_r<t_r,r=1,\dots,s}\Big|\mathbb{E}(A_{i_1j_1t_1}-\eta_{i_1j_1t_1})\dots (A_{i_vj_vt_v}-\eta_{i_vj_vt_v})\\
 &&\times\Big(\eta_{i_{v+1}j_{v+1}t_{v+1}}-\frac{a+(k^2-1)b}{n^2k^2}\Big)\dots \Big(\eta_{i_sj_st_s}-\frac{a+(k^2-1)b}{n^2k^2}\Big)\Big|\\
 &=&\frac{n^{2s}}{\binom{n}{3}^s}\sum_{i_r<j_r<t_r,r=2,\dots,s}\Big|\mathbb{E}(A_{i_2j_2t_2}-\eta_{i_2j_2t_2})^2\Big(\eta_{i_{3}j_{3}t_{3}}-\frac{a+(k^2-1)b}{n^2k^2}\Big)\dots \Big(\eta_{i_sj_st_s}-\frac{a+(k^2-1)b}{n^2k^2}\Big)\Big|\\
 &=&\frac{n^{2s}}{\binom{n}{3}^s}\sum_{i_r<j_r<t_r,r=2,\dots,s}\Big|\mathbb{E}\eta_{i_2j_2t_2}(1-\eta_{i_2j_2t_2})\Big(\eta_{i_{3}j_{3}t_{3}}-\frac{a+(k^2-1)b}{n^2k^2}\Big)\dots \Big(\eta_{i_sj_st_s}-\frac{a+(k^2-1)b}{n^2k^2}\Big)\Big|\\
 &\leq&\frac{n^{2s}}{\binom{n}{3}^s}\Big[\binom{n}{3}^{s-1}\frac{C_1}{n^2}\frac{C_1^{s-2}}{n^{2(s-2)}}\Big]= \frac{n^2C_1^{s-1}}{\binom{n}{3}}\leq\frac{C_2^s(2\xi_n)^{2\xi_n}}{n}.
 \end{eqnarray*}

When $3\leq v\leq s$, for each $r$ with $1\leq r\leq v$, there exists $r_0\neq r$ such that $(i_{r_0},j_{r_0},t_{r_0})=(i_r,j_r,t_r)$. Otherwise the expectation in (\ref{ab2}) will vanish. For example, if $v=4$ and $(i_1,j_1,t_1)\neq (i_2,j_2,t_2)=(i_3,j_3,t_3)=(i_4,j_4,t_4)$, then
\begin{eqnarray*}
\mathbb{E}\big[(A_{i_1j_1t_1}-\eta_{i_1j_1t_1})(A_{i_2j_2t_2}-\eta_{i_2j_2t_2})^3\big|\sigma\big]=\mathbb{E}\big[(A_{i_1j_1t_1}-\eta_{i_1j_1t_1})\big|\sigma\big]\mathbb{E}\big[(A_{i_2j_2t_2}-\eta_{i_2j_2t_2})^3\big|\sigma\big]=0,
\end{eqnarray*}
which implies that either $(i_1,j_1,t_1)=(i_2,j_2,t_2)=(i_3,j_3,t_3)=(i_4,j_4,t_4)$ or $(i_{r_1},j_{r_1},t_{r_1})=(i_{r_2},j_{r_2},t_{r_2})$ and $(i_{r_3},j_{r_3},t_{r_3})=(i_{r_4},j_{r_4},t_{r_4})$ for distinct $r_1, r_2, r_3, r_4\in\{1,2,3,4\}$. 
In the general case, suppose for some $q$ with $1\leq q<v$ and $p_r\geq2$ for $1\leq r\leq q$ that  
\[(A_{i_1j_1t_1}-\eta_{i_1j_1t_1})\dots (A_{i_vj_vt_v}-\eta_{i_vj_vt_v})=(A_{i_1j_1t_1}-\eta_{i_1j_1t_1})^{p_1}\dots (A_{i_qj_qt_q}-\eta_{i_qj_qt_q})^{p_q}.\]
Then, after relabeling the indexes, one has
 \begin{eqnarray*}
 &&\frac{n^{2s}}{\binom{n}{3}^s}\sum_{i_r<j_r<t_r,r=1,\dots,s}\Big|\mathbb{E}(A_{i_1j_1t_1}-\eta_{i_1j_1t_1})\dots (A_{i_vj_vt_v}-\eta_{i_vj_vt_v})\\
 &&\times\Big(\eta_{i_{v+1}j_{v+1}t_{v+1}}-\frac{a+(k^2-1)b}{n^2k^2}\Big)\dots \Big(\eta_{i_sj_st_s}-\frac{a+(k^2-1)b}{n^2k^2}\Big)\Big|\\
 &=&\frac{n^{2s}}{\binom{n}{3}^s}\sum_{i_r<j_r<t_r,r=1,\dots,{s-v+q}}\Big|\mathbb{E}(A_{i_1j_1t_1}-\eta_{i_1j_1t_1})^{p_1}\dots (A_{i_qj_qt_q}-\eta_{i_qj_qt_q})^{p_q}\\
 &&\times\Big(\eta_{i_{q+1}j_{q+1}t_{q+1}}-\frac{a+(k^2-1)b}{n^2k^2}\Big)\dots \Big(\eta_{i_{s-v+q}j_{s-v+q}t_{s-v+q}}-\frac{a+(k^2-1)b}{n^2k^2}\Big)\Big|\\
&=& \frac{n^{2s}}{\binom{n}{3}^s}\sum_{i_r<j_r<t_r,r=1,\dots,{s-v+q}}\Big|\mathbb{E}\eta_{i_1j_1t_1}\dots \eta_{i_qj_qt_q}
\Big(\eta_{i_{q+1}j_{q+1}t_{q+1}}-\frac{a+(k^2-1)b}{n^2k^2}\Big)\times\\
&&\dots \times\Big(\eta_{i_{s-v+q}j_{s-v+q}t_{s-v+q}}-\frac{a+(k^2-1)b}{n^2k^2}\Big)\Big|\\
 &\leq&\frac{n^{2s}}{\binom{n}{3}^s}\Big[\binom{n}{3}^{s-v+q}\frac{C_1^q}{n^{2q}}\frac{C_1^{s-v}}{n^{2(s-v)}}\Big]=\frac{(3!)^{v-q}C_1^{s-v+q}}{n^{v-q}}\leq \frac{C_2^s(2\xi_n)^{2\xi_n}}{n}.
 \end{eqnarray*}
  
 Hence, by (\ref{ab0}) and (\ref{ab1}) and for some large constant $C_3>1$, we conclude that
$\mathbb{E}(\widehat{\lambda}_3 -\lambda_3)^{s}\leq 2^{s}\frac{C_2^s(2\xi_n)^{2\xi_n}}{n}$  and
  \begin{eqnarray*}
 \mathbb{E}(\widehat{\lambda}_3^{\xi_n}-\lambda_3^{\xi_n})^2&\leq&\xi_n^2\xi_n^{2\xi_n}\lambda_3^{2\xi_n}2^{\xi_n}\frac{C_2^{\xi_n}(2\xi_n)^{2\xi_n}}{n}\leq \frac{(C_3\xi_n)^{C_3\xi_n}}{n}.
 \end{eqnarray*}
Let $N_n=C_3\xi_n\rightarrow\infty$, then $n=\gamma^{\lambda_3^{\frac{N_n}{\delta_0C_3}}}$. For large $N_n$, it holds that $\lambda_3^{\frac{N_n}{\delta_0C_3}}\geq C_4N_n^2$ for some constant $C_4>0$, which implies that
   \begin{equation*}
 \mathbb{E}(\widehat{\lambda}_3^{\xi_n}-\lambda_3^{\xi_n})^2\leq \frac{(C_3\xi_n)^{C_3\xi_n}}{n}=\frac{N_n^{N_n}}{\gamma^{\lambda_3^{\frac{N_n}{\delta_0C_3}}}}\leq\Big(\frac{N_n}{\gamma^{C_4N_n}}\Big)^{N_n}\rightarrow0,
 \end{equation*}
leading to $\widehat{\lambda}_3^{\xi_n}-\lambda_3^{\xi_n}=o_p(1)$. 

Now we conclude $2\xi_nX_{\xi_n}-\widehat{\lambda}_m^{\xi_n}=(1+o_p(1))(k-1)f^{\xi_n}$, which implies that $\widehat{f}=f+o_p(1)$. Since $\widehat{\lambda}_m$ and $\widehat{f}$ are consistent estimators of $\lambda_m$ and $f$, then $\widehat{a}_n$ and $\widehat{b}_n$ are consistent estimators of $a$ and $b$, respectively.

\subsection{Proof of Theorem \ref{ctgbounded}.}
Before proving Theorem \ref{ctgbounded}, we need several preliminary results,
i.e., Lemmas \ref{trace}, \ref{sumtrace}, \ref{cyclepois}, \ref{matrixde} and Proposition \ref{basic:prop}.

\begin{Lemma}\label{trace} 
Let $M_0$ be the following $k\times k$ matrix
\[M_0=
\begin{bmatrix}
a+(k^{m-2}-1)b & k^{m-2}b & \dots & k^{m-2}b \\
 k^{m-2}b &a+(k^{m-2}-1)b & \dots & k^{m-2}b \\
 \vdots & \vdots & \dots & \vdots\\
 k^{m-2}b & k^{m-2}b &\dots &a+(k^{m-2}-1)b
\end{bmatrix}.
\]
Then the trace of $M_0^j$ is 
\[Tr(M_0^j)=(a+(k^{m-2}-1)b)^j+(k-1)(a-b)^j,\]
for any positive integer $j$.
\end{Lemma}

\begin{Lemma}\label{sumtrace}
For any $1\le i_1,\ldots,i_m\le k$,
let $M_{i_1i_2\dots i_m}=(a-b)I[i_1=i_2=\dots=i_m]+b$.
If $j\ge 1$ and $1\le i_1,\dots,i_{jm-j}\le k$, then we have
\[\sum_{i_1,\dots,i_{jm-j}=1}^kM_{i_1i_2\dots i_m}M_{i_m\dots i_{2m-1}}M_{i_{2m-1}\dots i_{3m-2}}\dots M_{i_{(j-1)m-(j-2)}\dots i_{jm-j}i_1}=Tr(M_0^j),\]
where $M_0$ is the same as in Lemma \ref{trace}.
\end{Lemma}

\begin{Lemma}\label{cyclepois}
For any $h\ge2$, let $X_{hn}$ be the number of $h$-hyperedge loose cycles in $\mathcal{H}_{m}(n,\frac{d}{n^{m-1}})$, where $d=\frac{a+(k^{m-1}-1)b}{k^{m-1}}$. 
Then for any integer $s\geq 2$, $\{X_{hn}\}_{h=2}^s$ jointly converge to independent Poisson variables with means $\lambda_h=\frac{d^h}{2h[(m-2)!]^h}$.
\end{Lemma}

The following proposition is useful to prove Theorem \ref{ctgbounded}.
For any non-negative integer $x$, let
$[x]_j$ denote the product $x(x-1)\cdots(x-j+1)$.
\begin{Proposition}[Janson, 1995]\label{basic:prop}
	Let $\lambda_i>0$, $i=1,2,\ldots$, be constants and suppose that for each $n$ there are random variables
	$X_{in}$, $i=1,2,\ldots$, and $Y_n$ (defined on the same probability space) such that $X_{in}$ is non-negative integer valued
	and $\mathbb{E}\{Y_n\}\neq0$ (at least for large $n$), and furthermore the following conditions are satisfied:
	\begin{itemize}
		\item[(A1)] $X_{in}\overset{d}{\to}Z_i$ as $n\to\infty$, jointly for all $i$, where $Z_i\sim\textrm{Poisson}(\lambda_i)$
		are independent Poisson random variables;
		\item[(A2)] $\mathbb{E}\{Y_n[X_{1n}]_{j_1}\cdots[X_{kn}]_{j_k}\}/\mathbb{E}\{Y_n\}\to\prod_{i=1}^k\mu_i^{j_i}$, as $n\to\infty$,
		for some $\mu_i\ge0$ and every finite sequence $j_1,\ldots,j_k$ of non-negative integers;
		\item[(A3)] $\sum_{i=1}^\infty\lambda_i\delta_i^2<\infty$, where $\delta_i=\mu_i/\lambda_i-1$;
		\item[(A4)] $\mathbb{E}\{Y_n^2\}/(\mathbb{E}\{Y_n\})^2\to\exp\left(\sum_{i=1}^\infty\lambda_i\delta_i^2\right)$.
	\end{itemize}
	Then 
	\[
	\frac{Y_n}{\mathbb{E}\{Y_n\}}\overset{d}{\to}W\equiv\prod_{i=1}^\infty(1+\delta_i)^{Z_i}\exp(-\lambda_i\delta_i),\,\,\,\,\textrm{as $n\to\infty$},
	\]
    and $\mathbb{E}W=1$.
\end{Proposition}

For $u=1,\ldots,n$, let $\widetilde{\sigma}_u=(1_{[\sigma_u=1]}, \dots, 1_{[\sigma_u=k]})^T$, 
$\widetilde{\tau}_u=(1_{[\tau_u=1]}, \dots, 1_{[\tau_u=k]})^T$. 
Clearly, $\widetilde{\sigma}_u,\widetilde{\tau}_u\sim Multinomial(1, k, p)$ with $p=\frac{1}{k}$. Let $C$ be a $(k^2+2k)\times(k^2+2k)$ diagonal matrix, 
with the first $2k$ diagonal elements $c_1$, the last $k^2$ diagonal elements $c_2$. Let 
\[
\tilde{\rho}=(\tilde{\rho}_{10},\dots, \tilde{\rho}_{s0}, \tilde{\rho}_{01},\dots,\tilde{\rho}_{0s},\tilde{\rho}_{11}, \tilde{\rho}_{12},\dots, \tilde{\rho}_{ss})^T.
\]
Then $Z_n=\tilde{\rho}C\tilde{\rho}^T$. By central limit theorem, $\tilde{\rho}$ converges to $N(0,\Sigma)$, where $\Sigma$ is the covariance matrix of $(\widetilde{\sigma}_u^T,
\widetilde{\tau}_u^T,\widetilde{\sigma}_u^T\otimes\widetilde{\tau}_u^T)^T$.

\begin{Lemma}\label{matrixde}
The covariance matrix of $(\widetilde{\sigma}_u^T,
\widetilde{\tau}_u^T,\widetilde{\sigma}_u^T\otimes\widetilde{\tau}_u^T)^T$ has the following expression:
\[\Sigma=
\begin{bmatrix}
V				& 0					& V\otimes{\bf p}^T\\
0				& V					& {\bf p}^T\otimes V\\
V\otimes{\bf p}	& {\bf p}\otimes V	& V_2
\end{bmatrix},
\]
where $V=Var(\widetilde\sigma_u)=pI_{k}-p^2J_{k}$, ${\bf p}=E(\widetilde\sigma_u)$, $V_2=p^2I_{k^2}-p^4J_{k^2}$, $J_{k^2}$ is an $k^2\times k^2$ order matrix with all elements 1. Besides, $V^2=pV$, $V_2^2=p^2V_2$. Let
\[R=
\begin{bmatrix}
I_{k}	&0		&-I_{k}\otimes{\bf p}^T\\
0	&I_{k}		&-{\bf p}^T\otimes I_{k}\\
0	&0		&I_{k^2}
\end{bmatrix},\ 
\Lambda=
\begin{bmatrix}
V	&0	&0	\\
0	&V	&0	\\
0	&0	&\Omega_2
\end{bmatrix},\  
\Lambda_1=
\begin{bmatrix}
\frac{1}{\sqrt{p}}V		&0						&0	\\
0						&\frac{1}{\sqrt{p}}V	&0	\\
0						&0						&\frac{1}{p}\Omega_2
\end{bmatrix},
A=
\begin{bmatrix}
c_1I_{k}	&0		&0	\\
0		&c_1I_{k}	&0	\\
0		&0		&c_2I_{k^2}
\end{bmatrix}
\]
where $\Omega_2=V_2-p^2V\otimes J_{k}-p^2J_{k}\otimes V$ with $\Omega_2^2=p^2\Omega_2$.
Then $R^T\Sigma R=\Lambda$ and 
\[
R^{-1}=
\begin{bmatrix}
I_{k}	&0		&I_{k}\otimes{\bf p}^T\\
0	&I_{k}		&{\bf p}^T\otimes I_{k}\\
0	&0		&I_{k^2}
\end{bmatrix},\ \ 
\Lambda_1R^{-1}A(R^{-1})^T\Lambda_1=
\begin{bmatrix}
0	&0	&0	\\
0	&0	&0	\\
0	&0	&c_2\Omega_2
\end{bmatrix}.
\]
Hence, $Z_n\rightarrow c_2p^2\chi^2_{(k-1)^2}$. Furthermore, $\{\exp(Z_n)\}_{n=1}^{\infty}$ is uniformly integrable if $\kappa (k-1)^2<1$.
\end{Lemma}
\begin{proof}[Proof of Theorem \ref{ctgbounded}]
We check the conditions of Proposition \ref{basic:prop}.
Let $\lambda_h=\frac{1}{2h}\Big(\frac{a+(k^{m-1}-1)b}{k^{m-1}(m-2)!}\Big)^h$ and $\delta_h=(k-1)\Big(\frac{a-b}{a+(k^{m-1}-1)b}\Big)^h$. Condition (A1) follows from Lemma \ref{cyclepois}.

Next, we check condition (A2). Let $S=\{1,2,\dots,k\}$ and $H=(H_{hi})_{2\leq h\le s,1\le i\le j_s}$ be a $sj_s$-tuple of $h$-edge loose cycle $H_{hi}$ for any integers $s(\ge2)$ and $j_s(\ge1)$. Define $X_{hn}$ as the number of $h$-edge loose cycles in the hypergraph and $[x]_j=x(x-1)\dots (x-j+1)$. Note that for any sequence of positive integers $j_2$,$\dots$, $j_s$, we have
\begin{equation}\label{eyx}
\mathbb{E}_0Y_n[X_{2n}]_{j_2}\dots[X_{sn}]_{j_s}=\sum_{H\in B}\mathbb{E}_0Y_n1_H+\sum_{H\in\overline{B}}\mathbb{E}_0Y_n1_H,
\end{equation}
where $B$ is the collection of disjoint tuples $H$ and $\overline{B}$ is the complement, that is, any two tuples $H_1$ and $H_2$ in $\overline{B}$ have at least one vertex in common.
Direct computation yields
\begin{eqnarray*}
	\mathbb{E}_0Y_n1_H&=&\frac{1}{k^n}\sum_{\sigma\in S^n}\mathbb{E}_01_H\prod_{i\in c(m,n)}\left(\frac{p_{i_1:i_m}(\sigma)}{p_0}\right)^{A_{i_1:i_m}}\left(\frac{q_{i_1:i_m}(\sigma)}{q_0}\right)^{1-A_{i_1:i_m}}\\
	&=&\frac{1}{k^n}\sum_{\sigma\in S^n}\mathbb{E}_01_H\prod_{(i_1,\dots,i_m)\in\mathcal{E}(H)}\left(\frac{p_{i_1:i_m}(\sigma)}{p_0}\right)^{A_{i_1:i_m}}\left(\frac{q_{i_1:i_m}(\sigma)}{q_0}\right)^{1-A_{i_1:i_m}}\\
&&\times \mathbb{E}_0\prod_{(i_1,\dots,i_m)\not\in\mathcal{E}(H)}\left(\frac{p_{i_1:i_m}(\sigma)}{p_0}\right)^{A_{i_1:i_m}}\left(\frac{q_{i_1:i_m}(\sigma)}{q_0}\right)^{1-A_{i_1:i_m}}\\
	&=&\frac{1}{k^n}\sum_{\sigma\in S^n}\mathbb{E}_01_H\prod_{(i_1,\dots,i_m)\in\mathcal{E}(H)}\left(\frac{p_{i_1:i_m}(\sigma)}{p_0}\right)^{A_{i_1:i_m}}\left(\frac{q_{i_1:i_m}(\sigma)}{q_0}\right)^{1-A_{i_1:i_m}},
\end{eqnarray*}	
where the second equality follows by the independence of $A_{i_1:i_m}$.  Define $\sigma^{1hi}$ and $\sigma^{2hi}$ to be the restrictions of $\sigma$ on $\mathcal{V}(H_{hi})$ and $[n]\backslash  \mathcal{V}(H_{hi})$. Similarly, $\sigma^{1}$ and $\sigma^{2}$ are the restrictions of $\sigma$ on $\mathcal{V}(H)$ and $[n]\backslash \mathcal{V}(H)$. Then by the above equation, we have
\begin{eqnarray*}
\mathbb{E}_0Y_n1_H&=&\frac{1}{k^n}\sum_{\sigma^1\in S^{|\mathcal{V}(H)|}}\sum_{\sigma^2\in S^{[n]/|\mathcal{V}(H)|}}\mathbb{E}_01_H\prod_{(i_1,\dots,i_m)\in\mathcal{E}(H)}\left(\frac{p_{i_1:i_m}(\sigma)}{p_0}\right)^{A_{i_1:i_m}}\left(\frac{q_{i_1:i_m}(\sigma)}{q_0}\right)^{1-A_{i_1:i_m}}\\
	&=&k^{-|\mathcal{V}(H)|}\sum_{\sigma^1\in S^{|\mathcal{V}(H)|}}\mathbb{E}_01_H\prod_{(i_1,\dots,i_m)\in\mathcal{E}(H)}
	\left(\frac{p_{i_1:i_m}(\sigma^1)}{p_0}\right)^{A_{i_1:i_m}}\left(\frac{q_{i_1:i_m}(\sigma^1)}{q_0}\right)^{1-A_{i_1:i_m}}.
\end{eqnarray*}
Since $A_{i_1:i_m}=1$ for $(i_1,\ldots,i_m)\in\mathcal{E}(H)$, the above equals
\begin{eqnarray*}
	&&k^{-M_1}p_0^{M_1}\sum_{\sigma^1\in S^{|\mathcal{V}(H)|}}\prod_{(i_1,\dots,i_m)\in\mathcal{E}(H)}\frac{p_{i_1:i_m}(\sigma^1)}{p_0}
	=\mathbb{E}_{\sigma^1}\prod_{(u,v)\in\mathcal{E}(H)}p_{i_1:i_m}(\sigma^1)\\
	&=&\prod_{h=2}^s\prod_{i=1}^{j_h}\mathbb{E}_{\sigma^{1hi}}\prod_{(i_1,\dots,i_m)\in\mathcal{E}(H^{hi})}p_{i_1:i_m}(\sigma^{1hi})
	=\prod_{h=2}^s\prod_{i=1}^{j_h}\mathbb{E}_{\sigma^{1hi}}\prod_{(i_1,\dots,i_m)\in\mathcal{E}(H^{hi})}\frac{M_{\sigma^{1hi}_{i_1},\dots,\sigma^{1hi}_{i_m}}}{n}\\
&=&\prod_{h=2}^s\prod_{i=1}^{j_h}E_{\tau_{hi}}\frac{M_{\tau_{i_1}^{hi}\dots\tau_{i_m}^{hi}}M_{\tau_{i_m}^{hi}\dots\tau_{2m-1}^{hi}}\dots M_{\tau_{i_{(h-1)(m-1)}}^{hi}\dots\tau_{i_1}^{hi}}}{n^{h(m-1)}}\\
&=&\prod_{h=2}^s\prod_{i=1}^{j_h}\frac{Tr(M_0^h)}{k^{h(m-1)}n^{h(m-1)}},
\end{eqnarray*}
where we used Lemma \ref{sumtrace} for the last equality. Note $\#B=\frac{n!}{(n-M_1)}\prod_{h=2}^k(\frac{1}{2h(m-2)!^h})^{j_h}$, 
where $M_1=(m-1)\sum_{h=2}^shj_h$. Hence,
 by Lemma \ref{trace}, the first term in (\ref{eyx}) is

\begin{eqnarray*}
\#B\times\prod_{h=2}^s\prod_{i=1}^{j_h}\frac{Tr(M_0^h)}{k^{h(m-1)}n^{h(m-1)}}&=&\frac{n!}{(n-M_1)!n^{M_1}}\prod_{h=2}^s\left[\frac{1}{2h(m-2)!^h}\left(d^h+\frac{(k-1)(a-b)^h}{k^{m-1}}\right)\right]^{j_h}\\
&=&\frac{n!}{(n-M_1)!n^{M_1}}\prod_{h=2}^s\left[\lambda_h(1+\delta_h)\right]^{j_h}\rightarrow\prod_{h=2}^s\left[\lambda_h(1+\delta_h)\right]^{j_h}.
\end{eqnarray*}

For $H\in\overline{B}$, one has
\begin{eqnarray*}
	\mathbb{E}_0Y_n1_H&=&k^{-n}\sum_{\sigma\in S^n}\mathbb{E}_01_H\prod_{(i_1,\dots,i_m)\in\mathcal{E}(H)}\left(\frac{p_{i_1:i_m}(\sigma)}{p_0}\right)^{A_{i_1:i_m}}\left(\frac{q_{i_1:i_m}(\sigma)}{q_0}\right)^{1-A_{i_1:i_m}}\\
	&=&k^{-n}\sum_{\sigma\in S^n}\left(\prod_{(i_1,\dots,i_m)\in\mathcal{E}(H)}\frac{p_{i_1:i_m}(\sigma)}{p_0}\right)P_0(H)\\
	&\leq&k^{-n}p_0^{|\mathcal{E}(H)|}\sum_{\sigma\in S^n}p_0^{-|\mathcal{E}(H)|}\left(\frac{a}{n^{m-1}}\right)^{|\mathcal{E}(H)|}
	=\left(\frac{a}{n^{m-1}}\right)^{|\mathcal{E}(H)|}.
\end{eqnarray*}
Then it follows that
\[\sum_{\text{
$H'$ isomorphic to $H$}}\mathbb{E}_0Y_n1_H\leq\left(\frac{a}{n^{m-1}}\right)^{|\mathcal{E}(H)|}\binom{n}{|\mathcal{V}(H)|}|\mathcal{V}(H)|!\rightarrow 0,\]
and
$\sum_{H\in\overline{B}}\mathbb{E}_0Y_n1_H\rightarrow0$.
Hence,
$\mathbb{E}_0Y_n[X_{2n}]_{j_2}\dots[X_{sn}]_{j_s}\rightarrow\prod_{h=2}^s[\lambda_h(1+\delta_h)]^{j_h}$.

Then we check condition (A3). By (A1) and (A2), we have $\frac{\mu_h}{\lambda_h}-1=\frac{\lambda_h(1+\delta_h)}{\lambda_h}-1=\delta_h$. Besides, $\lambda_h\delta_h^2=\frac{1}{2h}\Big(\frac{(a-b)^2}{k^{m-1}(m-2)!(a+(k^{m-1}-1)b)}\Big)^h=\frac{\kappa^h}{2h}$. If $\kappa<1$, then $\sum_{h=2}^{\infty}\lambda_h\delta_h^2<\infty$.

Lastly, we check condition (A4). Note that 
\begin{eqnarray*}
\mathbb{E}_0Y_n^2&=&(1+o(1))\exp\Big\{\frac{1}{n^{m-1}d}\Big(\binom{n}{m}(b-d)^2+(a-b)^2\sum_{i\in c(m,n)}I[\sigma_{i_1}:\sigma_{i_m}]I[\eta_{i_1}:\eta_{i_m}]\\
&&+(a-b)(b-d)(\sum_{i\in c(m,n)}I[\sigma_{i_1}:\sigma_{i_m}]+\sum_{i\in c(m,n)}I[\eta_{i_1}:\eta_{i_m}])\Big)\Big\}.
\end{eqnarray*}
Let $C=\{(i_1,\dots,i_m)|\exists i_s,i_t: i_s=i_t, i_{s'}\neq i_{t'} \ if\  s', t'\notin \{s,t\}\}$. Then
\[\sum_{i_1,i_2,\dots,i_m}I[\sigma_{i_1}:\sigma_{i_m}]I[\eta_{i_1}:\eta_{i_m}]=m!\sum_{i\in c(m,n)}I[\sigma_{i_1}:\sigma_{i_m}]I[\eta_{i_1}:\eta_{i_m}]+\sum_{C}I[\sigma_{i_1}:\sigma_{i_m}]I[\eta_{i_1}:\eta_{i_m}]+O(n^{m-2}).\]
Direct computation yields
\begin{eqnarray*}
&&\sum_{i\in c(m,n)}I[\sigma_{i_1}:\sigma_{i_m}]I[\eta_{i_1}:\eta_{i_m}]\\
&=&\frac{1}{m!}\sum_{i_1,i_2,\dots,i_m}I[\sigma_{i_1}:\sigma_{i_m}]I[\eta_{i_1}:\eta_{i_m}]-\frac{1}{m!}\sum_{C}I[\sigma_{i_1}:\sigma_{i_m}]I[\eta_{i_1}:\eta_{i_m}]+O(n^{m-2})\\
&=&\frac{1}{m!}\sum_{s,t=1}^k(\sqrt{n}\tilde{\rho}_{st}+\frac{n}{k^2})^m-\frac{1}{m!}\binom{m}{2}\sum_{s,t=1}^k(\sqrt{n}\tilde{\rho}_{st}+\frac{n}{k^2})^{m-1}+O(n^{m-2})
\end{eqnarray*}
\begin{eqnarray*}
&=&\frac{1}{m!}\frac{n^m}{k^{2m-2}}+\frac{1}{m!}\frac{\binom{m}{2}n^{m-1}}{k^{2(m-2)}}\sum_{s,t=1}^k\tilde{\rho}_{st}^2\Big[1+\sum_{i=1}^{m-2}\frac{1}{k^{2i}}\frac{\binom{m}{i+2}}{\binom{m}{2}}\Big(\frac{\tilde{\rho}_{st}}{\sqrt{n}}\Big)^i\Big]\\
&&-\frac{\binom{m}{2}}{m!}\frac{k^2n^{m-1}}{k^{2(m-1)}}-\frac{\binom{m}{2}n^{m-1}}{m!}\sum_{s,t=1}^k\sum_{i=1}^{m-1}\binom{m-1}{i}\frac{1}{k^{2(m-1-i)}}\Big(\frac{\tilde{\rho}_{st}}{\sqrt{n}}\Big)^i+O(n^{m-2}).
\end{eqnarray*}
Similarly, one gets
\begin{eqnarray*}
\sum_{i\in c(m,n)}I[\sigma_{i_1}:\sigma_{i_m}]
&=&\frac{1}{m!}\frac{n^m}{k^{m-1}}+\frac{1}{m!}\frac{\binom{m}{2}n^{m-1}}{k^{(m-2)}}\sum_{s=1}^k\tilde{\rho}_{s0}^2\Big[1+\sum_{i=1}^{m-2}\frac{1}{k^{i}}\frac{\binom{m}{i+2}}{\binom{m}{2}}\Big(\frac{\tilde{\rho}_{s0}}{\sqrt{n}}\Big)^i\Big]\\
&&-\frac{\binom{m}{2}}{m!}\frac{kn^{m-1}}{k^{(m-1)}}-\frac{\binom{m}{2}n^{m-1}}{m!}\sum_{s=1}^k\sum_{i=1}^{m-1}\binom{m-1}{i}\frac{1}{k^{(m-1-i)}}\Big(\frac{\tilde{\rho}_{s0}}{\sqrt{n}}\Big)^i+O(n^{m-2})
\end{eqnarray*}
and
\begin{eqnarray*}
\sum_{i\in c(m,n)}I[\eta_{i_1}:\eta_{i_m}]
&=&\frac{1}{m!}\frac{n^m}{k^{m-1}}+\frac{1}{m!}\frac{\binom{m}{2}n^{m-1}}{k^{(m-2)}}\sum_{t=1}^k\tilde{\rho}_{0t}^2\Big[1+\sum_{i=1}^{m-2}\frac{1}{k^{i}}\frac{\binom{m}{i+2}}{\binom{m}{2}}\Big(\frac{\tilde{\rho}_{0t}}{\sqrt{n}}\Big)^i\Big]\\
&&-\frac{\binom{m}{2}}{m!}\frac{kn^{m-1}}{k^{(m-1)}}-\frac{\binom{m}{2}n^{m-1}}{m!}\sum_{t=1}^k\sum_{i=1}^{m-1}\binom{m-1}{i}\frac{1}{k^{(m-1-i)}}\Big(\frac{\tilde{\rho}_{0t}}{\sqrt{n}}\Big)^i+O(n^{m-2}).
\end{eqnarray*}
Note that
$\binom{n}{m}=\frac{n^m}{m!}-\frac{\binom{m}{2}}{m!}n^{m-1}+O(n^{m-2})$ and

\[\frac{n^m}{m!}\Big(\frac{(a-b)^2}{k^{2(m-2)}}+\frac{2(a-b)(b-d)}{k^{m-1}}+(b-d)^2\Big)=\frac{n^m}{m!}\Big(\frac{a-b}{k^{m-1}}+(b-d)\Big)^2=0,\]

\[\frac{\binom{m}{2}n^{m-1}}{m!}\Big(\frac{k^2(a-b)^2}{k^{2(m-1)}}+\frac{2k(a-b)(b-d)}{k^{m-1}}+(b-d)^2\Big)=\frac{\binom{m}{2}n^{m-1}}{m!}\frac{(k-1)^2(a-b)^2}{k^{2(m-1)}}.\]
Let $c_1=\frac{\binom{m}{2}}{m!d}\frac{(a-b)(b-d)}{k^{m-2}}$ and $c_2=\frac{\binom{m}{2}}{m!d}\frac{(a-b)^2}{k^{2(m-2)}}$. Since $|\frac{\tilde{\rho}_{st}}{\sqrt{n}}|\leq1$, $|\frac{\tilde{\rho}_{0t}}{\sqrt{n}}|\leq1$, $|\frac{\tilde{\rho}_{t0}}{\sqrt{n}}|\leq1$ and $|\frac{\tilde{\rho}_{st}}{\sqrt{n}}|\rightarrow0$, $|\frac{\tilde{\rho}_{0t}}{\sqrt{n}}|\rightarrow0$, $|\frac{\tilde{\rho}_{t0}}{\sqrt{n}}|\rightarrow0$ in probability. Hence, 
\begin{eqnarray*}
\tilde{Z}_n&=&c_2\sum_{s,t=1}^k\tilde{\rho}_{st}^2\Big[1+\sum_{i=1}^{m-2}\frac{1}{k^{2i}}\frac{\binom{m}{i+2}}{\binom{m}{2}}\Big(\frac{\tilde{\rho}_{st}}{\sqrt{n}}\Big)^i\Big]\\
& &+c_1\Big(\sum_{t=1}^k\tilde{\rho}_{0t}^2\Big[1+\sum_{i=1}^{m-2}\frac{1}{k^{i}}\frac{\binom{m}{i+2}}{\binom{m}{2}}\Big(\frac{\tilde{\rho}_{0t}}{\sqrt{n}}\Big)^i\Big]+\sum_{s=1}^k\tilde{\rho}_{s0}^2\Big[1+\sum_{i=1}^{m-2}\frac{1}{k^{i}}\frac{\binom{m}{i+2}}{\binom{m}{2}}\Big(\frac{\tilde{\rho}_{s0}}{\sqrt{n}}\Big)^i\Big]\Big)
\end{eqnarray*}
and $Z_n=c_2\sum_{s,t=1}^k\tilde{\rho}_{st}^2+c_1\Big(\sum_{t=1}^k\tilde{\rho}_{0t}^2+\sum_{s=1}^k\tilde{\rho}_{s0}^2\Big)$ are asymptotically equivalent. 

If $\tau_1(m,k)\leq1$, then $1+\sum_{i=1}^{m-2}\frac{1}{k^{i}}\frac{\binom{m}{i+2}}{\binom{m}{2}}\Big(\frac{\tilde{\rho}_{s0}}{\sqrt{n}}\Big)^i\geq0$, hence 
\[\tilde{Z}_n\leq c_2\sum_{s,t=1}^k\tilde{\rho}_{st}^2\Big[1+\sum_{i=1}^{m-2}\frac{1}{k^{2i}}\frac{\binom{m}{i+2}}{\binom{m}{2}}\Big(\frac{\tilde{\rho}_{st}}{\sqrt{n}}\Big)^i\Big]\leq c_2\tau_2(m)\sum_{s,t=1}^k\tilde{\rho}_{st}^2.\]
Let $f_j=\frac{1}{\sqrt{n}}\sum_{u=1}^j\Big(\left(1_{[\sigma_u=1]}1_{[\eta_u=1]}-\frac{1}{k^2}\right),\dots, \left(1_{[\sigma_u=k]}1_{[\eta_u=k]}-\frac{1}{k^2}\right)\Big)^T$ and $d_j=f_j-f_{j-1}$. Then $\|d_j\|^2=\frac{1}{n}\frac{k^2-1}{k^2}$ and $b_*^2=\sum_{j=1}^n\|d_j\|^2=\frac{k^2-1}{k^2}$. By Theorem 3.5 in Pinelis \cite{p94}, for any $t>0$,
\begin{eqnarray*}
	P\Bigg(\exp\left\{c_2\tau_2(m)\|f_n\|^2\right\}>t\Bigg)&=&P\Bigg(c_2\tau_2(m)\|f_n\|^2>\log(t)\Bigg)
	=P\left(\|f_n\|>\sqrt{\frac{\log(t)}{c_2\tau_2(m)}}\right)\\
	&\leq&2\exp\left(-\frac{\log(t)}{\kappa(k^2-1)\tau_2(m)}\right)
	=2t^{-\frac{1}{\kappa(k^2-1)\tau_2(m)}}.
\end{eqnarray*}
Hence, the condition $\kappa(k^2-1)\tau_2(m,k)<1$ implies that $\{\exp(\tilde{Z}_n)\}_{n=1}^{\infty}$ is uniformly integrable.

By Lemma \ref{matrixde},  we conclude that $Z_n$ converges to $\frac{c_2}{k^2}\chi^2_{(k-1)^2}$. Since $\kappa (k^2-1)\tau_2(m,k)<1$ implies $\kappa (k-1)^2<1$ and $\frac{c_2}{k^2}<\frac{1}{2}$, then it follows that
\begin{eqnarray*}
\mathbb{E}_0Y_n^2&\rightarrow&\exp\Big\{-\frac{\binom{m}{2}}{m!d}\frac{(k-1)^2(a-b)^2}{k^{2(m-1)}}\Big\}\mathbb{E}\exp\Big\{\frac{c_2}{k^2}\chi^2_{(k-1)^2}\Big\}\\
&=&\exp\Big\{-\frac{\binom{m}{2}}{m!d}\frac{(k-1)^2(a-b)^2}{k^{2(m-1)}}\Big\}\exp\Big\{-\frac{(k-1)^2}{2}\log\Big(1-2\frac{c_2}{k^2}\Big)\Big\}
=\exp\Big\{\sum_{h=2}^{\infty}\lambda_h\delta_h^2\Big\},
\end{eqnarray*}
where we used the fact that
\[\frac{(k-1)^2}{2}\Big(\frac{2c_2}{k^2}\Big)^h\frac{1}{h}=\frac{(k-1)^2}{2h}\Big(\frac{a+(k^{m-1}-1)b}{k^{m-1}(m-2)!}\Big)^h\Big(\frac{(a-b)^2}{(a+(k^{m-1}-1)b)^2}\Big)^h=\lambda_h\delta_h^2.\]
Obviously, $\mathbb{E}_0Y_n=1$. Hence, $H_0$ and $H_1$ are contiguous. 
\end{proof}
The proof of Theorem \ref{sharp:SNR:phase:transition} relies on the following lemma.
\begin{Lemma}\label{lemma:sharp:k=2} If $\sigma_i,\tau_i\in\{\pm\}$ for $i=1,\ldots,n$,
then it holds that
\begin{eqnarray}
&&\sum_{i_1,\ldots,i_m=1}^n\left[\prod_{l=1}^{m-1}(\sigma_{i_l}\sigma_{i_{l+1}}+1)-1\right]\left[\prod_{l=1}^{m-1}(\tau_{i_l}\tau_{i_{l+1}}+1)-1\right]\nonumber\\
&=&\sum_{\substack{2\le t,s\le m\\ 
\textrm{$t,s$ even}}}
\sum_{\max\{0,t+s-m\}\le c\le \min\{s,t\}}
\frac{m!}{c!(t-c)!(s-c)!(m-t-s+c)!}\nonumber\\
&&\times n^{m-t-s+c}(\sum_{i=1}^n\sigma_i\tau_i)^c(\sum_{i=1}^n\sigma_i)^{t-c}(\sum_{i=1}^n\tau_i)^{s-c}.
\end{eqnarray}
\end{Lemma}
 
\begin{proof}[Proof of Lemma \ref{lemma:sharp:k=2}]
Let 
\[
J_m:=\sum_{i_1,\ldots,i_m=1}^n\left[\prod_{l=1}^{m-1}(\sigma_{i_l}\sigma_{i_{l+1}}+1)-1\right]\left[\prod_{l=1}^{m-1}(\tau_{i_l}\tau_{i_{l+1}}+1)-1\right]
\]
which has an expression
\[
J_m=\sum_{i_1,\ldots,i_m=1}^n\sum_{\substack{\textrm{$t,s$ even}\\
1\le l_1<\cdots<l_t\le m\\
1\le h_1<\cdots<h_s\le m}}\sigma_{i_{l_1}}\cdots\sigma_{i_{l_t}}\tau_{i_{h_1}}\cdots\tau_{i_{h_s}}.
\]
Let $c=\#\{l_1,\ldots,l_s\}\cap\{h_1,\ldots,h_s\}$.
Hence, taking the sum over $i_1,\ldots,i_m=1,\ldots,n$, the terms of $J_m$ have a common expression
$n^{m-t-s+c}(\sum_{i=1}^n\sigma_i\tau_i)^c(\sum_{i=1}^n\sigma_i)^{t-c}(\sum_{i=1}^n\tau_i)^{s-c}$
and there are ${m\choose c}{m-c\choose t-c}{m-t\choose s-c}$ such terms.
Proof is completed.
\end{proof}
\begin{proof}[Proof of Theorem \ref{sharp:SNR:phase:transition}]

For convenience, we prove Part (\ref{thm:sharp:i}) first.

\textit{Proof of Part (\ref{thm:sharp:i}).}
For simplicity, assume that the random labels $\sigma_i$ are iid uniformly distributed over $\{\pm\}$,
similar as \cite{MNS15}. The likelihood ratio can be explicitly written as
\[
Y_n=2^{-n}\sum_{\tau\in\{\pm\}^n}\prod_{i_1<\cdots<i_m}\left(\frac{p_{i_1:i_m}(\tau)}{p_0}\right)^{A_{i_1:i_m}}\left(\frac{q_{i_1:i_m}(\tau)}{q_0}\right)^{1-A_{i_1:i_m}}.
\]
Recall that
\[
p_{i_1:i_m}(\sigma)=\left(\frac{a}{n^{m-1}}\right)^{1_{[\sigma_{i_1}=\cdots=\sigma_{i_m}]}}\left(\frac{b}{n^{m-1}}\right)^{1-1_{[\sigma_{i_1}=\cdots=\sigma_{i_m}]}},
q_{i_1:i_m}(\sigma)=1-p_{i_1:i_m}(\sigma),
\]
\[
p_0=\frac{a+(2^{m-1}-1)b}{(2n)^{m-1}}, q_0=1-p_0.
\]
It turns out that when $m\ge3$ the second moment of $Y_n$ under $H_0$ is asymptotically tricky so the second moment method considered by \cite{MNS15} doesn't work.
Instead we will use a truncation technique to show that $Y_n=O_{P_1}(1)$ leading to that $H_1$ is contiguous to $H_0$, where $P_1$ is the probability measure of $A$ under $H_1$.
With a slight abuse of notation, we will also use $P_1$ to represent the joint probability measure of $A$ and $\sigma$ under $H_1$.
Choose a large $C>0$ so that $P_1(|\sum_{i=1}^n\sigma_i|\ge C\sqrt{n})$ is small.
It suffices to show that $P_1(Y_n\ge \widetilde{C})$ is small when $\widetilde{C}$ is large.
Define $\widetilde{Y}_n=Y_n 1_{[|\sigma|\le C\sqrt{n}]}$, where $|\sigma|:=|\sum_{i=1}^n\sigma_i|$.
To proceed, consider
\begin{eqnarray*}
P_1(Y_n\ge\widetilde{C})&=&P_1(Y_n\ge\widetilde{C},\sigma\in\{\pm\}^n)\\
&\le&P_1(Y_n\ge\widetilde{C},|\sigma|\le C\sqrt{n})+P_1(|\sigma|>C\sqrt{n})\\
&\le&P_1(\widetilde{Y}_n\ge\widetilde{C})+P_1(|\sigma|>C\sqrt{n})\\
&\le&\mathbb{E}_1\widetilde{Y}_n/\widetilde{C}+P_1(|\sigma|>C\sqrt{n}).
\end{eqnarray*}
Next we will show that the first term in the last equation is small when $\widetilde{C}$ is large.
Note that
\begin{eqnarray*}
\mathbb{E}_1\widetilde{Y}_n&=&2^{-n}\sum_{\tau\in\{\pm\}^n}\mathbb{E}_1\prod_{i_1<\cdots<i_m}\left(\frac{p_{i_1:i_m}(\tau)}{p_0}\right)^{A_{i_1:i_m}}\left(\frac{q_{i_1:i_m}(\tau)}{q_0}\right)^{1-A_{i_1:i_m}}1_{[|\sigma|\le C\sqrt{n}]}\\
&=&4^{-n}\sum_{\tau\in\{\pm\}^n}\sum_{|\sigma|\le C\sqrt{n}}\mathbb{E}_1\left[\prod_{i_1<\cdots<i_m}\left(\frac{p_{i_1:i_m}(\tau)}{p_0}\right)^{A_{i_1:i_m}}\left(\frac{q_{i_1:i_m}(\tau)}{q_0}\right)^{1-A_{i_1:i_m}}\bigg|\sigma\right]\\
&=&4^{-n}\sum_{\tau\in\{\pm\}^n}\sum_{|\sigma|\le C\sqrt{n}}\prod_{i_1<\cdots<i_m}\left[\frac{p_{i_1:i_m}(\sigma)p_{i_1:i_m}(\tau)}{p_0}+
\frac{q_{i_1:i_m}(\sigma)q_{i_1:i_m}(\tau)}{q_0}\right]\\
&=&4^{-n}\sum_{\tau\in\{\pm\}^n}\sum_{|\sigma|\le C\sqrt{n}}\prod_{i_1<\cdots<i_m}\left[
\frac{1}{p_0}\left(\frac{a}{n^{m-1}}\right)^{N_{i_1:i_m}^{\sigma\tau}}\left(\frac{b}{n^{m-1}}\right)^{2-N_{i_1:i_m}^{\sigma\tau}}\right.\\
&&\left.+\frac{1}{q_0}\left(1-\frac{a}{n^{m-1}}\right)^{N_{i_1:i_m}^{\sigma\tau}}\left(1-\frac{b}{n^{m-1}}\right)^{2-N_{i_1:i_m}^{\sigma\tau}}
\right],
\end{eqnarray*}
where $N_{i_1:i_m}^{\sigma\tau}=1_{[\sigma_{i_1}=\cdots=\sigma_{i_m}]}+1_{[\tau_{i_1}=\cdots=\tau_{i_m}]}$.
Therefore,
\begin{eqnarray*}
\mathbb{E}_1\widetilde{Y}_n&=&4^{-n}\sum_{\tau\in\{\pm\}^n}\sum_{|\sigma|\le C\sqrt{n}}\prod_{N_{i_1:i_m}^{\sigma\tau}=0}\left[\frac{1}{p_0}\left(\frac{b}{n^{m-1}}\right)^2+\frac{1}{q_0}\left(1-\frac{b}{n^{m-1}}\right)^2\right]\\
&&\times\prod_{N_{i_1:i_m}^{\sigma\tau}=1}\left[\frac{1}{p_0}\left(\frac{a}{n^{m-1}}\right)\left(\frac{b}{n^{m-1}}\right)+\frac{1}{q_0}\left(1-\frac{a}{n^{m-1}}\right)\left(1-\frac{b}{n^{m-1}}\right)\right]\\
&&\times\prod_{N_{i_1:i_m}^{\sigma\tau}=2}\left[\frac{1}{p_0}\left(\frac{a}{n^{m-1}}\right)^2+\frac{1}{q_0}\left(1-\frac{a}{n^{m-1}}\right)^2\right]\\
&=&4^{-n}\sum_{\tau\in\{\pm\}^n}\sum_{|\sigma|\le C\sqrt{n}}\left[\frac{1}{p_0}\left(\frac{b}{n^{m-1}}\right)^2+\frac{1}{q_0}\left(1-\frac{b}{n^{m-1}}\right)^2\right]^{M_0^{\sigma\tau}}\\
&&\times\left[\frac{1}{p_0}\left(\frac{a}{n^{m-1}}\right)\left(\frac{b}{n^{m-1}}\right)+\frac{1}{q_0}\left(1-\frac{a}{n^{m-1}}\right)\left(1-\frac{b}{n^{m-1}}\right)\right]^{M_1^{\sigma\tau}}\\
&&\times\left[\frac{1}{p_0}\left(\frac{a}{n^{m-1}}\right)^2+\frac{1}{q_0}\left(1-\frac{a}{n^{m-1}}\right)^2\right]^{M_2^{\sigma\tau}},
\end{eqnarray*}
where $M_t^{\sigma\tau}=\#\{i_1<\cdots<i_m: N_{i_1:i_m}^{\sigma\tau}=t\}$ for $t=0,1,2$.
Straightforward calculations lead to that
\begin{eqnarray}
\frac{1}{p_0}\left(\frac{b}{n^{m-1}}\right)^2+\frac{1}{q_0}\left(1-\frac{b}{n^{m-1}}\right)^2&=&1+\gamma+O(n^{-3(m-1)}),\nonumber\\
\frac{1}{p_0}\left(\frac{a}{n^{m-1}}\right)\left(\frac{b}{n^{m-1}}\right)+\frac{1}{q_0}\left(1-\frac{a}{n^{m-1}}\right)\left(1-\frac{b}{n^{m-1}}\right)
&=&1-(2^{m-1}-1)\gamma+O(n^{-3(m-1)}),\nonumber\\
\frac{1}{p_0}\left(\frac{a}{n^{m-1}}\right)^2+\frac{1}{q_0}\left(1-\frac{a}{n^{m-1}}\right)^2&=&1+(2^{m-1}-1)^2\gamma+O(n^{-3(m-1)}),\label{some:taylor:app}
\end{eqnarray}
where 
\[
\gamma=\frac{(m-2)!\kappa}{n^{m-1}}+\frac{1}{n^{2(m-1)}}\left(\frac{a-b}{2^{m-1}}\right)^2,
\]
and $\kappa$ is the SNR defined in (\ref{def:SNR}) with $k=2$ therein.
Dropping the $O(n^{-3(m-1)})$ terms in (\ref{some:taylor:app}), we get 
\begin{eqnarray}
\mathbb{E}_1\widetilde{Y}_n&\sim&4^{-n}\sum_{\tau\in\{\pm\}^n}\sum_{|\sigma|\le C\sqrt{n}}(1+\gamma)^{M_0^{\sigma\tau}}
(1-(2^{m-1}-1)\gamma)^{M_1^{\sigma\tau}}(1+(2^{m-1}-1)^2\gamma)^{M_2^{\sigma\tau}}\nonumber\\
&\sim&4^{-n}\sum_{\tau\in\{\pm\}^n}\sum_{|\sigma|\le C\sqrt{n}}\exp\left(\gamma(M_0^{\sigma\tau}-(2^{m-1}-1)M_1^{\sigma\tau}+(2^{m-1}-1)^2M_2^{\sigma\tau})\right)\nonumber\\
&=&\mathbb{E}_{\tau\sigma}\left[\exp\left(\gamma(M_0^{\sigma\tau}-(2^{m-1}-1)M_1^{\sigma\tau}+(2^{m-1}-1)^2M_2^{\sigma\tau})\right)1_{|\sigma|\le C\sqrt{n}}\right].\label{eqn:1:k=2}
\end{eqnarray}
Since $M_0^{\sigma\tau}+M_1^{\sigma\tau}+M_2^{\sigma\tau}={n\choose n}$, we get
\[
M_0^{\sigma\tau}-(2^{m-1}-1)M_1^{\sigma\tau}+(2^{m-1}-1)^2M_2^{\sigma\tau}=
-(2^{m-1}-1){n\choose m}+2^{m-1}(M_0^{\sigma\tau}-M_2^{\sigma\tau})+2^{2(m-1)}M_2^{\sigma\tau}.
\]
Rewrite 
\begin{eqnarray*}
M_0^{\sigma\tau}&=&\sum_{i_1<\cdots<i_m}(1-1_{[\sigma_{i_1}=\cdots=\sigma_{i_m}]})(1-1_{[\tau_{i_1}=\cdots=\tau_{i_m}]})\\
M_2^{\sigma\tau}&=&\sum_{i_1<\cdots<i_m}1_{[\sigma_{i_1}=\cdots=\sigma_{i_m}]}1_{[\tau_{i_1}=\cdots=\tau_{i_m}]}.
\end{eqnarray*}
So
\begin{eqnarray*}
&&M_0^{\sigma\tau}-(2^{m-1}-1)M_1^{\sigma\tau}+(2^{m-1}-1)^2M_2^{\sigma\tau}\\
&=&-(2^{m-1}-1){n\choose m}+2^{m-1}\left[{n\choose m}-\sum_{i_1<\cdots<i_m}1_{[\sigma_{i_1}=\cdots=\sigma_{i_m}]}-\sum_{i_1<\cdots<i_m}1_{[\tau_{i_1}=\cdots=\tau_{i_m}]}\right]
\\
&&+2^{2(m-1)}\sum_{i_1<\cdots<i_m}1_{[\sigma_{i_1}=\cdots=\sigma_{i_m}]}1_{[\tau_{i_1}=\cdots=\tau_{i_m}]}\\
&=&{n\choose m}-2^{m-1}\sum_{i_1<\cdots<i_m}1_{[\sigma_{i_1}=\cdots=\sigma_{i_m}]}-2^{m-1}\sum_{i_1<\cdots<i_m}1_{[\tau_{i_1}=\cdots=\tau_{i_m}]}\\
&&+2^{2(m-1)}\sum_{i_1<\cdots<i_m}1_{[\sigma_{i_1}=\cdots=\sigma_{i_m}]}1_{[\tau_{i_1}=\cdots=\tau_{i_m}]},
\end{eqnarray*}
and by $1_{[\sigma_{i_1}=\cdots=\sigma_{i_m}]}=2^{-(m-1)}\prod_{l=1}^{m-1}(\sigma_{i_l}\sigma_{i_{l+1}}+1)$, the above is equal to
\begin{eqnarray*}
&&{n\choose m}-\sum_{i_1<\cdots<i_m}\prod_{l=1}^{m-1}(\sigma_{i_l}\sigma_{i_{l+1}}+1)-\sum_{i_1<\cdots<i_m}\prod_{l=1}^{m-1}(\tau_{i_l}\tau_{i_{l+1}}+1)\\
&&+\sum_{i_1<\cdots<i_m}\prod_{l=1}^{m-1}(\sigma_{i_l}\sigma_{i_{l+1}}+1)\prod_{l=1}^{m-1}(\tau_{i_l}\tau_{i_{l+1}}+1)\\
&=&\sum_{i_1<\cdots<i_m}\left[\prod_{l=1}^{m-1}(\sigma_{i_l}\sigma_{i_{l+1}}+1)-1\right]\left[\prod_{l=1}^{m-1}(\tau_{i_l}\tau_{i_{l+1}}+1)-1\right]\\
&=&\frac{1}{m!}\sum_{i_1,\ldots,i_m=1}^n\left[\prod_{l=1}^{m-1}(\sigma_{i_l}\sigma_{i_{l+1}}+1)-1\right]\left[\prod_{l=1}^{m-1}(\tau_{i_l}\tau_{i_{l+1}}+1)-1\right]+O(n^{m-1})\\
&\equiv&\frac{1}{m!}J_m+O(n^{m-1}).
\end{eqnarray*}
Note that $J_m$ has an expression given by Lemma \ref{lemma:sharp:k=2}.
Fix even $2\le t,s\le m$ and $|\sigma|\le C\sqrt{n}$. 
Consider the following cases:

\textit{Case 1}.
If $c$ satisfies Lemma \ref{lemma:sharp:k=2} with $t-c\ge2$,
then 
\[
n^{m-t-s+c}(\sum_{i=1}^n\sigma_i\tau_i)^c(\sum_{i=1}^n\sigma_i)^{t-c}(\sum_{i=1}^n\tau_i)^{s-c}
\le C^{t-c}n^{m-(t-c)/2}\le C^{t-c}n^{m-1}.
\]

\textit{Case 2}.
Suppose that $c$ satisfies Lemma \ref{lemma:sharp:k=2} with $t-c=1$, hence, $c=t-1\le s$. 
If $s=t-1$, then 
\begin{eqnarray*}
n^{m-t-s+c}|(\sum_{i=1}^n\sigma_i\tau_i)^c(\sum_{i=1}^n\sigma_i)^{t-c}(\sum_{i=1}^n\tau_i)^{s-c}|
&=&n^{m-t}|\sum_{i=1}^n\sigma_i\tau_i|^{t-1}|\sum_{i=1}^n\sigma_i|\\
&\le&Cn^{m-t+1/2}|\sum_{i=1}^n\sigma_i\tau_i|^{t-1}\\
&\le&Cn^{m-5/2}(\sum_{i=1}^n\sigma_i\tau_i)^2.
\end{eqnarray*}
If $s>t-1$, then
\begin{eqnarray*}
n^{m-t-s+c}|(\sum_{i=1}^n\sigma_i\tau_i)^c(\sum_{i=1}^n\sigma_i)^{t-c}(\sum_{i=1}^n\tau_i)^{s-c}|
\le Cn^{m-5/2}|\sum_{i=1}^n\sigma_i\tau_i|\cdot|\sum_{i=1}^n\tau_i|.
\end{eqnarray*}

\textit{Case 3}. Suppose that $c$ satisfies Lemma \ref{lemma:sharp:k=2} with $t-c=0$.
Then 
\[
n^{m-t-s+c}|(\sum_{i=1}^n\sigma_i\tau_i)^c(\sum_{i=1}^n\sigma_i)^{t-c}(\sum_{i=1}^n\tau_i)^{s-c}|\le
n^{m-2}(\sum_{i=1}^n\sigma_i\tau_i)^2.
\]
Note that the number times that \textit{Case 3} occurs is
\begin{eqnarray*}
N_m&\equiv&\sum_{\substack{2\le t\le s\le m\\
\textrm{$t,s$ even}}}\frac{m!}{t!(s-t)!(m-s)!}\\
&=&\sum_{\substack{2\le s\le m\\
\textrm{$s$ even}}}{m\choose s}\left(\sum_{
\substack{2\le t\le s\\
\textrm{$t$ even}}}{s\choose t}\right)\\
&=&\sum_{\substack{2\le s\le m\\
\textrm{$s$ even}}}{m\choose s}(2^{s-1}-1)\\
&=&[3^m+(-1)^m]/4-2^{m-1}+1/2.
\end{eqnarray*}
Based on the above \textit{Cases 1-3},
we get that
\begin{eqnarray*}
\mathbb{E}_1\widetilde{Y}_n&\sim&\mathbb{E}_{\sigma\tau}\exp\left(\gamma\left[\frac{1}{m!}J_m+O(n^{m-1})\right]\right)1_{[|\sigma|\le C\sqrt{n}]}\\
&\le&\mathbb{E}_{\sigma\tau}\exp\left(\gamma\left[\frac{N_m}{m!} n^{m-2}(\sum_{i=1}^n\sigma_i\tau_i)^2+O(n^{m-5/2})|\sum_{i=1}^n\sigma_i|\cdot|\sum_{i=1}^n\tau_i|\right.\right.\\
&&\left.\left.+O(n^{m-5/2})(\sum_{i=1}^n\sigma_i\tau_i)^2+O(n^{m-1})\right]\right)\\
&\le&\mathbb{E}_{\sigma\tau}\exp\left(\left[\frac{\kappa (m-2)!N_m}{m!}+O(n^{-1/2})\right](n^{-1/2}\sum_{i=1}^n\sigma_i\tau_i)^2
+O(1)|n^{-1/2}\sum_{i=1}^n\sigma_i\tau_i|\cdot|n^{-1}\sum_{i=1}^n\tau_i|\right)\\
&\le&\mathbb{E}_{\sigma\tau}\exp\left(\left[\frac{\kappa (m-2)!N_m}{m!}+O(n^{-1/2})+\varepsilon)\right](n^{-1/2}\sum_{i=1}^n\sigma_i\tau_i)^2
+\frac{C^2}{\varepsilon}(n^{-1}\sum_{i=1}^n\tau_i)^2\right),
\end{eqnarray*}
where $\varepsilon>0$ is small so that $\frac{\kappa (m-2)!N_m}{m!}+O(n^{-1/2})+\varepsilon<1/2$.
Due to the independence of $\sum_{i=1}^n\sigma_i\tau_i$ and $\sum_{i=1}^n\tau_i$, we get
\begin{eqnarray}\label{sharp:thm:eq1}
&&\mathbb{E}_1\widetilde{Y}_n\nonumber\\
&\le&\mathbb{E}_{\sigma\tau}\exp\left(\left[\frac{\kappa (m-2)!N_m}{m!}+O(n^{-1/2})+\varepsilon)\right](n^{-1/2}\sum_{i=1}^n\sigma_i\tau_i)^2\right)
\times\mathbb{E}_{\sigma\tau}\exp\left(\frac{C^2}{\varepsilon}(n^{-1}\sum_{i=1}^n\tau_i)^2\right).\nonumber\\
\end{eqnarray}
By Hoeffding inequality and uniform integrability, the two expectations in the last equation are bounded,
which leads to $\widetilde{Y}_n=O_{P_1}(1)$ and so $Y_n=O_{P_1}(1)$.
So $H_1$ is contiguous to $H_0$. 
Suppose there exists a valid asymptotically powerful test which satisfies $P_0(\textrm{reject $H_0$})\to0$ and $P_1(\textrm{reject $H_0$})\to1$.
However, $H_1$ is contiguous to $H_0$, hence, $P_1(\textrm{reject $H_0$})\to0$. Contradiction!
This finishes the proof of Part (\ref{thm:sharp:i}).

\textit{Proof of Part (\ref{thm:sharp:ii}).}
We conduct a finer analysis on \textit{Case 3}.
Let $\rho=\frac{1}{n}\sum_{i=1}^n\sigma_i\tau_i$ and $\eta=\frac{1}{n}\sum_{i=1}^n\tau_i$.
Then 
\begin{eqnarray*}
&&\sum_{\substack{2\le t\le s\le m\\
\textrm{\textrm{$t,s$ even}}}}{m\choose s}{s\choose t}
n^{m-s}(\sum_{i=1}^n\sigma_i\tau_i)^t(\sum_{i=1}^n\tau_i)^{s-t}\\
&=&n^m\sum_{\substack{2\le t\le s\le m\\
\textrm{\textrm{$t,s$ even}}}}{m\choose s}{s\choose t}
\rho^t\eta^{s-t}\\
&=&n^m\rho^2\sum_{\substack{0\le s\le m-2\\
\textrm{$s$ even}}}{m\choose s+2}\sum_{\substack{0\le t\le s\\
\textrm{$t$ even}}}{s+2\choose t+2}\rho^t\eta^{s-t}\\
&\le&\frac{m(m-1)}{2}n^m\rho^2\sum_{\substack{0\le s\le m-2\\
\textrm{$s$ even}}}{m-2\choose s}\sum_{\substack{0\le t\le s\\
\textrm{$t$ even}}}{s\choose t}\rho^t\eta^{s-t}\\
&=&\frac{m(m-1)}{8}n^m\rho^2\left[(\rho+\eta+1)^{m-2}+(-\rho-\eta+1)^{m-2}
+(-\rho+\eta+1)^{m-2}+(\rho-\eta+1)^{m-2}\right].
\end{eqnarray*}
Notice that
\[
|\rho+\eta+1|=|\frac{1}{n}\sum_{i=1}^n(\sigma_i+1)\tau_i+1|\le 1+\frac{1}{n}\sum_{i=1}^n|1+\sigma_i|=
1+\frac{2}{n}|S_\sigma^+|,
\]
where $S_\sigma^+=\{i: \sigma_i=+\}$. Since $|\sigma|\le C\sqrt{n}$,
$|S_\sigma^+|\sim n/2$, leading to that $|\rho+\eta+1|\le2$.
Similarly, one can show that $|-\rho-\eta+1|,|-\rho+\eta+1|,|\rho-\eta+1|$ are all less than or equal to $2$.
So 
\[
\sum_{\substack{2\le t\le s\le m\\
\textrm{\textrm{$t,s$ even}}}}{m\choose s}{s\choose t}
n^{m-s}(\sum_{i=1}^n\sigma_i\tau_i)^t(\sum_{i=1}^n\tau_i)^{s-t}\le 
\frac{m(m-1)2^m}{8}n^m\rho^2.
\]
Similar to (\ref{sharp:thm:eq1}), it holds that when $\kappa 2^m/8<1/2$, i.e., $\kappa<2^{2-m}$,
$\mathbb{E}_1\widetilde{Y}_n$ is asymptotically bounded, i.e., $H_1$ is asymptotically contiguous to $H_0$. 
This shows the desired conclusion.

Next we prove the lower bound for $\kappa$ so that contiguity may fail.
We begin with a slight modification of the likelihood ratio.
Consider
\begin{eqnarray*}
Y_n&=&2^{-n}\sum_{\tau\in\{\pm\}^n}\prod_{i_1<\cdots<i_m}\left(\frac{p_{i_1:i_m}(\tau)}{p_0}\right)^{A_{i_1:i_m}}\left(\frac{q_{i_1:i_m}(\tau)}{q_0}\right)^{1-A_{i_1:i_m}}\\
&=&2^{-n}\sum_{\tau}\prod_{i\in S_\tau}\left(\frac{a}{n^{m-1}p_0}A_i+\frac{1-\frac{a}{n^{m-1}}}{1-p_0}(1-A_i)\right)\times
\prod_{i\notin S_\tau}\left(\frac{b}{n^{m-1}p_0}A_i+\frac{1-\frac{b}{n^{m-1}}}{1-p_0}(1-A_i)\right),
\end{eqnarray*}
where $S_\tau=\{1\le i_1<\cdots<i_m\le n:\tau_{i_1}=\cdots=\tau_{i_m}\}$.
Define $S=\{1\le i_1<\cdots<i_m\le n:A_i=1\}$,
then we have
\begin{eqnarray*}
Y_n&=&2^{-n}\sum_{\tau}\left(\frac{a}{n^{m-1}p_0}\right)^{|S_\tau\cap S|}\left(\frac{b}{n^{m-1}p_0}\right)^{|\bar{S}_\tau\cap S|}
\left(\frac{1-\frac{a}{n^{m-1}}}{1-p_0}\right)^{|S_\tau\cap\bar{S}|}\left(\frac{1-\frac{b}{n^{m-1}}}{1-p_0}\right)^{|\bar{S}_\tau\cap\bar{S}|}\\
&=&2^{-n}\sum_{\tau}\left(\frac{a}{n^{m-1}p_0}\right)^{|S_\tau\cap S|}\left(\frac{b}{n^{m-1}p_0}\right)^{|\bar{S}_\tau\cap S|}\\
&&\times
\exp\left(-\frac{(2^{m-1}-1)(a-b)}{(2n)^{m-1}}|S_\tau\cap\bar{S}|+\frac{a-b}{(2n)^{m-1}}|\bar{S}_\tau\cap\bar{S}|\right).
\end{eqnarray*}
By Chernoff inequality \cite{BLM},
\[
P_1\left(\sum_{i\in S_\tau\cap S_\sigma}(A_i-\frac{a}{n^{m-1}})\le -\sqrt{\frac{2an^{2-m}}{|S_\tau\cap S_\sigma|}}|S_\tau\cap S_\sigma|\right)
\le e^{-n},
\]
which leads to
\[
P_1\left(\min_{\tau}\frac{\sum_{i\in S_\tau\cap S_\sigma}(A_i-\frac{a}{n^{m-1}})}{\sqrt{\frac{a}{n^{m-1}}|S_\tau\cap S_\sigma|}}\le-\sqrt{2n}\right)\le e^{-n}2^n\to0.
\]
Similarly,
\[
P_1\left(\min_{\tau}\frac{\sum_{i\in S_\tau\cap \bar{S}_\sigma}(A_i-\frac{b}{n^{m-1}})}{\sqrt{\frac{b}{n^{m-1}}|S_\tau\cap \bar{S}_\sigma|}}\le-\sqrt{2n}\right)\to0,
\]
\[
P_1\left(\max_{\tau}\frac{\sum_{i\in S_\tau\cap S_\sigma}(A_i-\frac{a}{n^{m-1}})}{\sqrt{n}+\sqrt{n+\frac{2a}{n^{m-1}}|S_\tau\cap S_\sigma|}}\ge\sqrt{n}\right)\to0,
\]
\[
P_1\left(\max_{\tau}\frac{\sum_{i\in S_\tau\cap \bar{S}_\sigma}(A_i-\frac{b}{n^{m-1}})}{\sqrt{n}+\sqrt{n+\frac{2b}{n^{m-1}}|S_\tau\cap \bar{S}_\sigma|}}\ge\sqrt{n}\right)\to0.
\]
Define 
\[
\mathcal{E}_1=\left\{\min_{\tau}\frac{\sum_{i\in S_\tau\cap S_\sigma}(A_i-\frac{a}{n^{m-1}})}{\sqrt{\frac{a}{n^{m-1}}|S_\tau\cap S_\sigma|}}\ge-\sqrt{2n},
\min_{\tau}\frac{\sum_{i\in S_\tau\cap \bar{S}_\sigma}(A_i-\frac{b}{n^{m-1}})}{\sqrt{\frac{b}{n^{m-1}}|S_\tau\cap \bar{S}_\sigma|}}\ge-\sqrt{2n}
\right\},
\]
\[
\mathcal{E}_2=\left\{\max_{\tau}\frac{\sum_{i\in S_\tau\cap S_\sigma}(A_i-\frac{a}{n^{m-1}})}{\sqrt{n}+\sqrt{n+\frac{2a}{n^{m-1}}|S_\tau\cap S_\sigma|}}\le\sqrt{n},
\max_{\tau}\frac{\sum_{i\in S_\tau\cap \bar{S}_\sigma}(A_i-\frac{b}{n^{m-1}})}{\sqrt{n}+\sqrt{n+\frac{2b}{n^{m-1}}|S_\tau\cap \bar{S}_\sigma|}}\le\sqrt{n}
\right\}.
\]
Therefore, $P(\mathcal{E}_1\cap\mathcal{E}_2)\to1$. 
On $\mathcal{E}_1\cap\mathcal{E}_2$, for any $\tau\in\{\pm\}^n$,
\begin{eqnarray*}
|\bar{S}_\tau\cap \bar{S}|-\sum_{i\in \bar{S}_\tau}q_i(\sigma)&=&
-\sum_{i\in \bar{S}_\tau\cap S_\sigma}(A_i-\frac{a}{n^{m-1}})-\sum_{i\in \bar{S}_\tau\cap \bar{S}_\sigma}(A_i-\frac{b}{n^{m-1}})\\
&\ge& -\sqrt{n}\left(2\sqrt{n}+\sqrt{n+\frac{2a}{n^{m-1}}|\bar{S}_\tau\cap S_\sigma|}+\sqrt{\frac{2b}{n^{m-1}}|\bar{S}_\tau\cap \bar{S}_\sigma|}\right)=O(n),
\end{eqnarray*}
and
\begin{eqnarray*}
|S_\tau\cap \bar{S}|-\sum_{i\in S_\tau}q_i(\sigma)&=&-\sum_{i\in S_\tau\cap S_\sigma}(A_i-\frac{a}{n^{m-1}})-\sum_{i\in S_\tau\cap \bar{S}_\sigma}(A_i-\frac{b}{n^{m-1}})\\
&\le&\sqrt{2n}\left(\sqrt{\frac{a}{n^{m-1}}|S_\tau\cap S_\sigma|}+\sqrt{\frac{b}{n^{m-1}}|S_\tau\cap \bar{S}_\sigma|}\right)=O(n).
\end{eqnarray*}
Therefore, on $\mathcal{E}_1\cap\mathcal{E}_2$,
\begin{eqnarray*}
Y_n&\gtrsim& 2^{-n}\sum_{\tau}\left(\frac{a}{n^{m-1}p_0}\right)^{|S_\tau\cap S|}\left(\frac{b}{n^{m-1}p_0}\right)^{|\bar{S}_\tau\cap S|}\\
&&\times\exp\left(-\frac{(2^{m-1}-1)(a-b)}{(2n)^{m-1}}\sum_{i\in S_\tau}q_i(\sigma)+
\frac{a-b}{(2n)^{m-1}}\sum_{i\in\bar{S}_\tau}q_i(\sigma)\right)\equiv\widetilde{Y}_n.
\end{eqnarray*}
Clearly, 
\begin{eqnarray*}
&&\widetilde{Y_n}\ge 2^{-n}\left(\frac{a}{n^{m-1}p_0}\right)^{|S_\sigma\cap S|}\left(\frac{b}{n^{m-1}p_0}\right)^{|\bar{S}_\sigma\cap S|}\\
&&\times\exp\left(-\frac{(2^{m-1}-1)(a-b)}{(2n)^{m-1}}\sum_{i\in S_\sigma}q_i(\sigma)+
\frac{a-b}{(2n)^{m-1}}\sum_{i\in\bar{S}_\sigma}q_i(\sigma)\right).
\end{eqnarray*}
It is easy to see that
\[
\frac{a}{n^{m-1}p_0}=t\equiv\frac{2^{m-1}a}{a+(2^{m-1}-1)b},\,\,\,\,\,\,\,\, 
\frac{b}{n^{m-1}p_0}=r\equiv\frac{2^{m-1}b}{a+(2^{m-1}-1)b}.
\]
Let $X_\sigma=|S_\sigma\cap S|$,
$Y_\sigma=|\bar{S}_\sigma\cap S|$, and let $\theta_\sigma\in(0,1)$ satisfy
\begin{eqnarray*}
-\log{\theta_\sigma}&=&(t-1-\log{t})\mathbb{E}_1(X_\sigma|\sigma)+(r-1-\log{r})\mathbb{E}_1(Y_\sigma|\sigma)\\
&&+C\sqrt{(\log{t})^2\mathbb{E}_1(X_\sigma|\sigma)+(\log{r})^2\mathbb{E}_1(Y_\sigma|\sigma)}.
\end{eqnarray*}
By Markov inequality,
\begin{eqnarray*}
&&P_1\left(t^{X_\sigma}r^{Y_\sigma}\le\theta_\sigma\mathbb{E}_1(t^{X_\sigma}r^{Y_\sigma}|\sigma)|\sigma\right)\\
&=&P_1\left((X_\sigma-\mathbb{E}_1(X_\sigma|\sigma))\log{t}+(Y_\sigma-\mathbb{E}_1(Y_\sigma|\sigma))\log{r}\right.\\
&&\left.\le \log{\theta_\sigma}+(t-1-\log{t})\mathbb{E}_1(X_\sigma|\sigma)+(r-1-\log{r})\mathbb{E}_1(Y_\sigma|\sigma)\right)\\
&\le&\frac{(\log{t})^2\mathbb{E}_1(X_\sigma|\sigma)+(\log{r})^2\mathbb{E}_1(Y_\sigma|\sigma)}{\left(\log{\theta_\sigma}+
(t-1-\log{t})\mathbb{E}_1(X_\sigma|\sigma)+(r-t-\log{r})\mathbb{E}_1(Y_\sigma|\sigma)\right)^2}=C^{-2},
\end{eqnarray*}
therefore, one can choose a sufficiently large $C>0$ such that $P_1(\mathcal{E}_3)\to1$, where 
\[
\mathcal{E}_3=\{t^{X_\sigma}r^{Y_\sigma}\ge\theta_\sigma\mathbb{E}_1(t^{X_\sigma}r^{Y_\sigma}|\sigma)\}.
\]
On $\mathcal{E}_3$,
\begin{eqnarray*}
\widetilde{Y}_n&\ge& 2^{-n}
\theta_\sigma\mathbb{E}_1(t^{X_\sigma}r^{Y_\sigma}|\sigma)
\times\exp\left(-\frac{(2^{m-1}-1)(a-b)}{(2n)^{m-1}}\sum_{i\in S_\sigma}q_i(\sigma)+
\frac{a-b}{(2n)^{m-1}}\sum_{i\in\bar{S}_\sigma}q_i(\sigma)\right)\\
&\sim&2^{-n}\theta_\sigma\exp\left(\frac{(2^{m-1}-1)(a-b)}{a+(2^{m-1}-1)b}\cdot\frac{a}{n^{m-1}}|S_\sigma|
-\frac{(2^{m-1}-1)(a-b)}{(2n)^{m-1}}\left(1-\frac{a}{n^{m-1}}\right)|S_\sigma|\right.\\
&&\left.-\frac{a-b}{a+(2^{m-1}-1)b}\cdot\frac{b}{n^{m-1}}|\bar{S}_\sigma|+\frac{a-b}{(2n)^{m-1}}\left(1-\frac{b}{n^{m-1}}\right)|\bar{S}_\sigma|\right)\\
&=&2^{-n}\theta_\sigma\exp\left(\frac{(2^{m-1}-1)^2(a-b)^2}{2^{m-1}(a+(2^{m-1}-1)b)}\cdot\frac{|S_\sigma|}{n^{m-1}}
+\frac{(a-b)^2}{2^{m-1}(a+(2^{m-1}-1)b)}\cdot\frac{|\bar{S}_\sigma|}{n^{m-1}}+O(1)\right).
\end{eqnarray*}
Since $|S_\sigma^+|\sim n/2$ and $|S_\sigma^-|\sim n/2$,
we have 
\[
|S_\sigma|\sim 2{n/2\choose m}\sim \frac{n^m}{m!2^{m-1}},\,\,\,\,
|\bar{S}_\sigma|={n\choose m}-|S_\sigma|\sim(1-2^{-m+1})\frac{n^m}{m!}.
\]
Therefore,
\begin{equation}\label{good:eq1}
\widetilde{Y}_n\gtrsim 2^{-n}\theta_\sigma
\exp\left(\frac{(2^{m-1}-1)(a-b)^2}{2^{m-1}m!(a+(2^{m-1}-1)b)}\cdot n\right).
\end{equation}
It is easy to see that
\begin{eqnarray*}
\log{\theta_\sigma}&\ge& -(t-1-\log{t})|S_\sigma|\frac{a}{n^{m-1}}-(r-1-\log{r})|\bar{S}_\sigma|\frac{b}{n^{m-1}}+O(\sqrt{n})\\
&\sim&-\left[\frac{(2^{m-1}-1)(a-b)}{a+(2^{m-1}-1)b}-\log\left(\frac{2^{m-1}a}{a+(2^{m-1}-1)b}\right)\right]\frac{an}{m!2^{m-1}}\\
&&+\left[\frac{a-b}{a+(2^{m-1}-1)b}+\log\left(\frac{2^{m-1}b}{a+(2^{m-1}-1)b}\right)\right]\frac{(1-2^{-m+1})bn}{m!}\\
&=&-\frac{(2^{m-1}-1)(a-b)^2}{m!2^{m-1}(a+(2^{m-1}-1)b)}\cdot n+\delta(m,a,b)\cdot n,
\end{eqnarray*}
where
\[
\delta(m,a,b)=\frac{1}{m!2^{m-1}}\left[a\log\left(\frac{2^{m-1}a}{a+(2^{m-1}-1)b}\right)
+(2^{m-1}-1)b\log\left(\frac{2^{m-1}b}{a+(2^{m-1}-1)b}\right)\right].
\]
Therefore, on $\mathcal{E}_1\cap\mathcal{E}_2\cap\mathcal{E}_3$, we have
\[
Y_n\gtrsim 2^{-n}\exp(\delta(m,a,b)\cdot n).
\]
For convenience, write $a=c+\varepsilon$ and $b=c-\varepsilon$.
For any $\kappa_0>\frac{m(m-1)\log{2}}{2^{m-1}-1}$, 
suppose that the SNR $\kappa$ for $H_0$ and $H_1$ is $\kappa_0$, which leads to that
\[
\frac{(a-b)^2}{2^{m-1}(m-2)!(a+(2^{m-1}-1)b)}=\frac{4\varepsilon^2}{2^{m-1}(m-2)!(2\varepsilon+2^{m-1}(c-\varepsilon))}=\kappa.
\] 
Suppose that both $c$ and $\varepsilon$ tend to infinity but $\varepsilon$ grows slower than $c$,
then the above equation leads to that 
\[
\frac{\varepsilon^2}{c}\to\frac{(2^{m-1})^2(m-2)!\kappa}{4}.
\]
Meanwhile, it is easy to find that
\begin{eqnarray*}
&&\delta(m,a,b)\\
&=&\frac{1}{m!2^{m-1}}\left[(c+\varepsilon)\log\left(\frac{2^{m-1}(c+\varepsilon)}{2\varepsilon+2^{m-1}(c-\varepsilon)}\right)
+(2^{m-1}-1)(c-\varepsilon)\log\left(\frac{2^{m-1}(c-\varepsilon)}{2\varepsilon+2^{m-1}(c-\varepsilon)}\right)\right]\\
&=&\frac{1}{m!2^{m-1}}\left[(c+\varepsilon)\log\left(1+\frac{2\varepsilon(2^{m-1}-1)}{2\varepsilon+2^{m-1}(c-\varepsilon)}\right)
+(2^{m-1}-1)(c-\varepsilon)\log\left(1-\frac{2\varepsilon}{2\varepsilon+2^{m-1}(c-\varepsilon)}\right)\right]\\
&\sim&\frac{1}{m!2^{m-1}}\left[(c+\varepsilon)\cdot\frac{2\varepsilon(2^{m-1}-1)}{2\varepsilon+2^{m-1}(c-\varepsilon)}
-(2^{m-1}-1)(c-\varepsilon)\frac{2\varepsilon}{2\varepsilon+2^{m-1}(c-\varepsilon)}\right]\\
&\sim&\frac{1}{m!2^{m-1}}\left[(c+\varepsilon)\frac{2\varepsilon(2^{m-1}-1)}{2^{m-1}c}-(2^{m-1}-1)(c-\varepsilon)\frac{2\varepsilon}{2^{m-1}c}\right]\\
&=&\frac{4(2^{m-1}-1)}{m!(2^{m-1})^2}\cdot\frac{\varepsilon^2}{c}\to\frac{(2^{m-1}-1)\kappa}{m(m-1)}>\log{2}.
\end{eqnarray*} 
Therefore, fixing the above large $c,\varepsilon$ so that $\delta(m,a,b)>\log{2}$,
one has that, with $P_1$-probability approaching one, 
\[
Y_n\gtrsim \exp\left((\delta(m,a,b)-\log{2})n\right)\to\infty.
\]
Note $\mathbb{E}_0Y_n=1$ which leads to 
$Y_n=O_{P_0}(1)$, while under $P_1$, $Y_n$ tends to infinity,
so $H_0$ and $H_1$ are distinguishable. This completes the proof of Part (\ref{thm:sharp:ii}).

\end{proof}

 \subsection{Proof of Theorem \ref{normality} and Theorem \ref{power}.}
The proof relies on the following lemma.

 \begin{Lemma}\label{rate}
 Under the condition of Theorem \ref{normality}, we have
 \begin{equation}\label{Erate}
 \mathbb{E}(\widehat{E}-E)^2=O\Big(\frac{a_1^2}{n}\Big),
 \end{equation}
 \begin{equation}\label{Vrate}
 \mathbb{E}(\widehat{V}-V)^2=O\Big(\frac{a_1^4}{n}\Big),
 \end{equation}
 \begin{equation}\label{Trate}
 \mathbb{E}(\widehat{T}-T)^2=O\Big(\frac{a_1^3}{n^{3(m-l)}}\Big),
 \end{equation}
 \begin{equation}\label{Tnormal}
 \frac{\sqrt{\binom{n}{3(m-l)}(m-l)}(\widehat{T}-T)}{\sqrt{T}}\stackrel{d}{\rightarrow} N(0,1).
 \end{equation}
 \end{Lemma}
 
\begin{proof}[Proof of Theorem \ref{normality}]
	
  It is easy to check the following expansion
 \begin{eqnarray}\nonumber
 \widehat{T}-\Big(\frac{\widehat{V}}{\widehat{E}}\Big)^3&=&T-\Big(\frac{V}{E}\Big)^3+(\widehat{T}-T)\\  \nonumber
 & &+\Big(\frac{V}{E}-\frac{\widehat{V}}{\widehat{E}}\Big)^3-3\frac{V}{E}\Big(\frac{V}{E}-\frac{\widehat{V}}{\widehat{E}}\Big)^2+3\Big(\frac{V}{E}\Big)^2\frac{V-\widehat{V}}{E}\\   \label{normal1}
 &&-3\Big(\frac{V}{E}\Big)^2\Big(\frac{1}{\widehat{E}}-\frac{1}{E}\Big)V-3\Big(\frac{V}{E}\Big)^2\Big(\frac{1}{\widehat{E}}-\frac{1}{E}\Big)(\widehat{V}-V).
 \end{eqnarray}
 By Lemma \ref{rate}, the first two terms in (\ref{normal1}) are the leading terms and hence we have
 \begin{eqnarray*}
&&\frac{\sqrt{\binom{n}{3(m-l)}(m-l)}\Big(\widehat{T}-\Big(\frac{\widehat{V}}{\widehat{E}}\Big)^3\Big)}{\sqrt{T}}-\frac{\sqrt{\binom{n}{3(m-l)}(m-l)}\Big(T-\Big(\frac{V}{E}\Big)^3\Big)}{\sqrt{T}}\\
&=&\frac{\sqrt{\binom{n}{3(m-l)}(m-l)}\Big(\widehat{T}-T\Big)}{\sqrt{T}}\stackrel{d}{\rightarrow} N(0,1).
\end{eqnarray*}
 Since $\widehat{T}=T+o_P(1)$, we have 
 \[\frac{\sqrt{\binom{n}{3(m-l)}(m-l)}\Big(\widehat{T}-\Big(\frac{\widehat{V}}{\widehat{E}}\Big)^3\Big)}{\sqrt{\widehat{T}}}-\delta\stackrel{d}{\rightarrow} N(0,1),\]
 which completes the proof.
\end{proof}
 
\begin{proof}[Proof of Theorem \ref{power}]
	We rewrite the statistic as
 \begin{eqnarray*}
 &&2\sqrt{\binom{n}{3(m-l)}(m-l)}\Bigg(\sqrt{\widehat{T}}-\Big(\frac{\widehat{V}}{\widehat{E}}\Big)^{\frac{3}{2}}\Bigg)\\
 &=&2\sqrt{\binom{n}{3(m-l)}(m-l)}\frac{T-\Big(\frac{V}{E}\Big)^3}{\sqrt{\widehat{T}}+\Big(\frac{\widehat{V}}{\widehat{E}}\Big)^{\frac{3}{2}}}+2\sqrt{\binom{n}{3(m-l)}(m-l)}\frac{\widehat{T}-T}{\sqrt{\widehat{T}}+\Big(\frac{\widehat{V}}{\widehat{E}}\Big)^{\frac{3}{2}}}+o_P(1).
 \end{eqnarray*}
 The first term is of the same order as $\delta$, while the second term is bounded in probability. Hence, we get the desired result.
\end{proof}
 
 \section*{Acknowledgement.} We are grateful to the co-Editor Professor Richard Samworth, the AE, and referees for their insightful comments which greatly improved the quality and scope of the paper.  We thank Sumit Mukherjee for suggesting the truncation technique which motivates Theorem \ref{sharp:SNR:phase:transition}.

\newpage
\setcounter{page}{1}
\begin{frontmatter}
\begin{center}
\textit{Supplement to}
\end{center}
\title{Testing Community Structure for Hypergraphs}
\end{frontmatter}

This supplement contains the proofs of Lemmas \ref{trace}, \ref{sumtrace},
\ref{cyclepois}, \ref{matrixde}, \ref{rate} and Propositions \ref{proph1}, \ref{fixeddegree}.

\begin{proof}[Proof of Lemma \ref{trace}]
	Note that $M_0=(a-b)I+k^{m-2}bJ$,
where $I$ is $k\times k$ identity matrix and $J$ is $k\times k$ matrix with every entry 1. For any real number $\lambda$, we have
\begin{eqnarray*}
M_0-\lambda I=(a-b-\lambda)I+k^{m-2}bJ
=k^{m-2}b\Big(J-\frac{\lambda-a+b}{k^{m-2}b}I\Big).
\end{eqnarray*}
Then $det(M_0-\lambda I)=0$ implies that $det(J-\frac{\lambda-a+b}{k^{m-2}b}I)=0$. The eigenvalue of $J$ are $k$ and 0 with multiplicity $k-1$, which implies $\lambda=a-b$, $a+(k^{m-2}-1)b$ and the desired result follows.
\end{proof}

\begin{proof}[Proof of Lemma \ref{sumtrace}]
	Let $I_j=(i_{(j-1)m-j+3},\dots, i_{jm-j})$. Then we have
\begin{eqnarray*}
&&\sum_{i_1,\dots,i_{jm-j}\in\{1,\dots,k\}}M_{i_1i_2\dots i_m}M_{i_m\dots i_{2m-1}}M_{i_{2m-1}\dots i_{3m-2}}\dots M_{i_{(j-1)m-(j-2)}\dots i_{jm-j}i_1}\\
&=&\sum_{I_1,I_2,\dots,I_j}\sum_{i_1,i_m,i_{2m-1},\dots,i_{(j-1)m-(j-2)}\in\{1,2,\dots,k\}}M_{i_1I_1i_m}M_{i_mI_2i_{2m-1}}\dots M_{i_{(j-1)m-(j-2)}I_ji_1}\\
&=&\sum_{I_1,I_2,\dots, I_j} Tr\Big(M(I_1)M(I_2)\dots M(I_j)\Big),
\end{eqnarray*}
where $M(I_t)=(M_{iI_ts})_{i,s=1}^k$ is a $k\times k$ matrix. By the definition of $M_{i_1i_2\dots i_m}$, it follows that
\begin{eqnarray*}
M(I_t)&=&\begin{bmatrix}
a & b & \dots & b \\
b &b & \dots & b \\
 \vdots & \vdots & \dots & \vdots\\
b & b &\dots & b
\end{bmatrix}+\begin{bmatrix}
b & b & \dots & b \\
b &a & \dots & b \\
 \vdots & \vdots & \dots & \vdots\\
b & b &\dots & b
\end{bmatrix}+\dots+
\begin{bmatrix}
b & b & \dots & b \\
b &b & \dots & b \\
 \vdots & \vdots & \dots & \vdots\\
b & b &\dots & a
\end{bmatrix}+\sum_{I_t:\ \mbox{elements are different}}M(I_t)\\
&=&\begin{bmatrix}
a+(k-1)b & kb & \dots & kb \\
kb &a+(k-1)b & \dots & kb \\
 \vdots & \vdots & \dots & \vdots\\
kb & kb &\dots & a+(k-1)b
\end{bmatrix}+
(k^{m-2}-k)\begin{bmatrix}
b & b & \dots & b \\
b &b & \dots & b \\
 \vdots & \vdots & \dots & \vdots\\
b & b &\dots & b
\end{bmatrix}=M_0,
\end{eqnarray*}
which completes the proof.
\end{proof}

\begin{proof}[Proof of Lemma \ref{cyclepois}]
	Let $H$ be a graph on a subset of $[n]$ with vertex set $\mathcal{V}(H)$ and edge set $\mathcal{E}(H)$. For any sequence of positive integers $j_2$, $j_3$, $\dots$, $j_s$, we have 
\[\prod_{h=2}^s[X_{hn}]_{j_h}=\sum_{(H_{hi})}\prod_{h=2}^s\prod_{i=1}^{j_h}1_{H_{hi}}.\]
Then
\begin{eqnarray}\label{x1}
\mathbb{E}_0\prod_{h=2}^s[X_{hn}]_{j_h}=\sum_{(H_{hi})}\mathbb{E}_0\prod_{h=2}^s\prod_{i=1}^{j_s}1_{H_{si}}
=\sum_{(H_{si})\in B}\mathbb{E}_0\prod_{h=2}^s\prod_{i=1}^{j_h}1_{H_{hi}}
+\sum_{(H_{hi})\in \overline{B}}\mathbb{E}_0\prod_{h=2}^s\prod_{i=1}^{j_h}1_{H_{hi}}.
\end{eqnarray}
The summand in the first term of (\ref{x1}) can be calculated as below
\begin{eqnarray*}
\mathbb{E}_0\prod_{h=2}^s\prod_{i=1}^{j_h}1_{H_{hi}}=\mathbb{E}_{\tau}\mathbb{E}_0\left[\prod_{h=2}^s\prod_{i=1}^{j_h}1_{H_{hi}}\Big|\tau\right]
=\prod_{h=2}^s\prod_{i=1}^{j_h}E_{\tau_{hi}}\prod_{(i_1,\dots,i_m)\in\mathcal{E}(H_{hi})}\frac{d}{n^{m-1}}
=\prod_{h=2}^k\prod_{i=1}^{j_h}\frac{d^h}{n^{h(m-1)}}.
\end{eqnarray*}
Note that $\#B=\frac{n!}{(n-M_1)}\prod_{h=2}^k(\frac{1}{2h(m-2)!^h})^{j_h}$, $M_1=(m-1)\sum_{h=2}^shj_h$. Hence the first term in the right hand side of (\ref{x1}) by Lemma \ref{trace} is
\[
\#B\times\prod_{h=2}^s\prod_{i=1}^{j_h}\frac{d^h}{k^{h(m-1)}n^{h(m-1)}}=\frac{n!}{(n-M_1)!n^{M_1}}\prod_{h=2}^s\left[\frac{d^h}{2h(m-2)!^h}\right]^{j_h}\rightarrow\prod_{h=2}^s\lambda_h^{j_h}.
\]

For $(H_{hi})\in \overline{B}$, $H=\cup H_{hi}$ has at most $M_1-1$ vertices and $\sum_{h=2}^shj_h$ hyperedges, and hence $|\mathcal{V}(H)|<|\mathcal{E}(H)|(m-1)$, and

\[\mathbb{E}_0\prod_{h=2}^s\prod_{i=1}^{j_h}1_{H_{hi}}=\prod_{(i_1,\dots,i_m)\in\mathcal{E}(H)}\left(\frac{a}{n^{m-1}}\right)^{1_{[\tau_u=\tau_v]}}\left(\frac{b}{n^{m-1}}\right)^{1_{[\tau_u\neq\tau_v]}}\leq\left(\frac{a}{n^{m-1}}\right)^{|\mathcal{E}(H)|}.\]
There are $\binom{n}{|\mathcal{V}(H)|}|\mathcal{V}(H)|!$ graphs isomorphic to $H$. Then 

\[\sum_{\text{$H'$ isomorphic to $H$}}E_1[1_{H^{\prime}}|\tau]\leq\left(\frac{a}{n^{m-1}}\right)^{|\mathcal{E}(H)|}\binom{n}{|\mathcal{V}(H)|}|\mathcal{V}(H)|!\rightarrow0.\]
Since the number of isomorphism classes is bounded, the second term in the right hand side of (\ref{x1}) goes to zero. Hence, 
$\mathbb{E}_0\prod_{h=2}^s[X_{hn}]_{j_h}\rightarrow\prod_{h=2}^s\lambda_h^{j_h}$, which completes the proof by Lemma 2.8 in Wormald \cite{W99}.
\end{proof}

\begin{proof}[Proof of Lemma \ref{matrixde}]
	We only need to find $Cov(\widetilde\sigma_u,\widetilde\sigma_u\otimes\widetilde\tau_u)$, $Cov(\widetilde\tau_u,\widetilde\sigma_u\otimes\widetilde\tau_u)$ and $Var(\widetilde\sigma_u\otimes\widetilde\tau_u)$.
\begin{eqnarray*}
&&Cov(\widetilde\sigma_u,\widetilde\sigma_u\otimes\widetilde\tau_u)=E[(\widetilde\sigma_u-{\bf p})\widetilde\sigma_u^T\otimes\widetilde\tau_u^T]\\
&=&E
\begin{bmatrix}
(1[\sigma_u=1]-p)1[\sigma_u=1]\widetilde\tau_u^T & (1[\sigma_u=1]-p)1[\sigma_u=2]\widetilde\tau_u^T	& \dots	& (1[\sigma_u=1]-p)1[\sigma_u=k]\widetilde\tau_u^T \\
(1[\sigma_u=2]-p)1[\sigma_u=1]\widetilde\tau_u^T &(1[\sigma_u=2]-p)1[\sigma_u=2]\widetilde\tau_u^T  & \dots	& (1[\sigma_u=2]-p)1[\sigma_u=k]\widetilde\tau_u^T \\
\vdots						 &		\vdots					& \vdots &\vdots\\
(1[\sigma_u=k]-p)1[\sigma_u=1]\widetilde\tau_u^T &(1[\sigma_u=k]-p)1[\sigma_u=2]\tau_u^T  & \dots	& (1[\sigma_u=k]-p)1[\sigma_u=k]\tau_u^T \\
\end{bmatrix}\\
&=& 
\begin{bmatrix}
(p-p^2){\bf p}^T	&-p^2{\bf p}^T		&\dots		&-p^2{\bf p}^T\\
-p^2{\bf p}^T		&-(p-p^2){\bf p}^T	&\dots		&-p^2{\bf p}^T\\
\vdots				&\vdots				&\vdots		&\vdots\\
-p^2{\bf p}^T		&-p^2{\bf p}^T	&\dots		&-(p-p^2){\bf p}^T\\
\end{bmatrix}
=V\otimes{\bf p}^T.
\end{eqnarray*}
Similarly one can get $Cov(\widetilde\tau_u,\widetilde\sigma_u\otimes\widetilde\tau_u)={\bf p}^T\otimes V$. The variance of $\widetilde\sigma_u\otimes\widetilde\tau_u$ can be calculated as
\begin{eqnarray*}
Cov(\widetilde\sigma_u\otimes\widetilde\tau_u,\widetilde\sigma_u\otimes\widetilde\tau_u)&=&E[(\widetilde\sigma_u\otimes\widetilde\tau_u-{\bf p}\otimes{\bf p} )\widetilde\sigma_u^T\otimes\widetilde\tau_u^T]\\
&=&E
\begin{bmatrix}
(1[\sigma_u=1]\widetilde\tau_u-p{\bf p})1[\sigma_u=1]\widetilde\tau_u^T &\dots 	&(1[\sigma_u=1]\widetilde\tau_u-p{\bf p})1[\sigma_u=k]\widetilde\tau_u^T\\
(1[\sigma_u=2]\widetilde\tau_u-p{\bf p})1[\sigma_u=1]\widetilde\tau_u^T &\dots 	&(1[\sigma_u=2]\widetilde\tau_u-p{\bf p})1[\sigma_u=k]\widetilde\tau_u^T\\
\vdots									&\vdots &\vdots\\
(1[\sigma_u=k]\widetilde\tau_u-p{\bf p})1[\sigma_u=1]\widetilde\tau_u^T &\dots 	&(1[\sigma_u=k]\widetilde\tau_u-p{\bf p})1[\sigma_u=k]\tau_u^T\\
\end{bmatrix}\\
&=&
\begin{bmatrix}
p^2I_k-p^4J_k	&	-p^4J_k		&\dots 	& -p^4J_k\\
	-p^4J_k	& p^2I_k-p^4J_k		&\dots 	& -p^4J_k\\
    \vdots  & \vdots		&\vdots & \vdots\\
 	-p^4J_k	& -p^4J_k	 		&\dots 	& p^2I_k-p^4J_k\\   
\end{bmatrix}
=p^2I_{k^2}-p^4J_{k^2}.
\end{eqnarray*}

Note that $(I_k\otimes{\bf p})V=V\otimes{\bf p}$, $V (I_k\otimes{\bf p}^T)=V\otimes {\bf p}^T$, $({\bf p}\otimes I_k)V={\bf p}\otimes V$, $V ({\bf p}^T\otimes I_k)= {\bf p}^T\otimes V$. Direct computation yields $R^T\Sigma R=\Lambda$ and
\begin{eqnarray*}
&&\Lambda_1R^{-1}A(R^{-1})^T\Lambda_1\\
&=&
\Lambda_1
\begin{bmatrix}
I_k	&0		&I_k\otimes{\bf p}^T\\
0	&I_k		&{\bf p}^T\otimes I_k\\
0	&0		&I_{k^2}
\end{bmatrix}
\begin{bmatrix}
c_1I_k	&0		&0	\\
0		&c_1I_k	&0	\\
0		&0		&c_2I_{k^2}
\end{bmatrix}
\begin{bmatrix}
I_k				&0					&0\\
0				&I_k					&0\\
I_k\otimes{\bf p}	&{\bf p}\otimes I_k	&I_{k^2}
\end{bmatrix}
\Lambda_1\\
&=&\Lambda_1
\begin{bmatrix}
c_1I_k	&0		&c_2I_k\otimes{\bf p}^T	\\
0		&c_1I_k	&c_2{\bf p}^T\otimes I_k	\\
0		&0		&c_2I_{k^2}
\end{bmatrix}
\begin{bmatrix}
I_k				&0					&0\\
0				&I_k					&0\\
I_k\otimes{\bf p}	&{\bf p}\otimes I_k	&I_{k^2}
\end{bmatrix}\Lambda_1\\
&=&\Lambda_1
\begin{bmatrix}
(c_1+c_2p)I_k				&c_2p^2J_k				&c_2I_k\otimes{\bf p}^T	\\
c_2p^2J_k					&(c_1+c_2p)I_k			&c_2{\bf p}^T\otimes I_k	\\
c_2I_k\otimes{\bf p}		&c_2{\bf p}\otimes I_k	&c_2I_{k^2}
\end{bmatrix}
\Lambda_1\\
&=&
\begin{bmatrix}
\frac{I_k}{\sqrt{p}}		&0						&0	\\
0						&\frac{I_k}{\sqrt{p}}	&0	\\
0						&0						&\frac{I_k}{p}
\end{bmatrix}
\begin{bmatrix}
(c_1+c_2p)V^2				&c_2p^2VJ_kV				&c_2V(I_k\otimes{\bf p}^T)\Omega_2	\\
c_2p^2VJ_kV					&(c_1+c_2p)V^2			&c_2V({\bf p}^T\otimes I_k)\Omega_2	\\
c_2\Omega_2(I_k\otimes{\bf p})V		&c_2\Omega_2({\bf p}\otimes I_k)V	&c_2\Omega_2^2
\end{bmatrix}
\begin{bmatrix}
\frac{I_k}{\sqrt{p}}		&0						&0	\\
0						&\frac{I_k}{\sqrt{p}}	&0	\\
0						&0						&\frac{I_k}{p}
\end{bmatrix}.
\end{eqnarray*}
Note that $VJ_k=J_kV=0$, 
\[c_1+c_2p=\frac{\binom{m}{2}}{m!d}\frac{(b-d)(a-b)}{k^{m-2}}+\frac{\binom{m}{2}}{m!d}\frac{(a-b)^2}{k^{2(m-2)}}\frac{1}{k}=0,\]
\begin{eqnarray*}
\Omega_2(I_k\otimes{\bf p})V&=&(V_2-p^2V\otimes J_k-p^2J_k\otimes V)(I_k\otimes {\bf p})V\\
&=&V_2(I_k\otimes {\bf p})V-p(V\otimes {\bf p})V
=p^2(V\otimes {\bf p})-p(V\otimes {\bf p})(pI_k-p^2J_k)
=0,
\end{eqnarray*}
and $V({\bf p}^T\otimes I_k)\Omega_2=V(I_k\otimes{\bf p}^T)\Omega_2=\Omega_2({\bf p}\otimes I_k)V=0$, which yields the desired result.

Let $Q=(\Lambda_1R^{-1})^T$ and $Z\sim N(0,I_{k^2})$. Then the covariance matrix $\Sigma$ can be decomposed as
\[\Sigma=(R^{-1})^T\Lambda R^{-1}=(\Lambda_1R^{-1})^T(\Lambda_1R^{-1})=QQ^T.\]
Hence
\[\tilde{\rho}A\tilde{\rho}^T\rightarrow Z^TQ^TAQZ=Z^T\Lambda_1R^{-1}A(R^{-1})^T\Lambda_1Z=c_2Z^T\Omega_2Z.\]
Note $\Omega_2^2=p^2\Omega_2$ implies the eigenvalues of $\Omega_2$ are either 0 or $p^2$ and 
\begin{eqnarray*}
Tr(\Omega_2)&=&Tr\left(V_2-p^2V\otimes J_k-p^2J_k\otimes V\right)\\
&=&Tr\left(p^2I_{k^2}-p^3I_k\otimes J_k-p^3J_k\otimes I_k+p^4J_{k^2}\right)\\
&=&k^2p^2-p^3k^2-p^3k^2+p^4k^2
=\frac{(k-1)^2}{k^2}.
\end{eqnarray*}
Hence $\Omega_2$ has $(k-1)^2$ eigenvalues $p^2$ with other eigenvalues 0. Then $c_2Z^T\Omega_2Z\sim c_2p^2\chi^2_{(k-1)^2}$.

Note that we can rewrite $Z_n$ as
\begin{eqnarray*}
Z_n&=&\frac{\binom{m}{2}}{m!d}\frac{(a-b)^2}{k^{2(m-2)}}\Big(\sum_{s,t}^k\tilde{\rho}_{st}^2-\frac{1}{k}\Big[\sum_{s=1}^k\tilde{\rho}_{s0}^2+\sum_{t=1}^k\tilde{\rho}_{0t}^2\Big]\Big)\\
&=&\frac{1}{2(m-2)!d}\frac{(a-b)^2}{k^{2(m-2)}}\sum_{s,t=1}^k\Big(\frac{1}{\sqrt{n}}\sum_{u=1}^n(I[\sigma_u=s]-\frac{1}{k})(I[\eta_u=t]-\frac{1}{k})\Big)^2.
\end{eqnarray*}
Let $f_j=\frac{1}{\sqrt{n}}\sum_{u=1}^j\Big(\left(1_{[\sigma_u=1]}-\frac{1}{k}\right)\left(1_{[\eta_u=1]}-\frac{1}{k}\right),\dots, \left(1_{[\sigma_u=k]}-\frac{1}{k}\right)\left(1_{[\eta_u=k]}-\frac{1}{k}\right)\Big)^T$ and $d_j=f_j-f_{j-1}$. Then $\|d_j\|^2=\frac{1}{n}\frac{(k-1)^2}{k^2}$ and $b_*^2=\sum_{j=1}^n\|d_j\|^2=\frac{(k-1)^2}{k^2}$. By Theorem 3.5 in \cite{p94}, we have for any $t>0$,
\begin{eqnarray*}
	P\Bigg(\exp\left\{\frac{1}{2(m-2)!d}\frac{(a-b)^2}{k^{2(m-2)}}\|f_n\|^2\right\}>t\Bigg)&=&P\Bigg(\frac{1}{2(m-2)!d}\frac{(a-b)^2}{k^{2(m-2)}}\|f_n\|^2>\log(t)\Bigg)\\
	&=&P\left(\|f_n\|>\sqrt{\frac{\log(t)}{\frac{1}{2(m-2)!d}\frac{(a-b)^2}{k^{2(m-2)}}}}\right)\\
	&\leq&2\exp\left(-\frac{\log(t)}{\kappa(k-1)^2}\right)
	=2t^{-\frac{1}{\kappa(k-1)^2}}.
\end{eqnarray*}
Hence, if $\kappa (k-1)^2<1$, $\{\exp(Z_n)\}_{n=1}^{\infty}$ is uniformly integrable.
\end{proof}

\begin{proof}[Proof of Proposition \ref{proph1}]
For convenience,  we denote $a_1=\frac{a_n}{n^{m-1}}$ and $b_1=\frac{b_n}{n^{m-1}}$.
Under $H_0$, we have $a_1=b_1$, and then
 \[\mathcal{T}=(\mathbb{E}W_1)^{3(m-2l)}\Big[b_1^3-\Big(\frac{b_1^2}{b_1}\Big)^3\Big]=0.\]
Under $H_1$, $k\geq 2$ and $a_1>b_1$.
For $l=1$, direct computation yields
 \[\mathcal{T}=(\mathbb{E}W_1)^{3(m-2)}\frac{(k-1)(a_1-b_1)^3}{k^{3(m-1)}}\neq 0.\]
Next we assume $l\geq 2$, let $E_1=(\mathbb{E}W_1)^{-m}E$, $V_1=(\mathbb{E}W_1)^{-2(m-l)}V$ and $T_1=(\mathbb{E}W_1)^{-3(m-2l)}T$. Then 
 \[\mathcal{T}=(\mathbb{E}W_1)^{3(m-2l)}\Big[T_1-\Big(\frac{V_1}{E_1}\Big)^3\Big].\]
 We calculate $T_1E_1^3-V_1^3$ to get the following
 \begin{eqnarray}\nonumber
 T_1E_1^3-V_1^3&=&(a_1-b_1)^6\frac{1-k^{-1}}{k^{6m-3l-4}}+3(a_1-b_1)^5b_1\Big(\frac{k^l-2}{k^{5m-2l-3}}+\frac{1}{k^{5m-l-4}}\Big)\\  \nonumber
 & &+3(a_1-b_1)^4b_1^2\Big(\frac{k^l-1-k^{-2l+1}}{k^{4m-3l-2}}+\frac{1}{k^{4(m-1)}}\Big)\\ \label{TH1}
& &+(a_1-b_1)^3b_1^3\Big(\frac{1-3k^{-2l+1}}{k^{3m-3l-1}}+\frac{2}{k^{3m-3}}\Big).
 \end{eqnarray}
Clearly, if $k\geq 2$,  $a_1>b_1>0$ and $l\geq 2$, each term in the right hand side of (\ref{TH1}) is positive, which implies that $ T_1E_1^3-V_1^3>0$ and hence $\mathcal{T}\neq0$.
\end{proof}

Before proving Lemma \ref{rate}, we introduce some notation and preliminary. 
For any tensors $A,B,C$, define
\begin{eqnarray*}
C_{2m-l}(A,B)&=&A_{i_1:i_m}B_{i_{m-l+1}:i_{2m-l}}+A_{i_2:i_{m+1}}B_{i_{m-l+2}:i_{2m-l}i_1}+\dots+A_{i_{2m-l}i_1:i_{m-1}}B_{i_{m-l}:i_{2m-l-1}},\\
C_{3(m-l)}(A,B,C)&=&A_{i_1:i_m}B_{i_{m-l+1}:i_{2m-l}}C_{i_{2m-2l+1}:i_{3(m-l)}i_1:i_l}
+A_{i_2:i_{m+1}}B_{i_{m-l+2}:i_{2m-l+1}}C_{i_{2m-2l+2}:i_{3(m-l)}i_1:i_{l+1}}\\
&&+\dots+A_{i_{m-l}:i_{2m-l-1}}B_{i_{2(m-l)}:i_{3(m-l)}i_1:i_{l-1}}C_{i_{3(m-l)}i_1:i_{m-1}}.
\end{eqnarray*}
The proof of Lemma \ref{rate} relies on the following high-moments driven asymptotic result due to Hall and Heyde \cite{HH}.

\begin{Theorem}[Hall and Heyde, 2014]\label{martingale}
 Suppose that for every $n\in\mathbb{N}$ and $\xi_n\rightarrow\infty$ the random variables $X_{n,1},\dots,X_{n,\xi_n}$ are a martingale difference sequence relative to an arbitrary filtration $\mathcal{F}_{n,1}\subset\mathcal{F}_{n,2}$ $\subset$ $\dots$ $\mathcal{F}_{n,\xi_n}$. If (1) $\sum_{i=1}^{\xi_n}\mathbb{E}(X_{n,i}^2|\mathcal{F}_{n,i-1})\rightarrow 1$ in probability,
 (2) $\sum_{i=1}^{\xi_n}\mathbb{E}(X_{n,i}^2I[|X_{n,i}|>\epsilon]|\mathcal{F}_{n,i-1})\rightarrow 0$ in probability for every $\epsilon>0$,
\noindent then $\sum_{i=1}^{\xi_n}X_{n,i}\rightarrow N(0,1)$ in distribution.
 \end{Theorem}
 
 \begin{proof}[Proof of Lemma \ref{rate}]
 	Let $W_{i_1:i_m}=W_{i_1}W_{i_2}\dots W_{i_m}$, 
$\eta_{i_1:i_m}=(a_1-b_1)I[\sigma_{i_1}=\sigma_{i_2}=\dots=\sigma_{i_m}]+b_1$
 and  $\theta_{i_1:i_m}=\eta_{i_1:i_m}W_{i_1:i_m}$. Clearly $\mathbb{E}(A_{i_1:i_m}|W,\sigma)=\theta_{i_1:i_m}$.

 Firstly, we show equation (\ref{Erate}). Write $\widehat{E}-E$ as
 \[\widehat{E}-E=\Big(\widehat{E}-\mathbb{E}(\widehat{E}|W,\sigma)\Big)+\Big(\mathbb{E}(\widehat{E}|W,\sigma)-\mathbb{E}(\widehat{E}|\sigma)\Big)+\Big(\mathbb{E}(\widehat{E}|\sigma)-E\Big).\]
 Note that the three terms in the right hand side are mutually uncorrelated. Hence
\begin{equation}\label{EEsquare}
\mathbb{E}(\widehat{E}-E)^2=\mathbb{E}\Big(\widehat{E}-\mathbb{E}(\widehat{E}|W,\sigma)\Big)^2+\mathbb{E}\Big(\mathbb{E}(\widehat{E}|W,\sigma)-\mathbb{E}(\widehat{E}|\sigma)\Big)^2+\mathbb{E}\Big(\mathbb{E}(\widehat{E}|\sigma)-E\Big)^2.
\end{equation}
It's easy to check that $A_{i_1:i_m}$ and $A_{j_1:j_m}$ are conditionally independent if $i_1:i_m\neq j_1:j_m$. For the first term, we have
 \begin{eqnarray}\nonumber
 \mathbb{E}\Big(\widehat{E}-\mathbb{E}(\widehat{E}|W,\sigma)\Big)^2&=&\mathbb{E}\Bigg(\frac{1}{\binom{n}{m}}\sum_{i\in c(m,n)}(A_{i_1:i_m}-\theta_{i_1:i_m})\Bigg)^2\\ \nonumber
 &=&\frac{1}{\binom{n}{m}^2}\sum_{i\in c(m,n),j\in c(m,n)}\mathbb{E}(A_{i_1:i_m}-\theta_{i_1:i_m})(A_{j_1:j_m}-\theta_{j_1:j_m})\\ \nonumber
 &=&\frac{1}{\binom{n}{m}^2}\sum_{i\in c(m,n)}\mathbb{E}(A_{i_1:i_m}-\theta_{i_1:i_m})^2\\ \nonumber
 &=&\frac{1}{\binom{n}{m}^2}\sum_{i\in c(m,n)}\mathbb{E}\theta_{i_1:i_m}(1-\theta_{i_1:i_m})\\ \nonumber
 &\leq&\frac{1}{\binom{n}{m}^2}\sum_{i\in c(m,n)}\mathbb{E}\theta_{i_1:i_m}\\ \nonumber
 &=&\frac{1}{\binom{n}{m}^2}\sum_{i\in (m,n)}(\mathbb{E}W_1)^m\Big(\frac{a_1+(k^{m-1}-1)b_1}{k^{m-1}}\Big)\\ \label{EEsquare1}
 &=&\frac{(\mathbb{E}W_1)^m}{\binom{n}{m}}\Big(\frac{a_1+(k^{m-1}-1)b_1}{k^{m-1}}\Big)
 =O\Big(\frac{a_1}{n^m}\Big).
 \end{eqnarray}
 For the third term in (\ref{EEsquare}), one has
 \begin{eqnarray}\nonumber
 & &\mathbb{E}\Big(\mathbb{E}(\widehat{E}|\sigma)-E\Big)^2\\   \nonumber
 &=&\mathbb{E}\Big(\frac{1}{\binom{n}{m}}\sum_{i\in c(m,n)}(\mathbb{E}W_1)^m(\eta_{i_1:i_m}-\mathbb{E}\eta_{i_1:i_m})\Big)^2\\  \nonumber
 &=&(\mathbb{E}W_1)^{2m}\mathbb{E}\Big(\frac{1}{\binom{n}{m}}\sum_{i\in c(m,n)}(a_1-b_1)\big(I[\sigma_{i_1}:\sigma_{i_m}]-\mathbb{P}[\sigma_{i_1}:\sigma_{i_m}]\big)\Big)^2\\  \label{EEsquare2}
 &\leq&(\mathbb{E}W_1)^{2m}2(a_1^2+b_1^2)\mathbb{E}\Big(\frac{1}{\binom{n}{m}}\sum_{i\in c(m,n)}\big(I[\sigma_{i_1}:\sigma_{i_m}]-\mathbb{P}[\sigma_{i_1}:\sigma_{i_m}]\big)\Big)^2.
 \end{eqnarray}
 Note that
 \begin{eqnarray}\nonumber
& & \mathbb{E}\Big(\frac{1}{\binom{n}{m}}\sum_{i\in c(m,n)}\big(I[\sigma_{i_1}: \sigma_{i_m}]-\mathbb{P}[\sigma_{i_1}: \sigma_{i_m}]\big)^2\Big)\\  \label{EE1}
&=&\frac{1}{\binom{n}{m}^2}\sum_{i\in c(m,n),j\in c(m,n)}\mathbb{E}\big(I[\sigma_{i_1}:\sigma_{i_m}]-\mathbb{P}[\sigma_{i_1}:\sigma_{i_m}]\big)\big(I[\sigma_{j_1}:\sigma_{j_m}]-\mathbb{P}[\sigma_{j_1}:\sigma_{j_m}]\big)
 \end{eqnarray}
 If there is no repeated index in $i_1:i_m$ and $j_1:j_m$, then
  \[\mathbb{E}\big(I[\sigma_{i_1}:\sigma_{i_m}]-\mathbb{P}[\sigma_{i_1}:\sigma_{i_m}]\big)\big(I[\sigma_{j_1}:\sigma_{j_m}]-\mathbb{P}[\sigma_{j_1}:\sigma_{j_m}]\big)=0.\]
If there is only one repeated index in $i_1:i_m$ and $j_1:j_m$, say, $i_1=j_1$ and other indices are different, then 
 \begin{eqnarray*}
 \mathbb{E}\big(I[\sigma_{i_1}:\sigma_{i_m}]-\mathbb{P}[\sigma_{i_1}:\sigma_{i_m}]\big)\big(I[\sigma_{j_1}:\sigma_{j_m}]-\mathbb{P}[\sigma_{j_1}:\sigma_{j_m}]\big)
 =\frac{k}{k^{2m-1}}-2\frac{k}{k^{m}}\frac{1}{k^{m-1}}+\frac{1}{k^{2(m-1)}}
 =0.
 \end{eqnarray*}
 If there are two or more indices in $i_1:i_m$ and $j_1:j_m$ are the same, it is easy to verify that 
 \[0<\mathbb{E}\big(I[\sigma_{i_1}:\sigma_{i_m}]-\mathbb{P}[\sigma_{i_1}:\sigma_{i_m}]\big)\big(I[\sigma_{j_1}:\sigma_{j_m}]-\mathbb{P}[\sigma_{j_1}:\sigma_{j_m}]\big)\leq 1.\]Hence, by (\ref{EEsquare2}) and (\ref{EE1}), we have
\begin{eqnarray}\label{EWthird}
\mathbb{E}\Big(\mathbb{E}(\widehat{E}|\sigma)-E\Big)^2
=O\Big((a_1^2+b_1^2)\frac{1}{\binom{n}{m}^2}\binom{n}{m}\binom{n}{m-2}\Big)  
=O\Big(\frac{a_1^2}{n^2}\Big).
\end{eqnarray}

 For the second term in (\ref{EEsquare}), we have
 
 \begin{eqnarray}\label{EW}
 & &\mathbb{E}\Big(\mathbb{E}(\widehat{E}|W,\sigma)-\mathbb{E}(\widehat{E}|\sigma)\Big)^2
 =\mathbb{E}\Big(\frac{1}{\binom{n}{m}}\sum_{i\in c(m,n)}\eta_{i_1:i_m}(W_{i_1:i_m}-\mathbb{E}W_{i_1:i_m})\Big)^2.
 \end{eqnarray}
Note that for some constants $c_{s_1}$, $c_{s_1s_2}$, $\dots$, $c_{s_1:s_{m-1}}$ dependent on $\mathbb{E}W_1$, $1\leq s_1,\dots, s_{m-1}\leq m$, one has
 \begin{eqnarray}\nonumber
 W_{i_1:i_m}-\mathbb{E}W_{i_1:i_m}&=&\sum_{s_1=1}^mc_{s_1}(W_{i_{s_1}}-\mathbb{E}W_{i_{s_1}})+\sum_{1\leq s_1\neq s_2\leq m}c_{s_1s_2}(W_{i_{s_1}}-\mathbb{E}W_{i_{s_1}})(W_{i_{s_2}}-\mathbb{E}W_{i_{s_2}})\\ \label{Wexpan}
 & &+\dots+(W_{i_1}-\mathbb{E}W_{i_1})(W_{i_2}-\mathbb{E}W_{i_2}) \dots (W_{i_m}-\mathbb{E}W_{i_m}).
 \end{eqnarray}
 Clearly, the summation terms in (\ref{Wexpan}) are mutually uncorrelated. And for $W_{i_1}-\mathbb{E}W_{i_1}$, we have
 \begin{eqnarray}\nonumber
\mathbb{E}\Big(\frac{1}{\binom{n}{m}}\sum_{i\in c(m,n)}\eta_{i_1:i_m}(W_{i_1}-\mathbb{E}W_{i_1})\Big)^2
 &=&\frac{1}{\binom{n}{m}^2}\sum_{i\in c(m,n),j\in c(m,n)}\mathbb{E}\Big(\eta_{i_1:i_m}\eta_{j_1:j_m}(W_{i_1}-\mathbb{E}W_{i_1})(W_{j_1}-\mathbb{E}W_{j_1})\Big)\\ \label{EWlead}
 &=&\frac{1}{\binom{n}{m}^2}O\Big(a_1^2\binom{n}{m}\binom{n}{m-1}\Big)
 =O\Big(\frac{a_1^2}{n}\Big).
 \end{eqnarray}
 It's easy to verify that the terms $\prod_{s=1}^t(W_{i_s}-\mathbb{E}W_{i_s})$ ($t\geq2$) are of higher order. By equation (\ref{EW}),
  \begin{eqnarray}\label{EW1}
 & &\mathbb{E}\Big(\mathbb{E}(\widehat{E}|W,\sigma)-\mathbb{E}(\widehat{E}|\sigma)\Big)^2
 =O\Big(\frac{a_1^2}{n}\Big).
 \end{eqnarray}
Combining (\ref{EEsquare1}), (\ref{EWthird}) and (\ref{EW1}) yields (\ref{Erate}).

 Next we prove (\ref{Vrate}). We can similarly decompose the mean square as
 \begin{equation}\label{VEsquare}
\mathbb{E}(\widehat{V}-V)^2=\mathbb{E}\Big(\widehat{V}-\mathbb{E}(\widehat{V}|W,\sigma)\Big)^2+\mathbb{E}\Big(\mathbb{E}(\widehat{V}|W,\sigma)-\mathbb{E}(\widehat{V}|\sigma)\Big)^2+\mathbb{E}\Big(\mathbb{E}(\widehat{V}|\sigma)-V\Big)^2.
\end{equation}
 Firstly  we have the following decomposition
 \begin{eqnarray*}
 &&A_{i_1:i_m}A_{i_{m-l+1}:i_{2m-l}}-\theta_{i_1:i_m}\theta_{i_{m-l+1}:i_{2m-l}}\\
 &=&(A_{i_1:i_m}-\theta_{i_1:i_m})(A_{i_{m-l+1}:i_{2m-l}}-\theta_{i_{m-l+1}:i_{2m-l}})\\
 & &+(A_{i_1:i_m}-\theta_{i_1:i_m})\theta_{i_{m-l+1}:i_{2m-l}}+\theta_{i_1:i_m}(A_{i_{m-l+1}:i_{2m-l}}-\theta_{i_{m-l+1}:i_{2m-l}}),
 \end{eqnarray*}
from which it follows
 \begin{eqnarray}\nonumber 
 & &\widehat{V}-\mathbb{E}(\widehat{V}|W,\sigma)\\  \nonumber
 &=&\frac{1}{\binom{n}{2m-l}}\sum_{i\in c(2m-l,n)}\frac{C_{2m-l}(A)-C_{2m-l}(\theta)}{2m-l}\\  \label{EV1}
 &=&\frac{1}{\binom{n}{2m-l}}\sum_{i\in c(2m-l,n)}\frac{C_{2m-l}(A-\theta)}{2m-l}+\frac{1}{\binom{n}{2m-l}}\sum_{i\in c(2m-l,n)}\frac{C_{2m-l}(A-\theta,\theta)+C_{2m-l}(\theta, A-\theta)}{2m-l}.
 \end{eqnarray}
In the last equation of (\ref{EV1}), the first summation and the second summation are conditionally uncorrelated. Hence 
\begin{eqnarray}\nonumber 
 & &\mathbb{E}\Big(\widehat{V}-\mathbb{E}(\widehat{V}|W,\sigma)\Big)^2\\  \nonumber  
 &=&\mathbb{E}\Big(\frac{1}{\binom{n}{2m-l}}\sum_{i\in c(2m-l,n)}\frac{C_{2m-l}(A-\theta)}{2m-l}\Big)^2 \\  \label{EV2}
&& +\mathbb{E}\Big(\frac{1}{\binom{n}{2m-l}}\sum_{i\in c(2m-l,n)}\frac{C_{2m-l}(A-\theta,\theta)+C_{2m-l}(\theta, A-\theta)}{2m-l}\Big)^2.
\end{eqnarray}
The terms in $C_{2m-l}(A-\theta)$ are also conditionally uncorrelated and
 \begin{eqnarray}\nonumber
 & &\mathbb{E}\Big(\frac{1}{\binom{n}{2m-l}}\sum_{i\in c(2m-l,n)}\frac{(A_{i_1:i_m}-\theta_{i_1:i_m})(A_{i_{m-l+1}:i_{2m-l}}-\theta_{i_{m-l+1}:i_{2m-l}})}{2m-l}\Big)^2\\   \nonumber
 &=&\frac{1}{\binom{n}{2m-l}^2}\sum_{i\in c(2m-l,n)}\frac{\mathbb{E}(A_{i_1:i_m}-\theta_{i_1:i_m})^2(A_{i_{m-l+1}:i_{2m-l}}-\theta_{i_{m-l+1}:i_{2m-l}})^2}{(2m-l)^2}\\  \label{EV3}
 &=& \frac{1}{\binom{n}{2m-l}^2}O\Big(a_1^2\binom{n}{2m-l}\Big)
 =O\Big(\frac{a_1^2}{n^{2m-l}}\Big),
 \end{eqnarray}
 which is the order of the first term in (\ref{EV2}).  For the second summand term in (\ref{EV2}), one has
 \begin{eqnarray}\nonumber
 & & \mathbb{E}\Big(\frac{1}{\binom{n}{2m-l}}\sum_{i\in c(2m-l,n)}\frac{(A_{i_1:i_m}-\theta_{i_1:i_m})\theta_{i_{m-l+1}:i_{2m-l}}}{2m-l}\Big)^2\\  \nonumber
 &=&\frac{1}{\binom{n}{2m-l}^2}\sum_{i\in c(m,n),i_m<j_{m+1}<\dots,j_{2m-l}\leq n}\frac{\mathbb{E}(A_{i_1:i_m}-\theta_{i_1:i_m})^2\theta_{i_{m-l+1}:i_{2m-l}}\theta_{i_{m-l+1}:i_mj_{m+1}:j_{2m-l}}}{(2m-l)^2}\\  \label{EV4} 
 &=&\frac{1}{\binom{n}{2m-l}^2}O\Big(a_1^3\binom{n}{2m-l}\binom{n}{m-l}\Big)
 =O\Big(\frac{a_1^3}{n^{m}}\Big).
 \end{eqnarray}
Hence, it follows from (\ref{EV3}) and (\ref{EV4}) that
 \begin{equation}\label{EV5}
 \mathbb{E}\Big(\widehat{V}-\mathbb{E}(\widehat{V}|W,\sigma)\Big)^2=O\Big(\frac{a_1^2}{n^{m}}\Big).
 \end{equation}

 For middle term in (\ref{VEsquare}), by definition, it's equal to
 \begin{eqnarray*} 
 \mathbb{E}\Big(\mathbb{E}(\widehat{V}|W,\sigma)-\mathbb{E}(\widehat{V}|\sigma)\Big)^2=\mathbb{E}\Big(\frac{1}{\binom{n}{2m-l}}\sum_{i\in c(2m-l,n)}\frac{C_{2m-l}(\theta)-\mathbb{E}(C_{2m-l}(\theta)|\sigma)}{2m-l}\Big)^2.
 \end{eqnarray*}
The first term in $C_{2m-l}(\theta)-\mathbb{E}(C_{2m-l}(\theta)|\sigma)$ is 
 \[\Big(W_{i_1:i_{m-l}}W_{i_{m-l+1}:i_{m}}^2W_{i_{m+1}:i_{2m-l}}-(\mathbb{E}W_1^2)^l(\mathbb{E}W_1)^{2(m-l)}\Big)\eta_{i_1:i_{m}}\eta_{i_{m-l+1}:i_{2m-l}},\]
 and we only need to bound this term since the remaining $2m-l-1$ terms can be similarly bounded.  Let $\delta_s=2$ if $s=m-l+1,\dots,m$ and  $\delta_s=1$ otherwise. For generic bounded constants $c_{s_1}$, $c_{s_1s_2}$, $\dots$, $c_{s_1\dots s_{2m-l-1}}$, the following expansion is true.
 \begin{eqnarray}\nonumber
 & &W_{i_1:i_{m-l}}W_{i_{m-l+1}:i_{m}}^2W_{i_{m+1}:i_{2m-l}}-(\mathbb{E}W_1^2)^l(\mathbb{E}W_1)^{2(m-l)}\\  \nonumber
 &=&\sum_{s_1=1}^{2m-l}c_{s_1}(W_{i_{s_1}}^{\delta_{s_1}}-\mathbb{E}W_{i_{s_1}}^{\delta_{s_1}})+\sum_{1\leq s_1\neq s_2\leq2m-l}c_{s_1s_2}(W_{i_{s_1}}^{\delta_{s_1}}-\mathbb{E}W_{i_{s_1}}^{\delta_{s_1}})(W_{i_{s_2}}^{\delta_{s_2}}-\mathbb{E}W_{i_{s_2}}^{\delta_{s_2}})\\  \label{EV7}
 & &+\dots+\prod_{s_1=1}^{2m-l}(W_{i_{s_1}}^{\delta_{s_1}}-\mathbb{E}W_{i_{s_1}}^{\delta_{s_1}})
 \end{eqnarray}
Clearly, the summation terms in (\ref{EV7}) are mutually uncorrelated.  For any $s_1$,  
 \begin{eqnarray*}\nonumber
 & & \mathbb{E}\Big(\frac{1}{\binom{n}{2m-l}}\sum_{i\in c(2m-l,n)}\frac{(W_{i_{s_1}}^{\delta_{s_1}}-\mathbb{E}W_{i_{s_1}}^{\delta_{s_1}})\eta_{i_1:i_{m}}\eta_{i_{m-l+1}:i_{2m-l}}}{2m-l}\Big)^2\\  \nonumber
 &=&\frac{1}{\binom{n}{2m-l}^2}\sum_{i\in c(2m-l,n),j\in c(2m-l,n)}\frac{\mathbb{E}(W_{i_{s_1}}^{\delta_{s_1}}-\mathbb{E}W_{i_{s_1}}^{\delta_{s_1}})(W_{j_{s_1}}^{\delta_{s_1}}-\mathbb{E}W_{j_{s_1}}^{\delta_{s_1}})O(a^4)}{(2m-l)^2}\\   
 &=&\frac{1}{\binom{n}{2m-l}^2}O(a^4)\mathbb{E}(W_{i_{s_1}}^{\delta_{s_1}}-\mathbb{E}W_{i_{s_1}}^{\delta_{s_1}})^2\binom{n}{2m-l}\binom{n}{2m-l-1} 
 =O\Big(\frac{a^4}{n}\Big).
 \end{eqnarray*}
 It's easy to verify that the product terms of $W_{i_{s_1}}^{\delta_{s_1}}-\mathbb{E}W_{i_{s_1}}^{\delta_{s_1}}$ are of higher order. Hence 
 \begin{equation}\label{EV9}
 \mathbb{E}\Big(\mathbb{E}(\widehat{V}|W,\sigma)-\mathbb{E}(\widehat{V}|\sigma)\Big)^2=O\Big(\frac{a_1^4}{n}\Big).
 \end{equation}

The last term in (\ref{VEsquare}) can be expressed as
 \begin{eqnarray}\nonumber
 \mathbb{E}\Big(\mathbb{E}(\widehat{V}|\sigma)-V\Big)^2&=&Var\Big(\frac{1}{\binom{n}{2m-l}}\sum_{c(i,2m-l,n)}\frac{C_{2m-l}(\eta)}{2m-l}\Big)\\ \label{EV10}
 &=&O\Big(Var\Big(\frac{1}{\binom{n}{2m-l}}\sum_{c(i,2m-l,n)}\frac{\eta_{i_1:i_m}\eta_{i_{m-l+1}:i_{2m-l}}}{2m-l}\Big)\Big).
 \end{eqnarray}
 To find the variance, let $H\subset [k]^{2m-l}$. We have
 \begin{eqnarray*}\nonumber
 & &\mathbb{E}\Big(\sum_{i\in c(2m-l,n)}\sum_{(h_{i_s})\in H}\Big(\prod_{s=1}^{2m-l}I[\sigma_{i_s}=h_{i_s}]-\mathbb{E}\prod_{s=1}^{2m-l}I[\sigma_{i_s}=h_{i_s}]\Big)\Big)^2\\  
 &\leq&|H|\sum_{(h_{i_s})\in H}\mathbb{E}\Big(\sum_{i\in c(2m-l,n)}\Big(\prod_{s=1}^{2m-l}I[\sigma_{i_s}=h_{i_s}]-\mathbb{E}\prod_{s=1}^{2m-l}I[\sigma_{i_s}=h_{i_s}]\Big)\Big)^2.
 \end{eqnarray*}
Since
 \begin{eqnarray*}\nonumber
 & &\prod_{s=1}^{2m-l}I[\sigma_{i_s}=h_{i_s}]-\mathbb{E}\prod_{s=1}^{2m-l}I[\sigma_{i_s}=h_{i_s}]\\  \nonumber
 &=&\sum_{s_1=1}^{2m-l}c_{s_1}\Big(I[\sigma_{i_{s_1}}=h_{i_{s_1}}]-\mathbb{E}I[\sigma_{i_{s_1}}=h_{i_{s_1}}]\Big)\\  \nonumber
 & &+\sum_{1\leq s_1\neq s_2\leq 2m-l}c_{s_1s_2}\Big(I[\sigma_{i_{s_1}}=h_{i_{s_1}}]-\mathbb{E}I[\sigma_{i_{s_1}}=h_{i_{s_1}}]\Big)\Big(I[\sigma_{i_{s_2}}=h_{i_{s_2}}]-\mathbb{E}I[\sigma_{i_{s_2}}=h_{i_{s_2}}]\Big)\\   
 &&+\dots+\prod_{s=1}^{2m-l}\Big(I[\sigma_{i_s}=h_{i_s}]-\mathbb{E}I[\sigma_{i_s}=h_{i_s}]\Big),
 \end{eqnarray*}
 and 
 \begin{eqnarray*} \nonumber
 &&\mathbb{E}\Big(\sum_{i\in c(2m-l,n)}\Big(I[\sigma_{i_{s_1}}=h_{i_{s_1}}]-\mathbb{E}I[\sigma_{i_{s_1}}=h_{i_{s_1}}]\Big)\Big)^2\\    \nonumber
 &=&\sum_{i\in c(2m-l,n),j\in c(2m-l,n)}\mathbb{E}\Big(I[\sigma_{i_{s_1}}=h_{i_{s_1}}]-\mathbb{E}I[\sigma_{i_{s_1}}=h_{i_{s_1}}]\Big)\Big(I[\sigma_{j_{s_1}}=h_{j_{s_1}}]-\mathbb{E}I[\sigma_{j_{s_1}}=h_{j_{s_1}}]\Big)\\   
 &=&O\Big(n^{2(2m-l)-1}\Big),
 \end{eqnarray*}
then
 \begin{equation}\label{EV14}
 \mathbb{E}\Big(\sum_{i\in c(2m-l,n)}\sum_{(h_{i_s})\in H}\Big(\prod_{s=1}^{2m-l}I[\sigma_{i_s}=h_{i_s}]-\mathbb{E}\prod_{s=1}^{2m-l}I[\sigma_{i_s}=h_{i_s}]\Big)\Big)^2=O\Big(n^{2(2m-l)-1}\Big).
 \end{equation}
 
 Note that
 \begin{eqnarray*}
 \eta_{i_1:i_m}\eta_{i_{m-l+1}:i_{2m-l}}&=&(a_1-b_1)^2I[\sigma_{i_1}:\sigma_{i_{2m-l}}]+(a_1-b_1)b_1I[\sigma_{i_1}:\sigma_{i_m}]\\
 & &+(a_1-b_1)b_1I[\sigma_{i_{m-l+1}}:\sigma_{i_{2m-l}}]+b_1^2.
 \end{eqnarray*}
 Then by (\ref{EV14}) we have 
 \begin{eqnarray}\nonumber
&& Var\Big(\frac{1}{\binom{n}{2m-l}}\sum_{i\in c(2m-l,n)}\frac{\eta_{i_1:i_m}\eta_{i_{m-l+1}:i_{2m-l}}}{2m-l}\Big)\\  \nonumber
&\asymp&Var\Big(\frac{1}{\binom{n}{2m-l}}\sum_{i\in c(2m-l,n)}\frac{(a_1-b_1)^2I[\sigma_{i_1}:\sigma_{i_{2m-l}}]}{2m-l}\Big)\\  \nonumber
& &+Var\Big(\frac{1}{\binom{n}{2m-l}}\sum_{i\in c(2m-l,n)}\frac{(a_1-b_1)b_1I[\sigma_{i_1}:\sigma_{i_m}]}{2m-l}\Big)\\ \nonumber
& &+Var\Big(\frac{1}{\binom{n}{2m-l}}\sum_{i\in c(2m-l,n)}\frac{(a_1-b_1)b_1I[\sigma_{i_{m-l+1}}:\sigma_{i_{2m-l}}]}{2m-l}\Big)\\  \label{EV15} 
&\asymp&\frac{a_1^4}{\binom{n}{2m-l}^2}n^{2(2m-l)-1}+\frac{a_1^4}{\binom{n}{2m-l}^2}n^{2(2m-l)-1}+\frac{a_1^4}{\binom{n}{2m-l}^2}n^{2(2m-l)-1}
=O\Big(\frac{a_1^4}{n}\Big).
 \end{eqnarray}
 
 By (\ref{EV10}) and (\ref{EV15}),  one gets
 \begin{equation}\label{EV16}
  \mathbb{E}\Big(\mathbb{E}(\widehat{V}|\sigma)-V\Big)^2=O\Big(\frac{a_1^4}{n}\Big).
 \end{equation}
 
From (\ref{EV5}), (\ref{EV9}), (\ref{EV16}) and the condition $n^{l-1}\ll a_n\asymp b_n$, we conclude (\ref{Vrate}).

 In the following, we prove (\ref{Trate}). Similar to the previous proof, we have
   \[\widehat{T}-T=\Big(\widehat{T}-\mathbb{E}(\widehat{T}|W,\sigma)\Big)+\Big(\mathbb{E}(\widehat{T}|W,\sigma)-\mathbb{E}(\widehat{T}|\sigma)\Big)+\Big(\mathbb{E}(\widehat{T}|\sigma)-T\Big),\]
   and
  \begin{equation}\label{TEsquare}
\mathbb{E}(\widehat{T}-T)^2=\mathbb{E}\Big(\widehat{T}-\mathbb{E}(\widehat{T}|W,\sigma)\Big)^2+\mathbb{E}\Big(\mathbb{E}(\widehat{T}|W,\sigma)-\mathbb{E}(\widehat{T}|\sigma)\Big)^2+\mathbb{E}\Big(\mathbb{E}(\widehat{T}|\sigma)-T\Big)^2.
\end{equation}
 For the second expection, one has
 \begin{eqnarray*}\label{ET1}
 \mathbb{E}\Big(\mathbb{E}(\widehat{T}|W,\sigma)-\mathbb{E}(\widehat{T}|\sigma)\Big)^2
 =\mathbb{E}\Big(\frac{1}{\binom{n}{3(m-l)}}\sum_{i\in c(3(m-l),n)}\frac{C_{3(m-l)}(\theta)-\mathbb{E}C_{3(m-l)}(\theta)}{m-l}\Big)^2.
 \end{eqnarray*}
The first term in $C_{3(m-l)}(\theta)-\mathbb{E}C_{3(m-l)}(\theta)$ is
\begin{eqnarray*}
&&\eta_{i_1:i_m}\eta_{i_{m-l+1}:i_{2m-l}}\eta_{i_{2m-2l-1}:i_{3(m-l)}i_1:i_l}\\
&\times&\Big(W_{i_1:i_m}W_{i_{m-l+1}:i_{2m-l}}W_{i_{2m-2l+1}:i_{3(m-l)}i_1:i_l}-\mathbb{E}W_{i_1:i_m}W_{i_{m-l+1}:i_{2m-l}}W_{i_{2m-2l+1}:i_{3(m-l)}i_1:i_l}\Big),
\end{eqnarray*}
and there are $m-1$ terms in it. Let $\delta_s=2$ if $s=m-l+1,\dots,m$ or $s=2m-2l+1,\dots,2m-l$ and  $\delta_s=1$ otherwise. Then following decomposition holds.
 \begin{eqnarray*} 
 & &W_{i_1:i_m}W_{i_{m-l+1}:i_{2m-l}}W_{i_{2m-2l-1}:i_{3(m-l)}i_1:i_l}-\mathbb{E}W_{i_1:i_m}W_{i_{m-l+1}:i_{2m-l}}W_{i_{2m-2l-1}:i_{3(m-l)}i_1:i_l}\\   
 &=&\sum_{s_1=1}^{3(m-l)}c_{s_1}(W_{i_{s_1}}^{\delta_{s_1}}-\mathbb{E}W_{i_{s_1}}^{\delta_{s_1}})+\sum_{1\leq s_1\neq s_2\leq 3(m-l)}^{3(m-l)}c_{s_1s_2}(W_{i_{s_1}}^{\delta_{s_1}}-\mathbb{E}W_{i_{s_1}}^{\delta_{s_1}})(W_{i_{s_2}}^{\delta_{s_2}}-\mathbb{E}W_{i_{s_2}}^{\delta_{s_2}})\\   
 & &+\dots+\prod_{s_1=1}^{3(m-l)}(W_{i_{s_1}}^{\delta_{s_1}}-\mathbb{E}W_{i_{s_1}}^{\delta_{s_1}}). 
 \end{eqnarray*}
Note that
\begin{eqnarray*}\nonumber
& &\mathbb{E}\Big(\frac{1}{\binom{n}{3(m-l)}}\sum_{i\in c(3(m-l),n)}\frac{\eta_{i_1:i_m}\eta_{i_{m-l+1}:i_{2m-l}}\eta_{i_{2m-2l-1}:i_{3(m-l)}i_1:i_l}(W_{i_{s_1}}^{\delta_{s_1}}-\mathbb{E}W_{i_{s_1}}^{\delta_{s_1}})}{m-l}\Big)^2\\  \label{ET3}
&=&\frac{1}{\binom{n}{3(m-l)}^2}\sum_{i\in c(3(m-l),n), j\in c(3(m-l),n)}\frac{O(a_1^6)\mathbb{E}(W_{i_{s_1}}^{\delta_{s_1}}-\mathbb{E}W_{i_{s_1}}^{\delta_{s_1}})(W_{j_{s_1}}^{\delta_{s_1}}-\mathbb{E}W_{j_{s_1}}^{\delta_{s_1}})}{(m-l)^2} 
=O\Big(\frac{a_1^6}{n}\Big),
\end{eqnarray*}
and the product terms of $W_{i_{s_1}}^{\delta_{s_1}}-\mathbb{E}W_{i_{s_1}}^{\delta_{s_1}}$ are of higher order.
 Hence,
 \begin{equation}\label{ET4}
 \mathbb{E}\Big(\mathbb{E}(\widehat{T}|W,\sigma)-\mathbb{E}(\widehat{T}|\sigma)\Big)^2=O\Big(\frac{a_1^6}{n}\Big).
 \end{equation}

For the third expectation in (\ref{TEsquare}), similar to  (\ref{EV10}), one has
 \begin{eqnarray}\nonumber
 \mathbb{E}\Big(\mathbb{E}(\widehat{T}|\sigma)-T\Big)^2&=&Var\Big(\frac{1}{\binom{n}{3(m-l)}}\sum_{i\in c(3(m-l),n)}\frac{C_{3(m-l)(\eta)}}{m-l}\Big)\\  \nonumber
 &\asymp&Var\Big(\frac{1}{\binom{n}{3(m-l)}}\sum_{i\in c(3(m-l),n)}\frac{\eta_{i_1:i_m}\eta_{i_{m-l+1}:i_{2m-l}}\eta_{i_{2m-2l-1}:i_{3(m-l)}i_1:i_l}}{m-l}\Big)\\  \label{ET5}
 &\asymp&O(\frac{a_1^6}{n}).
 \end{eqnarray}

For the first expectation in (\ref{TEsquare}), note that
 \begin{equation*}
 \widehat{T}-\mathbb{E}(\widehat{T}|W,\sigma)=\frac{1}{\binom{n}{3(m-l)}}\sum_{i\in c(3(m-l),n)}\frac{C_{3(m-l)}(A)-C_{3(m-l)}(\theta)}{m-l}.
 \end{equation*}
 The first term in it can be decomposed as
 \begin{eqnarray*}\nonumber
 &&A_{i_1:i_m}A_{i_{m-l+1}:i_{2m-l}}A_{i_{2m-2l-1}:i_{3(m-l)}i_1:i_l}-\theta_{i_1:i_m}\theta_{i_{m-l+1}:i_{2m-l}}\theta_{i_{2m-2l-1}:i_{3(m-l)}i_1:i_l}\\  \nonumber
 &=&(A_{i_1:i_m}-\theta_{i_1:i_m})\theta_{i_{m-l+1}:i_{2m-l}}\theta_{i_{2m-2l-1}:i_{3(m-l)}i_1:i_l}+(A_{i_{m-l+1}:i_{2m-l}}-\theta_{i_{m-l+1}:i_{2m-l}})\theta_{i_1:i_m}\theta_{i_{2m-2l-1}:i_{3(m-l)}i_1:i_l}\\  \nonumber
 & &+(A_{i_{2m-2l-1}:i_{3(m-l)}i_1:i_l}-\theta_{i_{2m-2l-1}:i_{3(m-l)}i_1:i_l})\theta_{i_1:i_m}\theta_{i_{m-l+1}:i_{2m-l}}\\  \nonumber
 & &+(A_{i_1:i_m}-\theta_{i_1:i_m})(A_{i_{m-l+1}:i_{2m-l}}-\theta_{i_{m-l+1}:i_{2m-l}})\theta_{i_{2m-2l-1}:i_{3(m-l)}i_1:i_l}+\dots\\
 & &+(A_{i_1:i_m}-\theta_{i_1:i_m})(A_{i_{m-l+1}:i_{2m-l}}-\theta_{i_{m-l+1}:i_{2m-l}})(A_{i_{2m-2l-1}:i_{3(m-l)}i_1:i_l}-\theta_{i_{2m-2l-1}:i_{3(m-l)}i_1:i_l}).
 \end{eqnarray*}
 Note that 
 \begin{eqnarray}\nonumber
 & &\mathbb{E}\Big(\frac{1}{\binom{n}{3(m-l)}}\sum_{i\in c(3(m-l),n)}(A_{i_1:i_m}-\theta_{i_1:i_m})(A_{i_{m-l+1}:i_{2m-l}}-\theta_{i_{m-l+1}:i_{2m-l}})\theta_{i_{2m-2l-1}:i_{3(m-l)}i_1:i_l}\Big)^2\\  \label{newet1}
 &=&O\Big(\frac{a_1^4}{\binom{n}{3(m-l)}^2}n^{3(m-l)}n^{3(m-l)-(2m-l)}\Big)
 =O\Big(\frac{a_1^4}{n^{2m-l}}\Big), 
 \end{eqnarray}
 and
  \begin{eqnarray}\nonumber
 & &\mathbb{E}\Big(\frac{1}{\binom{n}{3(m-l)}}\sum_{i\in c(3(m-l),n)}(A_{i_1:i_m}-\theta_{i_1:i_m})\theta_{i_{m-l+1}:i_{2m-l}}\theta_{i_{2m-2l-1}:i_{3(m-l)}i_1:i_l}\Big)^2\\  \label{newet2}
 &=&O\Big(\frac{a_1^5}{\binom{n}{3(m-l)}^2}n^{3(m-l)}n^{3(m-l)-m}\Big)
 =O\Big(\frac{a_1^5}{n^{m}}\Big). 
 \end{eqnarray}
 Let 
 \[G_{i_1:i_{3(m-l)}}=(A_{i_1:i_m}-\theta_{i_1:i_m})(A_{i_{m-l+1}:i_{2m-l}}-\theta_{i_{m-l+1}:i_{2m-l}})(A_{i_{2m-2l-1}:i_{3(m-l)}i_1:i_l}-\theta_{i_{2m-2l-1}:i_{3(m-l)}i_1:i_l}).\]
 Then
 \begin{eqnarray}\nonumber
 \mathbb{E}\Big(\frac{1}{\binom{n}{3(m-l)}}\sum_{i\in c(3(m-l),n)}G_{i_1:i_{3(m-l)}}\Big)^2
 &=&\frac{1}{\binom{n}{3(m-l)}^2}\sum_{i\in c(3(m-l),n)}\mathbb{E}G_{i_1:i_{3(m-l)}}^2\\  \nonumber
 &\asymp&\frac{1}{\binom{n}{3(m-l)}^2}\sum_{i\in c(3(m-l),n)}\mathbb{E}\theta_{i_1:i_m}\theta_{i_{m-l+1}:i_{2m-l}}\theta_{i_{2m-2l-1}:i_{3(m-l)}i_1:i_l}\\   \label{newet3}
 &=&\frac{T}{\binom{n}{3(m-l)}}
 =O\Big(\frac{a_1^3}{n^{3(m-l)}}\Big).
 \end{eqnarray}
Under the condition $a_n\asymp b_n\ll n^{\frac{3l-2}{3}}$, by (\ref{TEsquare}) , (\ref{ET4}),  (\ref{ET5}), (\ref{newet1}), (\ref{newet2}) and (\ref{newet3}),    we get (\ref{Trate}).
 
 In the end, we show the asymptotic normality by using Theorem \ref{martingale}. Let 
 \[\mathcal{W}_n=\Big\{\Big|\frac{1}{n}\sum_{i=1}^nW_i^2-
 \mathbb{E}W_1^2\Big|\leq n^{-\frac{1}{3}}\Big\},\ \ \Theta_n=\sqrt{\mathbb{E}\Big(\sum_{i\in c(3(m-l),n)}G_{i_1:i_{3(m-l)}}\Big)^2}.\]
 Clearly, $\Theta_n\asymp \sqrt{n^{3(m-l)}a_1^3}\rightarrow\infty$ if $n^{l-1}\ll a_n\asymp b_n$. Define 
 \[S_{n,t}=\frac{\sum_{i\in c(3(m-l),t)}G_{i_1:i_{3(m-l)}}}{\Theta_n},\]
 and let $X_{n,t}=S_{n,t}-S_{n,t-1}$. We show the asymptotic normality by applying the martingale central limit theorem to $X_{n,t}$ conditioning on $W$ and $\sigma$. Simple calculation yields that
 \[X_{n,t}=\frac{\sum_{i\in c(3(m-l)-1,t-1)}G_{i_1:i_{3(m-l)-1}t}}{\Theta_n},\]
and $\mathbb{E}(X_{n,t}|\mathcal{F}_{n,t-1})=0$. Hence, $X_{n,t}$ is martingale difference. Note that
 \begin{eqnarray}\label{M1}
&& \mathbb{E}\Big(\sum_{t=1}^n\mathbb{E}(S_{n,t}-S_{n,t-1})^2|\mathcal{F}_{n,t-1}, W,\sigma\Big)  \\ \nonumber
 &=&\sum_{t=1}^n\Big(\mathbb{E}(S_{n,t}^2|W,\sigma)-\mathbb{E}(S_{n,t-1}^2|W,\sigma)\Big) 
 =\mathbb{E}(S_{n,n}^2|W,\sigma)=1,
 \end{eqnarray}
 and 
 \begin{eqnarray*}\nonumber
 & &Var\Big(\sum_{t=1}^n\mathbb{E}[(S_{n,t}-S_{n,t-1})^2|\mathcal{F}_{n,t-1}, W,\sigma]\Big)\\  \nonumber
 &=&\frac{1}{\Theta_n^4}Var\Big(\sum_{t=1}^n\mathbb{E}\Big(\sum_{i\in c(3(m-l)-1,t-1)}G_{i_1:i_{3(m-l)-1}t}\Big)^2|\mathcal{F}_{n,t-1}, W,\sigma\Big)\\  \nonumber
 \end{eqnarray*}
 \begin{eqnarray*}
 &=&\frac{1}{\Theta_n^4}Var\Big(\sum_{t=1}^n\sum_{i\in c(3(m-l)-1,t-1)}(A_{i_1:i_m}-\theta_{i_1:i_m})^2(A_{i_{m-l+1}:i_{2m-l}}-\theta_{i_{m-l+1}:i_{2m-l}})^2O(a_1)|W,\sigma\Big)\\ \nonumber
 &=&\frac{O(a_1^2)}{\Theta_n^4}\sum_{s,t=1}^n\sum_{i\in c(3(m-l)-1,s-1),j\in c(3(m-l)-1,t-1)}Cov\Big((A_{i_1:i_m}-\theta_{i_1:i_m})^2(A_{i_{m-l+1}:i_{2m-l}}-\theta_{i_{m-l+1}:i_{2m-l}})^2,\\  \nonumber
 & &(A_{j_1:j_m}-\theta_{j_1:j_m})^2(A_{j_{m-l+1}:j_{2m-l}}-\theta_{j_{m-l+1}:j_{2m-l}})^2\Big)\\  \nonumber
 &=&\frac{O(a_1^5)}{\Theta_n^4}\sum_{s\leq t}\binom{t}{3(m-l)-1}\binom{s}{2m-3l-1}\\    \nonumber
 &=&\frac{O(a_1^5)}{\Theta_n^4}\sum_{s=1}^n\binom{s}{3(m-l)-1}\sum_{s=1}^n\binom{s}{2m-3l-1}\\  \label{M2}   
 &=&\frac{O(a_1^5)}{\Theta_n^4}n^{3(m-l)}n^{2m-3l}  
 =\frac{O(a_1^5)}{n^{6(m-l)}a_1^6}n^{3(m-l)}n^{2m-3l}
 =\frac{1}{a_1n^m}\rightarrow 0.
 \end{eqnarray*}
Equations (\ref{M1}) and (\ref{M2}) implies that
 \begin{equation*} 
 \sum_{t=1}^n\mathbb{E}\Big((S_{n,t}-S_{n,t-1})^2|\mathcal{F}_{n,t-1}, W,\sigma\Big)\rightarrow 1,
 \end{equation*}
 which is condition (1) in Theorem \ref{martingale}.
 
 Next we check the Lindeberg condition. For any $\epsilon>0$, we have
 \begin{eqnarray}\nonumber
 &&\sum_{t=1}^n\mathbb{E}\Big((S_{n,t}-S_{n,t-1})^2I[|S_{n,t}-S_{n,t-1}|>\epsilon]\Big|\mathcal{F}_{n,t-1}, W,\sigma\Big)\\  \nonumber
 &\leq&\sum_{t=1}^n\sqrt{\mathbb{E}\Big((S_{n,t}-S_{n,t-1})^4\Big|\mathcal{F}_{n,t-1}, W,\sigma\Big)}\sqrt{\mathbb{P}[|S_{n,t}-S_{n,t-1}|>\epsilon]\Big|\mathcal{F}_{n,t-1}, W,\sigma\Big)}\\    \nonumber
 &\leq&\frac{1}{\epsilon^2}\sum_{t=1}^n\mathbb{E}\Big((S_{n,t}-S_{n,t-1})^4\Big|\mathcal{F}_{n,t-1}, W,\sigma\Big)\\     \label{M4}
 &=& \frac{1}{\epsilon^2\Theta_n^4}\sum_{t=1}^n\mathbb{E}\Big(\Big(\sum_{i\in c(3(m-l)-1,t-1)}G_{i_1:i_{3(m-l)-1}t}\Big)^4\Big|\mathcal{F}_{n,t-1}, W,\sigma\Big).
 \end{eqnarray}
 For convenience, let $c(i)=c(i,3(m-l)-1,t-1)$, $D_{1i}=A_{i_1:i_m}-\theta_{i_1:i_m}$, $D_{2i}=A_{i_{m-l+1}:i_{2m-l}}-\theta_{i_{m-l+1}:i_{2m-l}}$, and
  $D_{3i}=A_{i_{2m-2l-1}:i_{3(m-l)}i_1:i_l}-\theta_{i_{2m-2l-1}:i_{3(m-l)}i_1:i_l}$. Then
 \begin{eqnarray}\nonumber
 &&\mathbb{E}\Big(\Big(\sum_{i\in c(3(m-l)-1,t-1)}G_{i_1:i_{3(m-l)-1}t}\Big)^4|\mathcal{F}_{n,t-1}, W,\sigma\Big)\\  \label{M5}
 &=&\mathbb{E}\Big(\sum_{c(i),c(j),c(r),c(s)}D_{1i}D_{2i}D_{3i}D_{1j}D_{2j}D_{3j}D_{1r}D_{2r}D_{3r}D_{1s}D_{2s}D_{3s}\Big|\mathcal{F}_{n,t-1}, W,\sigma\Big).
 \end{eqnarray}
 For indices $i_{2m-2l-1}:i_{3(m-l)}i_1:i_l$, $j_{2m-2l-1}:j_{3(m-l)}j_1:j_l$, $r_{2m-2l-1}:r_{3(m-l)}r_1:r_l$ and $s_{2m-2l-1}:s_{3(m-l)}s_1:s_l$, where $i_{3(m-l)}=j_{3(m-l)}=r_{3(m-l)}=s_{3(m-l)}=t$, either all of them are the same or two of them are the same and the other two are the same. Otherwise, the conditional expectation in (\ref{M5}) given $W, \sigma$ vanishes. The same is true for the other two sets of indices. We consider the case $i_1:i_{3(m-l)-1}=j_{1}:j_{3(m-l)-1}$ and 
$r_1:r_{3(m-l)-1}=s_{1}:s_{3(m-l)-1}$ for example.
 Then by (\ref{M5}), (\ref{M4}) is equal to
 \begin{eqnarray*} \label{M6}
\frac{1}{\epsilon^2\Theta_n^4}\sum_{t=1}^n\Big(\sum_{c(i),c(r)}\mathbb{E}D_{1i}^2D_{2i}^2D_{3i}^2D_{1r}^2D_{2r}^2D_{3r}^2\Big|\mathcal{F}_{n,t-1}, W,\sigma\Big)
 =\frac{nO(a_1^6)}{\epsilon^2n^{6(m-l)}a_1^6}n^{3(m-l)-1}n^{3(m-l)-1}\rightarrow0.
 \end{eqnarray*}
 The other cases can be similarly proved. Hence, 
 \begin{equation*}
 \sum_{t=1}^n\mathbb{E}\Big((S_{n,t}-S_{n,t-1})^2I[|S_{n,t}-S_{n,t-1}|>\epsilon]\Big|\mathcal{F}_{n,t-1}, W,\sigma\Big)\rightarrow0,
 \end{equation*}
 which is condition (2) in Theorem \ref{martingale}. Then we conclude that conditional on $W\in\mathcal{W}_n$ and $\sigma$,
 \begin{eqnarray}\label{M7}
 \frac{\sum_{i\in c(3(m-l),n)}G_{i_1:i_{3(m-l)}}}{\Theta_n}
 =\sum_{t=1}^n(S_{n,t}-S_{n,t-1})\rightarrow N(0,1).
 \end{eqnarray}
Since $\Theta_n\asymp\sqrt{\binom{n}{3(m-l)}T}$, and 
 \[\binom{n}{3(m-l)}(m-l)\Big(\widehat{T}-T\Big)=\sum_{i\in c(3(m-l),n)}G_{i_1:i_{3(m-l)}}+\dots+\sum_{i\in c(3(m-l),n)}G_{i_{m-l}:i_{3(m-l)i_1\dots i_{m-l-1}}}+o(1),\]
then (\ref{Tnormal}) follows from the fact the terms in the right hand side of the above equation are uncorrelated and a similar argument as in proving (\ref{M7}). 

 \end{proof}

 \begin{proof}[Proof of Corollary \ref{boundedcor1}]

The proof is similar to Theorems \ref{contiguous}. 
Note that $\mathcal{H}_2$ and $\mathcal{H}_3$ are independent given $\sigma$. Suppose $k=2$ for simplicity.
Rewrite
\[p_{ij}(\sigma)=\frac{(a_2-b_2)I[\sigma_i=\sigma_j]+b_2}{n^{\alpha_2}},\]
\[p_{ijk}(\sigma)=\frac{(a_3-b_3)I[\sigma_i=\sigma_j=\sigma_k]+b_3}{n^{\alpha_3}},\]
\[q_{ij}=1-p_{ij}(\sigma),\ \ q_{ijk}(\sigma)=1-p_{ijk}(\sigma),\]
\[Y_n=E_{\sigma}\prod_{i<j}\Big(\frac{p_{ij}(\sigma)}{p_{02}}\Big)^{A_{ij}}\Big(\frac{q_{ij}(\sigma)}{q_{02}}\Big)^{1-A_{ij}}\prod_{i<j<k}\Big(\frac{p_{ijk}(\sigma)}{p_{03}}\Big)^{A_{ij7k}}\Big(\frac{q_{ijk}(\sigma)}{q_{03}}\Big)^{1-A_{ijk}}.\]

Then by a similar analysis of equation (\ref{ctg}) in the paper it yields that
\begin{eqnarray}\nonumber
E_0Y_n^2&=&E_{\sigma\eta}\prod_{i<j}\Big(\frac{p_{ij}(\sigma)p_{ij}(\eta)}{q_{02}}+\frac{q_{ij}(\sigma)q_{ij}(\eta)}{q_{02}}\Big)\\  \nonumber
&&\times \prod_{i<j<k}\Big(\frac{p_{ijk}(\sigma)p_{ijk}(\eta)}{q_{03}}+\frac{q_{ijk}(\sigma)q_{ijk}(\eta)}{q_{03}}\Big)\\  \nonumber
&=&(1+o(1))E_{\sigma\eta}\exp\Big\{\frac{(a_2-d_2)^2}{d_2n^{\alpha_2}}s_2^{(2)}+\frac{(a_2-d_2)(b_2-d_2)}{d_2n^{\alpha_2}}s_1^{(2)}+\frac{(b_2-d_2)^2}{d_2n^{\alpha_2}}s_0^{(2)}\Big\}\\  \label{eq1}
&&\times \exp\Big\{\frac{(a_3-d_3)^2}{d_3n^{\alpha_3}}s_2^{(3)}+\frac{(a_3-d_3)(b_3-d_3)}{d_3n^{\alpha_3}}s_1^{(3)}+\frac{(b_3-d_3)^2}{d_3n^{\alpha_3}}s_0^{(3)}\Big\},
\end{eqnarray}
where $s_2^{(m)}=\#\{1\leq i_1<\dots<i_m\leq n:I[\sigma_{i_1}=\dots=\sigma_{i_m}]+I[\eta_{i_1}=\dots=\eta_{i_m}]=2\}$ for $m=2,3$, $s_1^{(m)}$ and $s_0^{(m)}$ are similarly defined and $a_m,b_m,d_m$ corresponds to $\mathcal{H}_m$ for $m=2,3$.

For $\alpha_m\geq m$, $s_2^{(m)}$,$s_1^{(m)}$, $s_0^{(m)}$ are bounded by $n^{\alpha_m}$. By the proof of Theorem \ref{contiguous}, $E_0Y_n^2\rightarrow1$.

For $\alpha_m\geq m-1+\delta$ with $\delta\in(0,1)$, after a transformation of $s_2^{(m)}$,$s_1^{(m)}$, $s_0^{(m)}$ as in the proof of Theorem \ref{contiguous}, each term in the exponent of (\ref{eq1}) is uniformly integrable. Hence, $E_0Y_n^2\rightarrow1$.
\end{proof}

\begin{proof}[Proof of Corollary \ref{boundedcor2}]

The proof is similar to Theorem \ref{ctgbounded}.
Let \[\lambda_h(m)=\frac{1}{2h}\Big(\frac{a+(k^{m-1}-1)b}{k^{m-1}(m-2)!}\Big)^h,\]
\[\delta_h(m)=(k-1)\Big(\frac{a-b}{a+(k^{m-1}-1)b}\Big)^h.\]
Condition (A1) follows from Lemma \ref{cyclepois}.

Next, we check condition (A2). Let $S=\{1,2,\dots,k\}$ and $H_m=(H_{mhi})_{2\leq h\le s,1\le i\le j_s}$ be a $sj_s$-tuple of $h$-edge loose cycle $H_{mhi}$ in $\mathcal{H}_m$ $(m=2,3)$ for any integers $s\ge2$ and $j_s\ge1$. Define $X_{mhn}$ as the number of $h$-edge loose cycles in the hypergraph $\mathcal{H}_m$ $(m=2,3)$ and $[x]_j=x(x-1)\dots (x-j+1)$. Note that for any sequence of positive integers $j_2$,$\dots$, $j_s$, we have
\begin{eqnarray}\label{ceyx}
&&\mathbb{E}_0Y_n[X_{23n}]_{j_2}\dots[X_{2sn}]_{j_s}[X_{32n}]_{t_2}\dots[X_{3ln}]_{t_l}\\  \nonumber
&&=\sum_{H_2\in B_2,H_3\in B_3}\mathbb{E}_0Y_n1_{H_2}1_{H_3}+\sum_{H_2\in B_2,H_3\in \overline{B}_3}\mathbb{E}_0Y_n1_{H_2}1_{H_3}\\ \nonumber
&&+\sum_{H_2\in \overline{B}_2,H_3\in B_3}\mathbb{E}_0Y_n1_{H_2}1_{H_3}+\sum_{H_2\in \overline{B}_2,H_3\in \overline{B}_3}\mathbb{E}_0Y_n1_{H_2}1_{H_3}\\  \nonumber
&&=\sum_{H_2\in B_2,H_3\in B_3}\mathbb{E}_0Y_n1_{H_2}1_{H_3}+o(1)\\  \nonumber
&&\rightarrow  \prod_{h=3}^s[\lambda_h(2)(1+\delta_h(2))]^{j_h}\prod_{h=2}^l[\lambda_h(3)(1+\delta_h(3))]^{t_h}.
\end{eqnarray}
where $B_m$ is the collection of disjoint tuples $H_m$ and $\overline{B}_m$ is the complement, that is, for any $H_m\in \overline{B}_m$, there exist two cycles in $H_{m}$ that have at least one vertex in common.
The second equality and the last step follows from the proof of Theorem \ref{ctgbounded}.

Then we check condition (A3). By (A1) and (A2), we have $\frac{\mu_h(m)}{\lambda_h(m)}-1=\frac{\lambda_h(m)(1+\delta_h(m))}{\lambda_h(m)}-1=\delta_h(m)$. Besides, $\lambda_h(m)\delta_h(m)^2=\frac{1}{2h}\Big(\frac{(a-b)^2}{k^{m-1}(m-2)!(a+(k^{m-1}-1)b)}\Big)^h=\frac{\kappa_m^h}{2h}$. If $\kappa_m<1$, then $\sum_{h=2}^{\infty}\lambda_h(m)\delta_h(m)^2<\infty$.

Lastly, we check condition (A4). Then by a similar analysis of equation (\ref{ctg}) in the paper it yields that
\begin{eqnarray}\nonumber
E_0Y_n^2&=&E_{\sigma\eta}\prod_{i<j}\Big(\frac{p_{ij}(\sigma)p_{ij}(\eta)}{q_{02}}+\frac{q_{ij}(\sigma)q_{ij}(\eta)}{q_{02}}\Big)\\  \nonumber
&&\times \prod_{i<j<k}\Big(\frac{p_{ijk}(\sigma)p_{ijk}(\eta)}{q_{03}}+\frac{q_{ijk}(\sigma)q_{ijk}(\eta)}{q_{03}}\Big)\\  \nonumber
&=&(1+o(1))E_{\sigma\eta}\exp\Big\{\frac{(a_2-d_2)^2}{d_2n^{\alpha_2}}s_2^{(2)}+\frac{(a_2-d_2)(b_2-d_2)}{d_2n^{\alpha_2}}s_1^{(2)}+\frac{(b_2-d_2)^2}{d_2n^{\alpha_2}}s_0^{(2)}\Big\}\\  \label{eq1}
&&\times \exp\Big\{\frac{(a_3-d_3)^2}{d_3n^{\alpha_3}}s_2^{(3)}+\frac{(a_3-d_3)(b_3-d_3)}{d_3n^{\alpha_3}}s_1^{(3)}+\frac{(b_3-d_3)^2}{d_3n^{\alpha_3}}s_0^{(3)}\Big\},
\end{eqnarray}
where $s_2^{(m)}=\#\{1\leq i_1<\dots<i_m\leq n:I[\sigma_{i_1}=\dots=\sigma_{i_m}]+I[\eta_{i_1}=\dots=\eta_{i_m}]=2\}$ for $m=2,3$,
 $s_1^{(m)}$ and $s_0^{(m)}$ are similarly defined and $a_m,b_m,d_m$ correspond to $\mathcal{H}_m$ for $m=2,3$.

Let $c_1^{(m)}=\frac{\binom{m}{2}}{m!d}\frac{(a-b)(b-d)}{k^{m-2}}$, $c_2^{(m)}=\frac{\binom{m}{2}}{m!d}\frac{(a-b)^2}{k^{2(m-2)}}$,
\begin{eqnarray*}
\tilde{Z}_n&=&c_2^{(3)}\sum_{s,t=1}^k\tilde{\rho}_{st}^2\Big[1+\sum_{i=1}^{m-2}\frac{1}{k^{2i}}\frac{\binom{m}{i+2}}{\binom{m}{2}}\Big(\frac{\tilde{\rho}_{st}}{\sqrt{n}}\Big)^i\Big]\\
& &+c_1^{(3)}\Big(\sum_{t=1}^k\tilde{\rho}_{0t}^2\Big[1+\sum_{i=1}^{m-2}\frac{1}{k^{i}}\frac{\binom{m}{i+2}}{\binom{m}{2}}\Big(\frac{\tilde{\rho}_{0t}}{\sqrt{n}}\Big)^i\Big]+\sum_{s=1}^k\tilde{\rho}_{s0}^2\Big[1+\sum_{i=1}^{m-2}\frac{1}{k^{i}}\frac{\binom{m}{i+2}}{\binom{m}{2}}\Big(\frac{\tilde{\rho}_{s0}}{\sqrt{n}}\Big)^i\Big]\Big)\\
&&\leq c_2^{(3)}(1+\frac{1}{3k^2})\sum_{s,t=1}^k\tilde{\rho}_{st}^2
\end{eqnarray*}
and $Z_n=c_2^{(2)}\sum_{s,t=1}^k\tilde{\rho}_{st}^2+c_1^{(2)}\Big(\sum_{t=1}^k\tilde{\rho}_{0t}^2+\sum_{s=1}^k\tilde{\rho}_{s0}^2\Big)$. Then

\begin{eqnarray*}
E_0Y_n^2&=&(1+o(1))E_{\sigma\eta}\exp\{Z_n+\tilde{Z}_n\}\\
&=&(1+o(1))E_{\sigma\eta}\exp\Big\{(c_2^{(2)}+c_2^{(3)})\sum_{s,t=1}^k\tilde{\rho}_{st}^2+(c_1^{(2)}+c_1^{(3)})\Big(\sum_{t=1}^k\tilde{\rho}_{0t}^2+\sum_{s=1}^k\tilde{\rho}_{s0}^2\Big)+o_P(1)\Big\}.
\end{eqnarray*}

Note that
\[Z_n+\tilde{Z}_n\leq \Big( c_2^{(2)}+c_2^{(3)}(1+\frac{1}{3k^2})\Big)\sum_{s,t=1}^k\tilde{\rho}_{st}^2 .\]
Denote $\tau=c_2^{(2)}+c_2^{(3)}(1+\frac{1}{3k^2})$. Let $f_j=\frac{1}{\sqrt{n}}\sum_{u=1}^j\Big(\left(1_{[\sigma_u=1]}1_{[\eta_u=1]}-\frac{1}{k^2}\right),\dots, \left(1_{[\sigma_u=k]}1_{[\eta_u=k]}-\frac{1}{k^2}\right)\Big)^T$ and $d_j=f_j-f_{j-1}$. Then $\|d_j\|^2=\frac{1}{n}\frac{k^2-1}{k^2}$ and $b_*^2=\sum_{j=1}^n\|d_j\|^2=\frac{k^2-1}{k^2}$. By Theorem 3.5 in Pinelis (\cite{p94}), for any $t>0$,
\begin{eqnarray*}
	P\Bigg(\exp\left\{c\|f_n\|^2\right\}>t\Bigg)&=&P\Bigg(c\|f_n\|^2>\log(t)\Bigg)
	=P\left(\|f_n\|>\sqrt{\frac{\log(t)}{c}}\right)\\
	&\leq&2\exp\left(-\frac{\log(t)}{2cb_*^2}\right)
	=2t^{-\frac{1}{\big(\kappa_2+\frac{\kappa_3}{3}(1+\frac{1}{3k^2})\big)(k^2-1)}}.
\end{eqnarray*}
Hence, the condition $\big(\kappa_2+\frac{\kappa_3}{3}(1+\frac{1}{3k^2})\big)(k^2-1)<1$ implies that $\{\exp\{Z_n+\tilde{Z}_n\}\}_{n=1}^{\infty}$ is uniformly integrable.

By the proof of Lemma \ref{matrixde}, we conclude that
\[(c_2^{(2)}+c_2^{(3)})\sum_{s,t=1}^k\tilde{\rho}_{st}^2+(c_1^{(2)}+c_1^{(3)})\Big(\sum_{t=1}^k\tilde{\rho}_{0t}^2+\sum_{s=1}^k\tilde{\rho}_{s0}^2\Big)\]
converges in distribution to $(c_2^{(2)}+c_2^{(3)})k^{-2}\chi^2_{(k-1)^2}$. Then it follows that
\begin{eqnarray*}
\mathbb{E}_0Y_n^2&\rightarrow&\exp\Big\{-\frac{\binom{2}{2}}{2!d}\frac{(k-1)^2(a-b)^2}{k^{2(2-1)}}\Big\}\exp\Big\{-\frac{\binom{3}{2}}{3!d}\frac{(k-1)^2(a-b)^2}{k^{2(3-1)}}\Big\}\mathbb{E}\exp\Big\{\frac{c_2^{(2)}+c_2^{(3)}}{k^2}\chi^2_{(k-1)^2}\Big\}\\
&=&\exp\Big\{-\frac{\binom{2}{2}}{2!d}\frac{(k-1)^2(a-b)^2}{k^{2(2-1)}}\Big\}\exp\Big\{-\frac{\binom{3}{2}}{3!d}\frac{(k-1)^2(a-b)^2}{k^{2(3-1)}}\Big\}\\
&&\times\exp\Big\{-\frac{(k-1)^2}{2}\log\Big(1-2\frac{c_2^{(2)}+c_2^{(3)}}{k^2}\Big)\Big\}\\
&=&\exp\Big\{\sum_{h=3}^{\infty}\lambda_h(2)\delta_h(2)^2\Big\}\exp\Big\{\sum_{h=2}^{\infty}\lambda_h(3)\delta_h(3)^2\Big\},
\end{eqnarray*}
where we used the fact that
\[\frac{(k-1)^2}{2}\Big(\frac{2c_2^{(m)}}{k^2}\Big)^h\frac{1}{h}=\frac{(k-1)^2}{2h}\Big(\frac{a+(k^{m-1}-1)b}{k^{m-1}(m-2)!}\Big)^h\Big(\frac{(a-b)^2}{(a+(k^{m-1}-1)b)^2}\Big)^h=\lambda_h(m)\delta_h(m)^2.\]
Obviously, $\mathbb{E}_0Y_n=1$. Hence, $H_0$ and $H_1$ are contiguous. 
\end{proof}

\begin{proof}[Proof of Proposition \ref{fixeddegree}]
Under $H_0^{\prime}$ and condition $1\ll ||W||_t^t=O( ||W||_1)=O(n)$ for $2\leq t\leq 12$,  we have
\[T_1=\mathbb{E}\Big[\frac{1}{n^6}\sum_{i_1,\dots,i_6:distinct}A_{i_1i_2i_3}A_{i_3i_4i_5}A_{i_5i_6i_1}\Big]=\frac{||W||_2^6||W||_1^3p_0^3}{n^6}+O\Big(\frac{||W||_1^5p_0^3}{n^6}\Big), \]
\[ E_1= \mathbb{E}\Big[\frac{1}{n^3}\sum_{i_1,i_2,i_3:distinct}A_{i_1i_2i_3}\Big]=\frac{||W||_1^3p_0}{n^3}(1+o(1)),\]
\[V_1= \mathbb{E}\Big[\frac{1}{n^5}\sum_{i_1,\dots, i_5:distinct}A_{i_1i_2i_3}A_{i_3i_4i_5}\Big]=\frac{||W||_2^2||W||_1^4p_0^2}{n^5}(1+o(1)).\]
Hence, $T_1-\Big(\frac{V_1}{E_1}\Big)^3=O\Big(\frac{||W||_1^5p_0^3}{n^6}\Big)$.
If $||W||_1\asymp ||W||_2^2$ and $p_0^2||W||_1^3=o(1)$, we have
\[\mathbb{E}(\widehat{T}_1-T_1)^2=\frac{||W||_2^6||W||_1^3p_0^3}{n^{12}}(1+o(1)).\]
Besides, direct calculation yields
\[\mathbb{E}(\widehat{E}_1-E_1)^2=\frac{||W||_1^3p_0}{n^{6}}(1+o(1)),\]
\[\mathbb{E}(\widehat{V}_1-V_1)^2=\Big(\frac{||W||_2^2||W||_1^4p_0^2}{n^{10}}+\frac{||W||_3^3||W||_1^6p_0^3}{n^{10}}\Big)(1+o(1)).\]
Then by equation (\ref{normal1}), it's easy to verify that
 \begin{eqnarray}\label{fixexpan}
 \widehat{T}_1-\Big(\frac{\widehat{V}_1}{\widehat{E}_1}\Big)^3&=&T_1-\Big(\frac{V_1}{E_1}\Big)^3+(\widehat{T}_1-T_1)+o_P\Big(\sqrt{\frac{||W||_2^6||W||_1^3p_0^3}{n^{12}}}\Big).
 \end{eqnarray}
As a result, under the conditions $1\ll ||W||_t^t=O( ||W||_1)=O(n)$ for $2\leq t\leq 12$,  $||W||_1\asymp ||W||_2^2$ and $p_0^2||W||_1^3=o(1)$, we have
 \begin{eqnarray*}\label{fixnormal}
&&\frac{\sqrt{n^6}\Big(\widehat{T}_1-\Big(\frac{\widehat{V}_1}{\widehat{E}_1}\Big)^3\Big)}{\sqrt{T_1}}
=\frac{\sqrt{n^6}\Big(\widehat{T}_1-T_1\Big)}{\sqrt{T_1}}+o_P(1)\\
&=&\frac{\sum_{i_1,\dots,i_6:distinct}(A_{i_1i_2i_3}-W_{i_1}W_{i_2}W_{i_3}p_0)(A_{i_3i_4i_5}-W_{i_3}W_{i_4}W_{i_5}p_0)(A_{i_5i_6i_1}-W_{i_5}W_{i_6}W_{i_1}p_0)}{\sqrt{n^6T_1}}+o_P(1),
\end{eqnarray*}
which converges in distribution to $N(0,1)$ under $H_0^{\prime}$ by a similar proof as in Theorem \ref{normality}.

Under $H_1^{\prime}$, if we further assume $1\ll a_n\asymp b_n\leq n^{\frac{1}{3}}$, then equation (\ref{fixexpan}) still holds. By a similar proof of Lemma \ref{rate}, we conclude that $\frac{\sqrt{n^6}\Big(\widehat{T}_1-T_1\Big)}{\sqrt{T_1}}=O_p(1)$ and 
 \begin{eqnarray*}\label{fixexpan}
 \frac{\sqrt{n^6}\Big(\widehat{T}_1-\Big(\frac{\widehat{V}_1}{\widehat{E}_1}\Big)^3\Big)}{\sqrt{T_1}}&=&\frac{\sqrt{n^6}\Big(T_1-\Big(\frac{V_1}{E_1}\Big)^3\Big)}{\sqrt{T_1}}+O_P(1)=\delta_1+O_P(1).
 \end{eqnarray*}
Then the power of the test goes to 1 if $\delta_1\rightarrow\infty$.

\end{proof}
 
\end{document}